\documentclass[a4paper]{amsart}

\author{Eske Ewert}
\title{Pseudo-differential extension for graded nilpotent Lie~groups}
\address[Eske Ewert]{Institut für Analysis, Leibniz Universität Hannover, Welfengarten 1,
	30167 Hannover}
\pdfoutput=1
\usepackage[utf8]{inputenc}
\usepackage[T1]{fontenc}
\usepackage{lmodern}
\usepackage{amssymb}
\usepackage{xfrac,mathtools,amsmath,amsxtra}
\usepackage{faktor}
\usepackage[shortlabels]{enumitem}
\usepackage{amsrefs}
\usepackage{microtype}
\usepackage{mathrsfs}
\usepackage{hyperref}
\usepackage[capitalize]{cleveref}
\usepackage[norefs]{refcheck}
\usepackage[top=0.5in,margin=1.2in,includeheadfoot,bottom=1in]{geometry} 

\usepackage{tikz}
\usetikzlibrary{intersections}
\usetikzlibrary{matrix,arrows}
\tikzset{cd/.style=matrix of math nodes,row sep=2em,column sep=2em, text height=1.5ex, text depth=0.5ex}
\tikzset{cdar/.style=->,auto}
\usepackage{tikz-cd}

\crefdefaultlabelformat{#2{\upshape#1}#3}

\tikzset{dar/.style={double,double equal sign distance,-implies}}
\tikzset{mid/.style={anchor=mid}} 

\tikzset{triar/.style={anchor=mid,->}}
\tikzset{tridar/.style={anchor=mid,double,double equal sign distance,-implies}}
\tikzset{narrowfill/.style={inner sep=0pt, fill=white}}

\setlist[enumerate,1]{label=\textup{(\arabic*)}}
\setlist[enumerate,2]{label=\textup{(\alph*)}}

\newtheorem{theorem}{Theorem}[section]
\newtheorem{lemma}[theorem]{Lemma}
\newtheorem{proposition}[theorem]{Proposition}
\newtheorem{corollary}[theorem]{Corollary}

\theoremstyle{definition}
\newtheorem{definition}[theorem]{Definition}

\theoremstyle{remark}
\newtheorem{remark}[theorem]{Remark}
\newtheorem{example}[theorem]{Example}

\newcommand*{\defeq}{\mathrel{\vcentcolon=}}

\DeclarePairedDelimiter{\abs}{\lvert}{\rvert}
\DeclarePairedDelimiter{\norm}{\lVert}{\rVert}
\DeclarePairedDelimiterX{\braket}[2]{\langle}{\rangle}{#1\,\delimsize\vert\,\mathopen{}#2}
\DeclarePairedDelimiterX{\braketop}[3]{\langle}{\rangle}{#1\,\delimsize\vert #2\delimsize\vert\,\mathopen{}#3}
\DeclarePairedDelimiterX{\BRAKET}[2]{\langle}{\rangle}{\!\delimsize\langle#1\,\delimsize\vert\,\mathopen{}#2\delimsize\rangle\!}
\DeclarePairedDelimiter{\BRA}{\langle\!\langle}{\rvert}
\DeclarePairedDelimiter{\KET}{\lvert}{\rangle\!\rangle}
\DeclarePairedDelimiterX{\setgiven}[2]{\{}{\}}{#1\,{:}\,\mathopen{}#2}

\newcommand{\idealin}{\mathrel{\triangleleft}} 
\newcommand{\properideal}{%
	\mathrel{\ooalign{$\lneq$\cr\raise.22ex\hbox{$\lhd$}\cr}}}

\newcommand*{\injto}{\hookrightarrow}

\newcommand*{\acts}{\curvearrowright}

\newcommand*{\Mult}{\mathcal M}
\newcommand{\C}{\mathbb{C}}
\newcommand{\N}{\mathbb{N}}

\newcommand{\Z}{\mathbb{Z}}
\newcommand{\R}{\mathbb{R}}
\newcommand{\T}{\mathbb{T}}
\newcommand{\Uni}{\mathcal{U}}

\newcommand{\Rp}{{\mathbb{\R}_{>0}}}

\newcommand{\lspan}{\mathrm{span}}
\newcommand{\Rock}{R}
\newcommand{\kernel}{\mathcal{K}}
\DeclareMathOperator{\cp}{{cp}}
\newcommand*{\Hilm}[1][F]{\mathcal #1}
\newcommand*{\Hils}[1][H]{\mathcal #1}

\newcommand*{\Cont}{\mathrm C}
\newcommand{\Schwartz}{\mathcal{S}}
\newcommand*{\Fix}{\mathrm{Fix}}
\newcommand*{\Prime}{\mathrm{Prime}}

\newcommand*{\nb}{\nobreakdash}
\newcommand*{\Cst}{\mathrm C^*}
\newcommand*{\Cred}{\mathrm C^*_\mathrm r}
\newcommand*{\Prim}{\mathrm{Prim}}

\newcommand*{\diff}{\,\mathrm{d}}
\newcommand{\Four}{\mathcal F}

\newcommand{\Comp}{\mathbb K}
\newcommand{\Bound}{\mathbb B}

\newcommand{\lie}[1]{\mathfrak{#1}} 

\makeatletter
\newsavebox\myboxA
\newsavebox\myboxB
\newlength\mylenA

\newcommand*\xoverline[2][0.75]{%
	\sbox{\myboxA}{$\m@th#2$}%
	\setbox\myboxB\null
	\ht\myboxB=\ht\myboxA%
	\dp\myboxB=\dp\myboxA%
	\wd\myboxB=#1\wd\myboxA
	\sbox\myboxB{$\m@th\overline{\copy\myboxB}$}
	\setlength\mylenA{\the\wd\myboxA}
	\addtolength\mylenA{-\the\wd\myboxB}%
	\ifdim\wd\myboxB<\wd\myboxA%
	\rlap{\hskip 0.5\mylenA\usebox\myboxB}{\usebox\myboxA}%
	\else
	\hskip -0.5\mylenA\rlap{\usebox\myboxA}{\hskip 0.5\mylenA\usebox\myboxB}%
	\fi}
\makeatother

\newcommand*{\conj}[1]{\overline{#1}}
\newcommand*{\cl}[1]{\xoverline{#1}}

\newcommand*{\Id}{\mathrm{id}}
\newcommand*{\K}{\mathrm{K}}
\newcommand*{\KK}{\mathrm{KK}}
\newcommand*{\ev}{\mathrm{ev}}

\DeclareMathOperator{\supp}{supp}
\DeclareMathOperator{\triv}{triv}
\DeclareMathOperator{\Aut}{Aut}
\DeclareMathOperator{\Ad}{Ad}

\DeclareMathOperator{\tr}{tr}
\DeclareMathOperator{\Ind}{Ind} 
\DeclareMathOperator{\coad}{co-ad}
\DeclareMathOperator{\coAd}{co-Ad}
\DeclareMathOperator{\Span}{span}

\DeclareMathOperator{\class}{cl}

\DeclareMathOperator{\princg}{princ}
\DeclareMathOperator{\Op}{Op}
\DeclareMathOperator{\VN}{VN}

\newcommand{\Rel}{\mathcal{R}} 
\newcommand{\si}{\mathrm{si}} 

\newcommand{\lieg}{\mathfrak{g}} 
\newcommand{\lieh}{\mathfrak{h}}

\newcommand{\grpd}{\mathcal{G}}

\makeatletter
\newcommand{\refcheckize}[1]{%
	\expandafter\let\csname @@\string#1\endcsname#1%
	\expandafter\DeclareRobustCommand\csname relax\string#1\endcsname[1]{%
		\csname @@\string#1\endcsname{##1}\wrtusdrf{##1}}%
	\expandafter\let\expandafter#1\csname relax\string#1\endcsname
}
\makeatother

\refcheckize{\cref}
\refcheckize{\Cref}

\begin{document}

\begin{abstract}
	Classical pseudo-differential operators of order zero on a graded nilpotent Lie group~\(G\) form a \(^*\)-subalgebra of the bounded operators on \(L^2(G)\). We show that its \(\Cst\)-closure is an extension of a noncommutative algebra of principal symbols by compact operators. As a new approach, we use the generalized fixed point algebra of an \(\Rp\)\nb-action on a certain ideal in the \(\Cst\)-algebra of the tangent groupoid of \(G\). The action takes the graded structure of \(G\) into account. Our construction allows to compute the \(\K\)-theory of the algebra of symbols.
\end{abstract} 

\maketitle
 
\makeatletter
\providecommand\@dotsep{5}
\makeatother

\section{Introduction}
A homogeneous Lie group is a nilpotent Lie group \(G\) with a dilation action of~\(\Rp\) by group automorphisms. The dilation action allows to scale with different speed in different tangent directions. A slightly less general class are graded nilpotent Lie groups. A prominent example is the Heisenberg group whose Lie algebra is generated by \(\{X,Y,Z\}\) with \([X,Y]=Z\) and \([X,Z]=[Y,Z]=0\). Then \(A_\lambda(X)=\lambda X\), \(A_\lambda(Y)=\lambda Y\) and \(A_\lambda(Z)=\lambda^2Z\) for \(\lambda>0\) define dilations on the Heisenberg algebra. The dilations induce a new notion of order and homogeneity for differential operators on \(G\). For example, in the case of the Heisenberg group, one would assign order \(2\) to \(Z\) and order \(1\) to \(X\) and \(Y\). 

Certain hypoelliptic operators, like Hörmander's sum of squares or Kohn's Laplacian \(\square_b\), can be analysed using homogeneous convolution operators on homogeneous Lie groups \cite{folland1977applications}. Therefore, it is desirable to have a pseudo-differential calculus that takes the homogeneous structure into account. 
In the 80s, a kernel-based pseudo-differential calculus for homogeneous Lie groups was developed in \cite{christ1992pseudo}. Recently, Fischer and Ruzhansky introduced in \cite{fischer2016quantization} a symbolic calculus for graded nilpotent Lie groups. Instead of functions on the cotangent bundle as in the Euclidean case, the symbols are given here by fields of operators using operator valued Fourier transform. This uses that the representation theory of graded nilpotent Lie groups is well-known and the abstract Plancherel Theorem \cite{dixmier} applies. In~\cite{fischer2017defect} homogeneous expansions, classical pseudo-differential operators and their principal symbols were defined for this calculus. Graded nilpotent Lie groups are also instances of filtered manifolds, for which a pseudo-differential calculus was developed in \cite{erp2015groupoid}.

This article describes a different approach to pseudo-differential operators on homogeneous Lie groups using generalized fixed point algebras. Generalized fixed point algebras were introduced by Rieffel \cites{rieffel1998,rieffel1988} to generalize proper group actions on spaces to the noncommutative setting. If a locally compact group~\(H\) acts properly on a locally compact Hausdorff space \(X\), the orbit space~\(H\backslash X\) is again locally compact. The generalized fixed point algebra is in this case \(\Cont_0(H\backslash X)\), which can be viewed as a subalgebra of the \(H\)\nb-invariant multipliers of \(A=\Cont_0(X)\). Moreover, \(\Rel=\Cont_c(X)\) can be completed into an imprimitivity bimodule between an ideal in the reduced crossed product \(\Cred(H,\Cont_0(X))\) and the generalized fixed point algebra. 
In \cite{meyer2001} it is investigated for which group actions \(\alpha\colon H \curvearrowright A\) on a \(\Cst\)\nb-algebra \(A\), one can build a generalized fixed point algebra which is Morita--Rieffel equivalent to an ideal in \(\Cred(H,A)\). The crucial step is to find a dense subset \(\Rel\subset A\) that is continuously square-integrable. As it turns out, such \(\Rel\) can fail to exist or to be unique. If \(\Rel\) satisfies the requirements, the generalized fixed point algebra \(\Fix^H(A,\Rel)\) is generated by averages \(\int_H\alpha_x(a^*b)\diff x\) for \(a,b\in\Rel\), understood as \(H\)-invariant multipliers of \(A\). 

A classical pseudo-differential operator of order \(k\) on a manifold \(M\) is determined up to operators of lower order by its principal symbol. The principal symbol is a \(k\)\nb-homogeneous function on \(T^*M\!\setminus\!( M\times\{0\})\). Hence, for \(k=0\) the principal symbol is a generalized fixed point of the scaling action of \(H=\Rp\) on \(T^*M\!\setminus\! (M\times\{0\})\) in the cotangent direction. Therefore, the \(\Cst\)\nb-closure of the~\(0\)\nb-homogeneous symbols \(\Cont_0(S^*M)\) is a generalized fixed point algebra.
As it turns out, not only the principal symbols, but also the pseudo-differential operators of order zero themselves are generalized fixed points. A special case of the results in \cite{debordskandalis2014} is that the classical pseudo-differential calculus for a manifold \(M\) can be recovered from Connes' tangent groupoid \cite{connes}. Moreover, they observed that each pseudo-differential operator of order zero can be written as an average \(\int_0^ \infty f_t\tfrac{\diff t}{t}\), where \((f_t)_{t\in[0,\infty)}\) is an element of the \(\Cst\)-algebra of the tangent groupoid of \(M\) satisfying certain conditions. Elements of a generalized fixed point algebra are obtained in exactly this fashion.

It was shown in \cite{miller2017} that the \(\Cst\)-closure of classical pseudo-differential operators of order zero on \(\R^n\) inside the bounded operators on \(L^2(\R^n)\) is a generalized fixed point algebra. In fact, it is the generalized fixed point algebra of a zoom action of \(\Rp\) on an ideal in the \(\Cst\)-algebra of the tangent groupoid. 
In this article, we generalize this result to graded nilpotent Lie groups \(G\). We describe a variant of Connes' tangent groupoid
\[\grpd= (TG\times \{0\}\cup (G\times G) \times (0,\infty)\rightrightarrows G\times [0,\infty)),\]
where the operation on the tangent bundle \(TG\) is given by group multiplication in the fibres. This is a special case of the tangent groupoid of a filtered manifold which was considered before in \cites{erp2017tangent, choi2019tangent, haj-higson}. It is equipped with a zoom action of \(\Rp\), which is induced by the dilations on \(G\).

Let \(J\) be the ideal in \(\Cst(\grpd)\) that consists of all elements whose restriction to~\((x,0)\), which is an element of \(\Cst(T_xG)\cong\Cst(G)\), lies in the kernel of the trivial representation of \(G\) for all \(x\in G\). In the commutative case, this corresponds under Fourier transform to taking out the zero section in~\(T^*\R^n\), which is necessary to obtain a proper action. We show that there is a subset \(\Rel\subset J\) such that \(\Fix^\Rp(J,\Rel)\) is defined. Let \(J_0\) and \(\Rel_0\) be the restriction of \(J\) and \(\Rel\) to \(t=0\), respectively.
Then we obtain another generalized fixed point algebra \(\Fix^\Rp(J_0,\Rel_0)\). These fit into an extension
\begin{equation}\label{eq:fix_pdo_ext}
\begin{tikzcd}
\Comp(L^2 G)\arrow[hook,r] & \Fix^\Rp(J,\Rel)\arrow[r,twoheadrightarrow] &\Fix^\Rp(J_0,\Rel_0).
\end{tikzcd}
\end{equation}

We call this the pseudo-differential extension of \(G\). We justify the name by showing that this extension is the \(\Cst\)-completion of the order zero pseudo-differential extension defined in \cite{fischer2017defect}
\begin{equation*}
\begin{tikzcd}
\Psi^{-1}_{\class}\arrow[hook,r] & \Psi^0_{\class}\arrow[r,twoheadrightarrow,"\princg"] &\dot{S}_c^0.
\end{tikzcd}
\end{equation*}
The \(\Cst\)-algebra \(\Fix^\Rp(J_0,\Rel_0)\) of principal symbols is, in general, noncommutative. However, as it is a generalized fixed point algebra, it is Morita--Rieffel equivalent to an ideal in \(\Cred(\Rp,J_0)\). Using the representation theory of nilpotent Lie groups and, in particular, Kirillov theory \cite{kirillov1962} and Pukanszky's stratification \cite{pukanszky}, we show that it is actually Morita--Rieffel equivalent to the whole crossed product. Furthermore, \(\Fix^\Rp(J,\Rel)\) is Morita equivalent to \(\Cred(\Rp,J)\), which was observed before in \cite{debordskandalis2014} for the case of dilations given by scalar multiplication.

The Morita equivalence allows us to prove that \(\Fix^\Rp(J_0,\Rel_0)\) is \(\KK\)-equivalent to \(\Cont_0(S^*\R^n)\). Hence, although the symbols in the homogeneous and Euclidean case differ, the resulting \(\Cst\)\nb-algebras have the same \(\K\)-theory. Moreover, our approach can be used to recover the computation of the spectrum of \(\Cst(\dot{S}_c^0)\) in \cite{fischer2017defect}.

The article is organized as follows. \cref{fixedpointalg} introduces generalized fixed point algebras and examines their behaviour for extensions of \(\Cst\)-algebras. \cref{homogeneous} compiles some facts about analysis on homogeneous Lie groups and their representation theory. 
The tangent groupoid \(\grpd\) of a homogeneous Lie group and its \(\Cst\)-algebra are defined in \cref{sec:tangent}. In \cref{sec:pdo_using_fix} we build the pseudo-differential extension \eqref{eq:fix_pdo_ext} using generalized fixed point algebras. \cref{sec:fix_type_0} relates \(\Fix^\Rp(J_0,\Rel_0)\) with operators of type \(0\) on \(G\). In \cref{sec:symbols} we compare the generalized fixed point algebra extension to the calculus of Fischer--Ruzhansky--Fermanian-Kammerer.
In \cref{fullness} we show the mentioned Morita--Rieffel equivalences. Moreover, we compute the spectrum and \(\K\)-theory of \(\Fix^\Rp(J_0,\Rel_0)\). The results in this article are also contained in the author's PhD~thesis~\cite{ewert2020index}.
\subsection*{Acknowledgements} The author thanks her PhD advisors Ralf Meyer and Ryszard Nest for their suggestions and advice. Furthermore, she would like to thank V\'eronique Fischer for discussions during her visit in Göttingen. This research was supported by the RTG 2491 ``Fourier Analysis and Spectral Theory''. 

\section{Generalized fixed point algebras and extensions}\label{fixedpointalg}

Rieffel proposes a notion for proper group action on \(\Cst\)\nb-algebras in \cites{rieffel1998,rieffel1988}, which generalizes proper actions on locally compact Hausdorff spaces. This leads to the construction of generalized fixed point algebras. We follow the approach taken in \cite{meyer2001}. In this section, we recall the notions used there and prove some results regarding the behaviour of generalized fixed point algebras under extensions of \(\Cst\)-algebras, which will be needed in the later chapters.
 \subsection{The construction}
For this section, let \(H\) be a locally compact group and \(A\) a \(\Cst\)\nb -algebra with a strongly continuous action \(\alpha \colon H \to \Aut(A)\). 

If \(H\) acts properly on a locally compact Hausdorff space \(X\), the generalized fixed point algebra is given by \(\Cont_0(H\backslash X)\), where \(H\backslash X\) denotes the orbit space. It is Morita--Rieffel equivalent to an ideal in the reduced crossed product \(\Cred(H,\Cont_0(X))\). A feature of the generalized fixed point algebra construction is that this property carries over to noncommutative \(A\): the generalized fixed point algebra is Morita--Rieffel equivalent to an ideal in \(\Cred(H,A)\). The following definitions and results
are taken from \cite{meyer2001}. 

We recall first the definition of the crossed product \(\Cred(H,A)\).
There are covariant representation \((\rho^A,\rho^H)\) of the \(\Cst\)-dynamical system \((A,H,\alpha)\) on the right Hilbert \(A\)-module \(L^2(H,A)\) defined by
	\begin{align*}
	(\rho^A_a\psi)(x)&=\alpha_x(a)\psi(x)  &&\text{for }a\in A\text{, }x\in H,\\
	(\rho^H_y\psi)(x)&=\psi(xy)		&&\text{for }x,y\in H,
	\end{align*}
for \(\psi\in\Cont_c(H,A)\). Equip \(\Cont_c(H,A)\) with the following convolution and involution
	\begin{align}
	(f*g)(x)&=\int_H f(y)\alpha_y(g(y^{-1}x))\diff y,\label{conv}\\
	f^*(x)&= \alpha_x(f(x^{-1}))^*\label{inv}
	\end{align}
for \(x\in H\). The Haar measure on \(H\) is used to define the convolution. The \(I\)\nb-norm is defined by
	\[ \norm{f}_I = \max \left\{ \int_H\norm{f(x)}\diff x, \int_H\norm{f^*(x)}\diff x\right\}.\]
The representation \((\rho^A,\rho^H)\) integrates to the \(^*\)\nb-representation \(\rho\) of \(\Cont_c(H,A)\) with
	\begin{align}\label{crossedprod} (\rho_f\psi)(x)=\int_H \alpha_x(f(x^{-1}y))\psi(y)\diff y\qquad \text{for }f,			 	\psi\in\Cont_c(H,A),\end{align}
which satisfies \(\norm{\rho_f}\leq \norm{f}_I\) for all \(f\in\Cont_c(H,A)\). The \emph{reduced crossed product} \(\Cred(H,A)\) is the norm closure of \(\rho(\Cont_c(H,A))\) inside \(\Bound(L^2(H,A))\).

\begin{lemma}\label{approxunit}
	The representation \(\rho^A\) maps to the multiplier algebra of \(\Cred(H,A)\). If \(				(u_\lambda)\) is an approximate identity for \(A\), then \(\norm{F-\rho^A_{u_\lambda}	
	\circ 	F}\to 0\) for each \(F\in\Cred(H,A)\).
\end{lemma}

\begin{proof}The first claim follows from \(\rho^A_a\circ\rho_f=\rho_{a f}\) for all \(a\in A\) and \(f\in\Cont_c(H,A)\). For the second claim note that
		\[\norm{\rho_f-\rho^A_{u_\lambda}\circ\rho_f}=\norm{\rho_{f-u_\lambda f}}\leq \norm{f-u_\lambda f}_I,\]	
	which converges to zero for compactly supported \(f\). 
	As \(\Cont_c(H,A)\) is dense, the same holds for arbitrary elements of \(\Cred(H,A)\) by continuity.
\end{proof}

 The diagonal action of \(H\) on \(\Cont_b(H,A)\) or \(\Cont_c(H,A)\) is given by \((h\cdot f)(x)=\alpha_h(f(h^{-1}x))\). For \(a\in A\) the operators
\begin{align}\label{BRAKET}
	\BRA{a}\colon& A \to \Cont_b(H,A), &\left(\BRA{a}b\right)(x)&\defeq\alpha_x(a)^*b,\\
	\KET{a}\colon& \Cont_c(H,A) \to A, &\KET{a}f&\defeq \int_H \alpha_x(a)f(x)\diff x.
\end{align}
are \(H\)\nb-equivariant and adjoint to each other with respect to the pairings \(\braket{a}{b}=a^*b\) for \(a,b\in A\) and \(\braket{f}{g}=\int_H f(x)^*g(x)\diff x\) for \(f\in\Cont_b(H,A)\) and \(g\in\Cont_c(H,A)\).

Let \(\chi_i\colon H\to [0,1]\), \(i\in I\), be a net of continuous, compactly supported functions with \(\chi_i\to 1\) uniformly on compact subsets. A function \(f\in \Cont_b(H,A)\) is called \emph{square-integrable} if and only if \((\chi_if)\) converges in \(L^2(H,A)\). It is shown in \cite{miller2017}*{1.13,~1.15} that the convergence of \((\chi_if)\) and its limit do not depend on the chosen net. 
\begin{definition}
	An element \(a\in A\) is called \emph{square-integrable} if \(\BRA{a}b\in\Cont_b(H,A)\) 			is square-integrable for all \(b\in A\).
\end{definition}
In this case, we understand \(\BRA{a}\) as an operator \(A\to L^2(H,A)\). By \cite{meyer2001}, \(a\in A\) is square-integrable if and only if \(\KET{a}\) extends to an adjointable operator \(L^2(H,A)\to A\). We also denote it by \(\KET{a}\). Its adjoint is \(\BRA{a}\). Let \(A_\si\) be the vector space of all square-integrable elements in \(A\). It becomes a Banach space with respect to the norm \	\			\[\norm{a}_\si\defeq \norm{a}+\norm{\BRA{a}\circ\KET{a}}^{1/2}=\norm{a}+\norm{\KET{a}}.\]

\begin{definition}
	A subset \(\Rel\subset A_\si\) is called \emph{relatively continuous} if for all 
	\(a,b\in\Rel\) the operator \(\BRAKET{a}{b}\defeq\BRA{a}\circ\KET{b}\in\Bound(L^2(H,A))\) 			is contained in the reduced crossed product \(\Cred(H,A)\subset \Bound(L^2(H,A))\). 
	It is called \emph{complete} if \(\Rel\) is a closed linear subspace of \(A_\si\) with 
	respect to \(\norm{\,\cdot\,}_\si\) and satisfies \(\KET{a}(\Cont_c(H,A))\subset \Rel\) for 
	all \(a\in\Rel\).
\end{definition}
\begin{definition}
	A \emph{continuously square-integrable \(H\)\nb-\(\Cst\)\nb-algebra} \((A,\Rel)\) consists of a 
	\(\Cst\)\nb-algebra~\(A\) with a strongly continuous action of \(H\) and a subset 	
	\(\Rel\subset A\) which is relatively continuous, complete and dense in \(A\).
\end{definition}

If \(H\) acts properly on a locally compact Hausdorff space \(X\), \((\Cont_0(X),\cl{\Cont_c(X)}^\si)\) is a continuously square-integrable \(H\)\nb-\(\Cst\)\nb-algebra. Here, \(\Cont_c(X)\) is completed with respect to the \(\norm{\,\cdot\,}_\si\)-norm above. For an arbitrary \(\Cst\)\nb-algebra A, a subset \(\Rel\subset A\) satisfying the requirements above can fail to exist or to be unique as shown in \cite{meyer2001}. 
However, there is a sufficient condition that guarantees the existence of a unique such \(\Rel\).
Let the primitive ideal space of \(A\) be equipped with the Jacobson topology. There is a continuous \(H\)-action on \(\Prim(A)\) defined by \(x\cdot P=\alpha_x(P)\) for \(x\in H\) and \(P\in\Prim(A)\). The \(H\)-\(\Cst\)-algebra \(A\) is called \emph{spectrally proper}, if the action on the primitive ideal space is proper. 
\begin{theorem}[\cite{meyer2001}*{9.4}]\label{res:spectrallyproper}
	Let \(A\) be spectrally proper \(H\)\nb-\(\Cst\)\nb-algebra. Then there is a unique relatively continuous, complete and dense subset.
\end{theorem}

\begin{definition}\label{deffix}
	Let \((A,\Rel)\) be a continuously square-integrable \(H\)\nb-\(\Cst\)\nb-algebra. Let \(
	\Hilm^H(A,\Rel)\) be the closure of \(\KET{\Rel}\subset\Bound(L^2(H,A),A)\). The \emph{generalized fixed point algebra} \(\Fix^H(A,
	\Rel)\) is defined as the closed linear span of \(\KET{\Rel}\BRA{\Rel}\) in the \(H\)-invariant 
	multiplier algebra \(\Mult^H(A)\). 
\end{definition}

Since \(\Rel\) is complete, there is a right \(\Cont_c(H,A)\)-module structure on \(\Rel\) defined by \(a*f=\KET{a}(\breve{f})\) for \(a\in\Rel\) and \(f\in\Cont_c(H,A)\), where \(\breve{\,}\,\colon \Cont_c(H,A)\to\Cont_c(H,A)\) is given by \(\breve{f}(h)\defeq\alpha_h(f(h^{-1}))\) for \(h\in H\).  Because of the identity \(\KET{a}\circ\rho_f=\KET{a*f}\) for \(a\in\Rel\) and \(f\in\Cont_c(H,A)\), this can be extended continuously to a right Hilbert \(\Cred(H,A)\)\nb-module structure on \(\Hilm^H(A,\Rel)\).  

For \(a,b,c,d\in\Rel\) the operator \(\BRAKET{b}{c}\in\Cred(H,A)\) can be approximated by a sequence \((\rho_{f_n})\) with \(f_n\in\Cont_c(H,A)\). Therefore, the product
\[\left(\KET{a}\BRA{b}\right)\left(\KET{c}\BRA{d}\right)=\lim_{n\to\infty} \KET{a}\circ\rho_{f_n}\circ\BRA{d}=\lim_{n\to\infty}\KET{a*f_n}\BRA{d}\]
lies again in the generalized fixed point algebra. As \(\left(\KET{a}\BRA{b}\right)^*=\KET{b}\BRA{a}\), this shows that \(\Fix^H(A,\Rel)\) is a \(\Cst\)-subalgebra of \(\Mult^H(A)\).

Now, we describe the elements of \(\Fix^H(A,\Rel)\) more explicitly. In the commutative case \(H\acts X\), functions on the orbit space can be obtained by averaging functions in \(\Cont_c(X)\) over the action:
\begin{example}
	For \(H\curvearrowright X\) proper and \(f\in \Cont_c(X)\) there is a function \(F\in\Cont_0(H\backslash X)\) defined by
	\[F(Hx)\defeq\int_H f(h^{-1}\cdot x)\diff h\quad \text{for }Hx\in H\backslash X.\]		
\end{example}
The following lemma suggests to also think of elements of \(\Fix^H(A,\Rel)\) for a noncommutative \(A\) as averages over the group action of certain elements of~\(A\). 
\begin{lemma}[\cite{meyer2001}*{(19)}]\label{res:strictlimit}
	Let \((\chi_i)_{i\in I}\) be a net of continuous, compactly supported functions on \(H\) that converges uniformly to \(1\) on compact subsets as above. Let \(a,b\in\Rel\). The net
	\begin{align*}
	\int_H \chi_i(x)\alpha_x(a^*b)\diff x
	\end{align*}
	converges to \(\KET{a}\BRA{b}\) with respect to the strict topology as multipliers of \(A\).
\end{lemma}

Returning to the construction, \(\Hilm^H(A,\Rel)\) is a full left Hilbert \(\Fix^H(A,\Rel)\)-module. 
Let \(J^H(A,\Rel)\) denote the closed linear span of \(\BRAKET{\Rel}{\Rel}\subset\Cred(H,A)\), which is an ideal. 
Then \(\Hilm^H(A,\Rel)\) is a \(\Fix^H(A,\Rel)\)-\(J^H(A,\Rel)\) imprimitivity bimodule. 

The ideal \(J^H(A,\Rel)\) need not be the whole reduced crossed product. The following definition is due to Rieffel \cite{rieffel1988}. 

\begin{definition}
	Let \((A,\Rel)\) be a continuously square-integrable \(H\)\nb-\(\Cst\)\nb-algebra. Call \((A,\Rel)\) \emph{saturated} if \(J^H(A,\Rel)=\Cred(H,A)\).
\end{definition}

\begin{example}
	For a proper action \(H\acts X\), Rieffel observed in \cite{rieffelapplications} that \((\Cont_0(X),\cl{\Cont_c(X)})\) is saturated if the action of \(H\) on \(X\) is free. We will argue in  \cref{res:freesaturated} that the converse is true as well. 
\end{example}
The next lemma, proved in \cite{miller2017}, gives a criterion when a set \(\Rel\subset A_\si\) can be completed to a relatively continuous, complete and dense subset of \(A\).
\begin{lemma}\label{res:completion}
	Let \(\Rel\subset A\) be a dense, linear subspace. Suppose \(\Rel\) consists of square-integrable elements, is relatively 	continuous and \(H\)\nb-invariant,  and satisfies \(\Rel\cdot\Rel\subset\Rel\). Denote by \(\cl{\Rel}\) the closure of \(\Rel\subset A_\si\) with respect to the \(\norm{\,\cdot\,}_{\si}\)\nb-norm.  Then \((A,		 	\cl{\Rel})\) is a continuously square-integrable \(H\)\nb-\(\Cst\)\nb-algebra.  The generalized fixed point algebra \(\Fix^H(A,\cl{\Rel})\) is the closed 				linear span of \(\KET{\Rel}\BRA{\Rel}\).
\end{lemma}
\begin{proof}
	The inclusion \(A_\si\injto A\) is continuous. Since \(\Rel\) is dense in \(A\), also \(\cl{\Rel}\) is a dense subspace of \(A\). As \(\norm{\BRA{a}}=\norm{\KET{a}} \leq \norm{a}_\si\) for all \(a\in A_\si\), elements of \(\BRAKET{\cl{\Rel}}{\cl{\Rel}}\) can be approximated with respect to the operator norm on \(L^2(H,A)\) by elements of \(\BRAKET{\Rel}{\Rel}\). This shows that \(\cl{\Rel}\) is relatively continuous as well. 
	
	It remains to verify that \(\cl{\Rel}\) is complete. First, we show that \(\cl{\Rel}\cdot A\subset \cl{\Rel}\) holds. Let \(r\in\cl{\Rel}\) and~\(a\in A\) and choose sequences \((r_n), (a_n)\) in \(\Rel\) such that \(\norm{r-r_n}_\si\to 0\) and \(\norm{a-a_n}\to 0\). Note that \(ra\in A_\si\) because \(\KET{ra}=\KET{r}\circ\rho^A_a\) and \(r\) is square-integrable. By assumption \(r_na_n\in\Rel\) holds for all~\(n\in\N\). We estimate using \cite{meyer2001}*{(17)} that
	\[\norm{ra-r_na_n}_\si\leq\norm{r}_\si\norm{a_n-a}+\norm{r-r_n}_\si\norm{a_n},\]
	which converges to zero. Furthermore, \(\cl{\Rel}\) is also \(H\)\nb-invariant, which follows from the invariance of~\(\Rel\) and \cite{meyer2001}*{(18)}. This implies that \(\KET{\cl{\Rel}}(\Cont_c(H,A))\subset\cl{\Rel}\). 
	
	Using similar arguments as for the relative continuity of \(\cl{\Rel}\), one obtains that any \(\KET{a}\BRA{b}\) with \(a,b\in\cl{\Rel}\) is a norm limit of elements of \(\KET{\Rel}\BRA{\Rel}\). 
\end{proof}
\begin{remark}
	Suppose \(\Rel\subset A\) is a dense, \(H\)-invariant \(^*\)-subalgebra such that \(\BRA{a}b\) is bounded with respect to the \(I\)-norm for all \(a,b\in \Rel\) as required in the original definition in \cite{rieffel1988}. Then by \cite{meyer2001}*{6.8} \(\Rel\) is relatively continuous and square-integrable. Therefore, \cref{res:completion} shows that \((A,\cl{\Rel})\) is a continuously square-integrable \(H\)\nb-\(\Cst\)\nb-algebra.
\end{remark}
\subsection{Extensions of generalized fixed point algebras}
Let \(I\) be an \(H\)\nb -invariant ideal in \(A\) such that the sequence 
\begin{equation}
\label{exact}
\begin{tikzcd}
\Cred(H,I)\arrow[r,hook] & \Cred(H,A)\arrow[r,twoheadrightarrow] & \Cred(H, A/I)\end{tikzcd}
\end{equation}
is exact.
If \(H\) is an exact group, this is true for all \(H\)\nb-invariant ideals \(I\idealin A\). For example, this holds in our applications in the later chapters where \(H = (\Rp,\,\cdot\,)\cong (\R,+)\).

Let \(\Rel\subset A\) be a subset such that \((A,\Rel)\) is a continuously square-integrable \(H\)\nb-\(\Cst\)\nb-algebra. Consider \(\Rel\cap I\subset I\) and the image of \(\Rel\) under the projection \(q\colon A \to A/I\). We will show that the generalized fixed point algebra construction can be applied to \((I,\Rel\cap I)\) and \((A/I, \cl{q(\Rel)})\), and relate the respective generalized fixed point algebras to each other. 

In particular, we are interested in what can be said about saturatedness in this case. This is inspired by the simple observation that if an \(H\)-space \(X\) can be partitioned into two \(H\)-invariant subsets \(X=X_1\sqcup X_2\), then the action on \(X\) is free if and only if it is free on \(X_1\) and \(X_2\).

\begin{lemma}[\cite{miller2017}]\label{multI}
	Let \(\Rel \subseteq A\) be a relatively continuous, complete subspace of \(A\). If 				\(I\idealin A\) is an \(H\)\nb-invariant ideal such that \eqref{exact} is exact,  then
	\( \Rel \cap I = 	\Rel \cdot I\) holds.
\end{lemma}

\begin{proof}
	Because \(I\) is an ideal in \(A\) and \(\Rel\cdot A=\Rel\) by \cite{meyer2001}*{Cor.~6.7}, 	\(\Rel \cdot I \subseteq \Rel \cap I\) follows. The other inclusion uses exactness in 			\eqref{exact}. Let \(r \in \Rel \cap I\). As
	\[ \BRAKET{r}{r} (L^2(H,A))\subseteq L^2(H,I)\]
	and \eqref{exact} is exact, we have \(\BRAKET{r}{r} \in \Cred(H,I)\). Now, let 							\((u_\lambda)_{\lambda\in \Lambda}\) be an approximate unit for \(I\), satisfying 					\(u_\lambda^*=u_\lambda\) and \( \norm{u_\lambda}\leq 1\) for all \(\lambda\in \Lambda\). 			One computes
	\begin{align*}
	\norm{\KET{r}-\KET{ru_\lambda}}^2 &= \norm{\BRAKET{r-ru_\lambda}{r-ru_\lambda}}\\
	&\leq \norm{\BRAKET{r}{r}-\rho^I_{u_\lambda^*}\circ\BRAKET{r}{r}}+\norm{\BRAKET{r}{r}\circ\rho^I_{u_\lambda}-\rho^I_{u_\lambda^*}\circ\BRAKET{r}{r}\circ\rho^I_{u_\lambda}}\\&\leq 2 			\cdot \norm{\BRAKET{r}{r}-\rho^I_{u_\lambda}\circ \BRAKET{r}{r}}.
	\end{align*}
	By \cref{approxunit} this converges to zero. Furthermore, 
	\(\norm{r-ru_\lambda}\to 0\) holds. Hence, \(r\in \Rel\cdot I\) follows from Cohen's 				Factorization Theorem applied to \((\Rel,\norm{\,\cdot\,}_\si)\) as a right 
	\(I\)\nb-module.
\end{proof}

\begin{lemma}\label{res:rel_subsets}
	Let \((A, \Rel)\) be a continuously square-integrable \(H\)\nb-\(\Cst\)\nb-algebra and let
	\(I\idealin A\) be an \(H\)-invariant ideal such that the sequence in~\eqref{exact} is exact. Let \(q\colon A\to A/I\) be the quotient map. Then the following holds:
	\begin{enumerate}
		\item \label{ideal} \((I,\Rel\cap I)\) is a continuously square-integrable 
		\(H\)\nb-\(\Cst\)\nb-algebra.
		\item \label{quotient} \((A/I,\cl{q(\Rel)})\) is a continuously square-integrable 
		\(H\)\nb-\(\Cst\)\nb-algebra. Here, \(\cl{q(\Rel)}\) denotes the closure of \(q(\Rel)\subset (A/I)_{\si}\) in the \(\norm{\,\cdot\,}_\si\)-norm. 
	\end{enumerate}
\end{lemma}
\begin{proof}
	We prove \ref{ideal}. The linear subspace \(\Rel\cap I=\Rel \cdot I\) is dense in~\(I\) because any element \(i \in I\) can be factorized as \(i = a\cdot j\) for some \(a\in A\) and \(j\in I\). Since \(\Rel\) is dense in \(A\), there is a net 
	\((r_\lambda)_{\lambda\in \Lambda}\subset \Rel\) with \(r_\lambda \to a\) and hence 
	\( i = \lim_\lambda r_\lambda \cdot j\).	
	The square-integrability of elements in \(\Rel\cap I\) is inherited from \(\Rel\),					and
	\(\KET{\Rel \cap I}(\Cont_c(H,I))\subseteq \Rel\cap I \)
	holds. Then
	\(\BRAKET{\Rel\cap I}{\Rel\cap I}\subset \Cred(H,I)\) follows from the same argument as in the proof of \cref{multI} using that~\eqref{exact} is exact.		
	Note that 
	\(\norm{\BRAKET{i}{i}}_{\Cred(H,I)}=\norm{\BRAKET{i}{i}}_{\Cred(H,A)}\)
	for \(i\in \Rel\cap I\). 
	Because \(I\idealin A\) is closed and \(\Rel\) is closed with respect to 
	\(\norm{\,\cdot\,}_{\si,A}\), this means that \(\Rel\cap I\) is closed with respect to 				\(\norm{\,\cdot\,}_{\si,I}\). Hence, \((I,\Rel\cap I)\) is a continuously square-integrable 
	\(H\)\nb-\(\Cst\)\nb-algebra.
	
	To prove \ref{quotient} we show that \cref{res:completion} can be applied to \(q(\Rel)\subset A/I\). As \(\Rel\subset A\) is a dense linear subspace, the same holds for \(q(\Rel)\subset A/I \). For \(a\in\Rel\) and \(i\in I\) their product \(ai\in \Rel\cdot I = \Rel\cap I\) lies in \(\Rel\). All elements \(q(a)\) for \(a\in \Rel\) are square-integrable because the quotient map \(L^2(H,A)\to L^2(H,A/I)\) is continuous. 
	Let \(Q\colon \Bound(L^2(H,A))\to \Bound(L^2(H,A/I))\) be the canonical map. We have 
	\begin{equation*}
	\BRAKET{q(a)}{q(b)}=Q(\BRAKET{a}{b}) \qquad \text{ for }a,b\in\Rel.
	\end{equation*}
	The relative continuity of \(q(\Rel)\) follows as \(Q\) maps \(\Cred(H,A)\) to \(\Cred(H,A/I)\).	
	By \cite{meyer2001}*{6.7}, \(\Rel\) is \(H\)\nb-invariant and an essential right \(A\)\nb-module, that is, \(\Rel\cdot A= \Rel\). This implies that \(q(\Rel)\) is also \(H\)\nb-invariant and satisfies \(q(\Rel)\cdot q(\Rel)\subset q(\Rel)\). Therefore, the claim follows from \cref{res:completion}.
\end{proof}

\begin{remark}\label{rem:contsi}
	The restricted map \(q\colon A_{\si}\to(A/I)_\si\) is continuous with respect to the respective \(\norm{\,\cdot\,}_\si\)-norms as for \(a\in A_\si\)
	\[\norm{q(a)}+\norm{\BRAKET{q(a)}{q(a)}}^{1/2}=\norm{q(a)}+\norm{Q(\BRAKET{a}{a})}^{1/2}\leq\norm{a}+\norm{\BRAKET{a}{a}}.\]
	In particular, for \(\Rel\subset A_{\si}\) one has \(\cl{q\left(\cl{\Rel}\right)}=\cl{q(\Rel)}\) with respect to the \(\norm{\,\cdot\,}_\si\)-norms.
\end{remark}
In the situation of \cref{res:rel_subsets}, \(\Hilm^H(I,\Rel\cap I)\) is a closed \(\Fix^H(A,\Rel)\)-\(J^H(A,\Rel)\) submodule of \(\Hilm^H(A,\Rel)\). Under the Rieffel correspondence (see for example~\cite{raeburnwilliams}*{3.22}), \(\Hilm^H(I,\Rel\cap I)\) corresponds to the ideals \(\Fix^H(I,\Rel\cap I)\) in \(\Fix^H(A,\Rel)\) and \(J^H(I,\Rel\cap I)\) in \(J^H(A,\Rel)\). 

To study saturatedness, we relate the corresponding ideals in the reduced crossed products for~\(I,A\) and \(A/I\).
\begin{lemma}\label{ses:reduced}
	Let \((A,\Rel)\) be a continuously square-integrable \(H\)\nb-\(\Cst\)\nb-algebra and 
	\(I\idealin A\) an \(H\)\nb-invariant ideal such that \eqref{exact} is exact. 
	
	The restrictions of \(\Cred(H,I)\to\Cred(H,A)\) and \(Q\colon\Cred(H,A)\to\Cred(H,A/I)\) to \(J^H(I,\Rel\cap I)\) and \(J^H(A,\Rel)\), respectively, yield a commutative diagram with exact rows
	\begin{equation}\label{crossedideals}
	\begin{tikzcd}
	J^H(I,\Rel\cap I)\arrow[r,hook]\arrow[d,hook] & J^H(A,\Rel)\arrow[r,twoheadrightarrow] \arrow[d,hook]& J^H(A/I,\cl{q(\Rel)})\arrow[d,hook]\\
	\Cred(H,I) \arrow[r,hook]& \Cred(H,A) \arrow[r,twoheadrightarrow,"Q"]& \Cred(H,A/I).
	\end{tikzcd}
	\end{equation}
\end{lemma}
\begin{proof}
	The ideal \(J^H(I,\Rel\cap I)\) is mapped into \(J^H(A,\Rel)\) under the inclusion. As \(Q(\BRAKET{a}{b})=\BRAKET{q(a)}{q(b)}\) for \(a,b\in\Rel\), we see that \(J^H(A,\Rel)\) maps to \(J^H(A/I,\cl{q(\Rel)})\). Moreover, the linear span of elements of this form is dense in \(J^H(A/I,\cl{q(\Rel)})\) so that the restriction is onto. Hence, the claim follows from exactness of the bottom row in \eqref{crossedideals} once we show that \(J^H(I,\Rel\cap I)=J^H(A,\Rel)\cap \Cred(H,I)\). 
	
	As \(J^H(A,\Rel)\) and \(\Cred(H,I)\) are both closed ideals in \(\Cred(H,A)\), \[J^H(A,\Rel)\cap\Cred(H,I)=J^H(A,\Rel)\cdot \Cred(H,I)\] holds. Consequently, the linear span of \(\BRAKET{a}{b}\circ\rho_f=\BRAKET{a}{b*f}\) for \(a,b\in\Rel\) and \(f\in\Cont_c(H,I)\) is dense in \(J^H(A,\Rel)\cap\Cred(H,I)\). Let \((u_\lambda)_{\lambda\in \Lambda}\) be a approximate unit for \(I\) consisting of self-adjoint \(u_\lambda\). \cref{approxunit} implies that \(\BRAKET{a}{b*f}\) is the limit of
	\(\rho_{u_\lambda}\circ\BRAKET{a}{b*f}=\BRAKET{a u_\lambda}{b*f}\). This net lies in \(J^H(I,\Rel\cap I)\) as \(au_\lambda\in\Rel\cdot I=\Rel\cap I\) and \(b*f\in\Rel\cap I\). Thus, the inclusion \(J^H(A,\Rel)\cap\Cred(H,I)\subseteq J^H(I,\Rel\cap I)\) follows. The converse inclusion is clear. 
\end{proof}

\begin{corollary}\label{res:ses_saturated}
	Let \((A,\Rel)\) be a continuously square-integrable \(H\)\nb-\(\Cst\)\nb-algebra and 
	\(I\idealin A\) an \(H\)\nb-invariant ideal such that \eqref{exact} is exact. Then \((A,\Rel)\) is saturated if and only if \((I,\Rel\cap I)\) and \((A/I,\cl{q(\Rel)})\) are saturated.
\end{corollary}

\begin{proof}
	Suppose first that \((A,\Rel)\) is saturated. In the proof of \cref{ses:reduced} we showed that \(J^H(I,\Rel\cap I)=J^H(A,\Rel)\cap\Cred(H,I)\). Hence \((I,\Rel\cap I)\) is saturated. Because~\eqref{crossedideals} has exact rows, this implies that \((A/I,\cl{q(\Rel)})\) is saturated as well. If \((I,\Rel\cap I)\) and \((A/I,\cl{q(\Rel)})\) are saturated, \((A,\Rel)\) is saturated because \eqref{crossedideals} is exact.
\end{proof}
As an application we show the following result for actions on spaces.
\begin{lemma}\label{res:freesaturated}
	Let \(H\) act properly on a locally compact Hausdorff space \(X\) and assume that \((\Cont_0(X),\cl{\Cont_c(X)})\) is saturated. Then the action \(H\curvearrowright X\) is free. 
\end{lemma}
\begin{proof}		
	Let \(x\in X\) and let \(Hx\subseteq X\) be its orbit. Since the action is proper, \(H x\) is \(H\)-equivariantly homeomorphic to \(H/H_x\). Here \(H_x\) is the stabilizer of~\(x\), which is a compact subgroup of \(H\). Hence, \(\Cont_0(Hx)\) is a quotient of \(\Cont_0(X)\) by an \(H\)\nb-invariant ideal. Because \(\Cont_0(Hx)\) is spectrally proper, \(\cl{\Cont_c(H x)}\) is the unique relatively continuous, complete and dense subset by \cref{res:spectrallyproper}. By \cref{res:ses_saturated}, \((\Cont_0(H x),\cl{\Cont_c(H x)})\) is saturated. Hence, \(\Fix^H(\Cont_0(Hx),\cl{\Cont_c(H x)})\) is Morita--Rieffel equivalent to \(\Cred(H,\Cont_0(H x))\). By the Imprimitivity Theorem, \(\Cred(H, \Cont_0(H/H_x))\) is Morita--Rieffel equivalent to \(\Cst(H_x)\).
	
	On the other hand, \(\Fix^H(\Cont_0(Hx),\cl{\Cont_c(H x)})\) is isomorphic to the functions on the orbit space. As \(H x\) consists of a single \(H\)-orbit, this generalized fixed point algebra is isomorphic to \(\C\).   Hence, \(\C\) and \(\Cst(H_x)\) are Morita--Rieffel equivalent. This can only be true if \(H_x=\{e\}\). Therefore, the \(H\)-action on \(X\) is free. 
\end{proof}

Not only the ideals in the crossed product algebras fit into an exact sequence. The same is true for the corresponding generalized fixed point algebras. The surjective homomorphism~\(q\colon A\to A/I\) has a unique strictly continuous extension \(\Mult(A)\to\Mult(A/I)\). Let \(\widetilde{q}\) be its restriction to \(\Fix^H(A,\Rel)\).
\begin{proposition}\label{res:sesfixed}
	Let \((A,\Rel)\) be a continuously square-integrable \(H\)\nb-\(\Cst\)\nb-algebra and 
	\(I\idealin A\) an \(H\)\nb-invariant ideal such that \eqref{exact} is exact. There is an extension of generalized fixed point algebras
	\begin{equation*}
	\begin{tikzcd}
	\Fix^H(I,\Rel\cap I)\arrow[r,hook] & \Fix^H(A,\Rel)\arrow[r,twoheadrightarrow,"\widetilde{q}"] & \Fix^H(A/I,\cl{q(\Rel)}).
	\end{tikzcd}
	\end{equation*}
\end{proposition}

\begin{proof}
	For \(a,b\in\Rel\cap I\), we can view \(\KET{a}\BRA{b}\) as a multiplier of \(I\) or \(A\). As \(\KET{a}\BRA{b}(A)\subset I\) it follows that \(\norm{\KET{a}\BRA{b}}_I=\norm{\KET{a}\BRA{b}}_A\). Hence, by extending continuously we obtain an injective \(^*\)-homomorphism \(\Fix^H(I,\Rel\cap I)\to\Fix^H(A,\Rel)\). 
	
	Denote by \(\beta\) the induced \(H\)-action on \(A/I\). Strict continuity of \(\widetilde{q}\) and \cref{res:strictlimit} imply
	\[\widetilde{q}(\KET{a}\BRA{b})=\lim_s \int_H q(\alpha_x(a^*b))\diff x= \lim_s \int_H \beta_x(q(a^*b))\diff x= \KET{q(a)}\BRA{q(b)}\]
	for \(a,b\in\Rel\). This shows that the image of \(\widetilde{q}\) is contained in \(\Fix^H(A,\cl{q(\Rel)})\). Moreover, the linear span of elements of this form is dense in \(\Fix^H(A,\cl{q(\Rel)})\). So \(\widetilde{q}\) is onto. 
	
	It remains to show that the kernel of \(\widetilde{q}\) is \(\Fix^H(I,\Rel\cap I)\). The computation above yields \(\widetilde{q}(\KET{a}\BRA{b})=\KET{q(a)}\BRA{q(b)}=0\) for \(a,b\in\Rel\cap I\). Thus, \(\Fix^H(I,\Rel\cap I)\) is contained in \(\ker(\widetilde{q})\). Let \(T\in\Fix^H(A,\Rel)\) be such that \(\widetilde{q}(T)=0\). By the \(\Cst\)-identity in \(\Fix^H(A,\Rel)/\Fix^H(I,\Rel\cap I)\) it will suffice to show that \(T^*T\in\Fix^H(I,\Rel\cap I)\). By \cite{meyer2001}*{(13)}, \(T^*\KET{a}\BRA{b}=\KET{T^*a}\BRA{b}\) holds for \(a,b\in\Rel\). As \(T^*a\) is square-integrable and \(\KET{T^*a}=T^*\KET{a}\in\Fix^H(A,\Rel)\cdot\Hilm^H(A,\Rel)\subseteq \Hilm^H(A,\Rel)\), \cite{meyer2001}*{6.5} implies \(T^*a\in\Rel\). Moreover, \(q(T^*a)=\widetilde{q}(T^*)q(a)=0\) shows that \(T^*a\in\Rel\cap I\). The equalities \(\Rel\cap I=\Rel\cdot I\) and \(I=I^2\) imply that there are \(c\in\Rel\) and \(i,j\in I\) with \(T^*a=cij\). The computation
	\[\KET{cij}\BRA{b}=(\KET{ci}\circ\rho_j)\circ\BRA{b}=\KET{ci}(\KET{b}\circ\rho_{j^*})^*=\KET{ci}\BRA{bj^*}\]
	shows that \(T^*\KET{a}\BRA{b}\in\Fix^H(I,\Rel\cap I)\). By definition of the generalized fixed point algebra, \(T\) is the limit of a sequence in the linear span of \(\KET{\Rel}\BRA{\Rel}\). Hence, it follows that \(T^*T\in\Fix^H(I,\Rel\cap I)\).
\end{proof}
\subsection{Generalized fixed point algebras and \(\Cont_0(X)\)-algebras}
In this section, we consider generalized fixed point algebras of \(H\)-\(\Cont_0(X)\)-algebras. For the definition of \(\Cont_0(X)\)-algebras and their relation to continuous fields of \(\Cst\)-algebras see \cite{nilsen1996}.
\begin{definition}[\cite{Kasparov_equivariant}]
	Let \(H\) act on a locally compact Hausdorff space \(X\). Denote by \(\tau_h(f)(x)=f(h^{-1}\cdot x)\) for \(h\in H\) and \(x\in X\) the induced action on \(\Cont_0(X)\). 	
	A \(\Cont_0(X)\)-algebra \(A\) with \(\theta\colon \Cont_0(X)\injto Z\Mult(A)
	\) and an \(H\)-action \(\alpha\colon H\acts A\) is called an \emph{\(H\)-\(\Cont_0(X)\)-algebra} if the actions are compatible in the following sense
	\begin{align*}
	\alpha_h(\theta(\varphi)a)=\theta(\tau_h(\varphi))\alpha_h(a)\quad\text{for all }h\in H\text{, }\varphi\in\Cont_0(X)\text{ and }a\in A.
	\end{align*}
\end{definition}

In this case, one can relate the spectra of the \(H\)-\(\Cont_0(X)\)-algebra and the corresponding generalized fixed point algebra as follows. If \(H\curvearrowright A\) is a strongly continuous action, there is a continuous action of \(H\) on the spectrum \(\widehat{A}\) given by \((x\cdot\pi)(a)=\pi(\alpha_{x^{-1}}(a))\) for \([\pi]\in\widehat{A}\), \(x\in H\) and \(a\in A\). 

Suppose now \((A,\Rel)\) is a continuously square-integrable  \(H\)-\(\Cst\)-algebra. Every non-degenerate representation \(\pi\) of \(A\) can be extended to \(\Mult(A)\). Denote its restriction to \(\Fix^H(A,\Rel)\) by \(\widetilde{\pi}\).

In the commutative case, this procedure allows to completely describe the representation theory of the generalized fixed point algebra.	By Gelfand duality, the spectrum of the commutative \(\Cst\)-algebra \(\Cont_0(X)\) is \(X\). If \(H\acts X\) is a proper action and \((A,\Rel)=(\Cont_0(X),\cl{\Cont_c(X)})\), the map \(\pi\mapsto \widetilde{\pi}\) induces a homeomorphism \[H\backslash\widehat{A}=H\backslash X\to\widehat{\Fix^H(A,\Rel)}.\] 

This generalizes to \(H\)-\(\Cont_0(X)\)-algebras as follows.
We summarize some results concerning \(H\)-\(\Cont_0(X)\)-algebras proved in \cites{meyer2001,anhuefequivariantbrauer,rieffel1998,anhuef2002}.
\begin{proposition}\label{res:saturatedspectrum}
	Let \(H\curvearrowright X\) be a proper action on a locally compact Hausdorff space \(X\) and \(A\) an \(H\)-\(\Cont_0(X)\)-algebra with \(\theta\colon\Cont_0(X)\to Z\Mult(A)\).
	\begin{enumerate}
		\item The subset \(\Rel\defeq \overline{\theta(\Cont_c(X))A}\) is  dense, relatively continuous and complete. Moreover, it is the unique such subset.
		\item If the action of \(H\) on \(X\) is free, the map \(\pi\mapsto\widetilde{\pi}\) induces a homeomorphism \(H\backslash \widehat{A}\to \widehat{\Fix^H(A,\Rel)}\) and \((A,\Rel)\) is saturated.
	\end{enumerate}		 
\end{proposition}
\begin{proof}An \(H\)-\(\Cont_0(X)\)-algebra with a proper action \(H\acts X\) is spectrally proper as argued in \cite{meyer2001}*{Sec.~9}. The \(\Rel\) above is the unique, dense, relatively continuous and complete subset constructed in \cite{meyer2001}*{9.4}. 		
	It is proved in \cite{anhuefequivariantbrauer}*{3.9} that the map \(\pi\mapsto\widetilde{\pi}\) induces a homeomorphism. That \((A,\Rel)\) is saturated is shown in the preprint version of \cite{rieffel1998} and in \cite{anhuef2002}*{4.1}.
\end{proof}

The following result applies to trivial continuous fields \(\Cont_0(X,A)\), where the action is taking place on the base space \(X\). 
\begin{lemma}[\cite{rieffel1988}*{2.6}, \cite{raeburnpullbacks}*{3.2}]\label{res:trivialfix}
	Let \(H\curvearrowright X\) be a proper action and \(A\) a \(\Cst\)-algebra. Let \(H\) act on \(\Cont_0(X,A)\) by \((\tau_hf)(x)=f(h^{-1}\cdot x)\) for \(h\in H\), \(f\in\Cont_0(X,A)\) and \(x\in X\). There is an isomorphism
	\begin{align*}\Psi&\colon\Fix^H(\Cont_0(X,A),\cl{\Cont_c(X,A)})\to \Cont_0(H\backslash X, A),\\
	\Psi(\KET{f}\BRA{g})(Hx)&=\int_H (f^*\cdot g)(h^{-1}\cdot x)\diff h\quad\text{for } Hx\in H\backslash X \text{ and }f,g\in\Cont_c(H,A).
	\end{align*}		
\end{lemma}
In particular, for \(A=\C\) the generalized fixed point algebra is isomorphic to \(\Cont_0(H\backslash X)\) as claimed before. 
In contrast to the situation above, consider now what happens if the \(H\)-action only takes place in the fibres of a field of \(\Cst\)-algebras. 

\begin{theorem}[\cite{rieffel1988}*{3.2}]\label{res:fixcontfield}
	Let \(A\) be a continuous field of \(\Cst\)-algebras over~\(X\) with fibre projections \(p_x\colon A\to A_x\) for \(x\in X\). Suppose that \((A,\Rel)\) is a continuously square-integrable \(H\)-\(\Cst\)-algebra such that \(\ker(p_x)\) is \(H\)-invariant for all \(x\in X\). Furthermore, assume that \(H\) is a \(\sigma\)-compact, exact group.	
	Then \(\Fix^H(A,\Rel)\) is a continuous field of \(\Cst\)-algebras over \(X\) with fibre projections \[\widetilde{p}_x\colon \Fix^H(A,\Rel)\to\Fix^H(A_x,\cl{p_x(\Rel)}).\]
\end{theorem}
\begin{remark}
	The above theorem in \cite{rieffel1988} requires that the field \(A\) is Hilbert continuous and that \(\Cred(H,A_x)=\Cst(H,A_x)\) for all \(x\in M\). These assumptions are only needed to show that \(\Cred(H,A)\) defines a continuous field of \(\Cst\)-algebras over~\(X\) with fibres \(\Cred(H,A_x)\). But this is true for any \(\sigma\)-compact, exact group \(H\) by \cite{wassermann}*{4.2}.
\end{remark}
Now we study what can be said about saturatedness in this case.
\begin{lemma}\label{res:properideal}
	Let \(A\) be an upper semi-continuous field of \(\Cst\)\nb-algebras over \(X\) with fibre projections \(p_x\colon A\to A_x\). If \(I \properideal A\) is a proper ideal, then there is \(x \in X\) such that \(p_x(I)\idealin A_x\) is a proper ideal. 
\end{lemma}
\begin{proof}
	By Lee's Theorem (see \cite{lee} or \cite{nilsen1996}*{3.3}) there is a continuous map \(\psi\colon \Prim(A)\to X\) satisfying
	\[ \psi(P)=x \,\Leftrightarrow P\subseteq K_x=\{a\in A\mid p_x(a)=0\}\]
	and \(A_x\cong A /K_x\) for all \(x\in X\). As \(I\) can be written as the intersection of primitive ideals, it follows that there is a primitive ideal \(P\in\Prim(A)\) with \(I\subseteq P \subsetneq A\). Let \(x=\psi(P)\). The homeomorphism \(\{Q\in \Prim(A)\mid K_x\subseteq P\}\to \Prim(A/K_x)=\Prim(A_x)\) maps \(P\) to \(p_x(P)\). Then \(p_x(I)\subseteq p_x(P)\subseteq A_x\), and \(p_x(P)\neq A_x\) as otherwise \(p_x(P)\) would correspond to \(A\) under this homeomorphism. 
\end{proof}
\begin{corollary}\label{res:field_saturated}
	In the situation of \cref{res:fixcontfield}, \((A,\Rel)\) is saturated if and only if  \((A_x,\cl{p_x(\Rel)})\) is saturated for all \(x\in X\). 
\end{corollary}
\begin{proof}
	Suppose first that \((A,\Rel)\) is saturated. By \cref{res:ses_saturated} \((A_x,\cl{p_x(\Rel)})\) is saturated for all \(x\in X\). 	
	Assume now that all \((A_x,\cl{p_x(\Rel)})\) are saturated. By assumption on \(H\), \(\Cred(H,A)\) is a continuous field of \(\Cst\)\nb-algebras over \(X\) with fibre projections \((p_x)_*\colon \Cred(H,A)\to\Cred(H,A_x)\) for \(x\in X\). Because 
	\[ (p_x)_{*}(\BRAKET{a}{b})= \BRAKET{p_x(a)}{p_x(b)} \quad\text{for } a,b\in \Rel,\]
	it follows that \((p_x)_*(J^H(A,\Rel))=\Cred(H,A_x)\) for all \(x\in X\). Now \cref{res:properideal} implies that \(J^H(A,\Rel)=\Cred(H,A)\). 
\end{proof}

\section{Homogeneous Lie groups}\label{homogeneous}

In the following, we will consider homogeneous Lie groups, which are Lie groups that are equipped with a dilation action of \(\Rp\). They allow to define homogeneity with respect to the dilations. A detailed discussion of homogeneous Lie groups can be found in \cite{folland1982homogeneous} or \cite{fischer2016quantization}. We recall some notions used there, which proved to be convenient to do analysis on these groups. 
\subsection{Homogeneous and graded Lie groups}
\begin{definition}
	A \emph{homogeneous Lie group} is a simply connected Lie group \(G\) whose Lie algebra~\(\lieg\) is equipped with a family of \emph{dilations} \(\{A_\lambda\colon\lieg\to\lieg\}_{\lambda>0}\). That is, there is a diagonalizable, linear map \(D\colon \lieg\to\lieg\) with positive eigenvalues \(q_1\leq q_2\leq \ldots \leq q_n\), such that all \(A_\lambda\defeq\mathrm{Exp}(D\ln(\lambda))\) are Lie algebra homomorphisms. Here, \(\mathrm{Exp}\) denotes the matrix exponential.
	The eigenvalues \(q_1,\ldots,q_n\) are called \emph{weights}. 
\end{definition}

Folland and Stein assume in \cite{folland1982homogeneous} that \(q_1=1\). This can be achieved by scaling appropriately. We shall also assume this in the following, in particular, all weights satisfy \(q_j\geq 1\).
Fix a corresponding basis of eigenvectors \(\{X_1,\ldots,X_n\}\) of \(D\). Then 
\(A_\lambda(X_j)=\lambda^{q_j}X_j\) for \(1\leq j\leq n\). 
If \(X,Y\) are eigenvectors to the eigenvalues~\(q_i,q_j\) of \(D\), respectively, it follows from \(A_\lambda[X,Y]=[A_\lambda(X),A_\lambda(Y)]=\lambda^{q_i+q_j}[X,Y]\), that \([X,Y]=0\) or that \([X,Y]\) is an eigenvector of \(D\) to the eigenvalue \(q_i+q_j\). From that one deduces that \(\lieg\), and therefore \(G\), is nilpotent. 
Consequently, the exponential map \(\exp\colon \lieg\to G\) is a diffeomorphism. In the following, we often identify \((x_1,\ldots,x_n)\in\R^n\) with its image \(\exp(x_1X_1+\cdots + x_nX_n)\in G\) under this global coordinate chart. In particular, \(0\in G\) denotes the neutral element. 

Because \(A_\lambda\circ A_\mu= A_{\lambda\mu}\) for \(\lambda,\mu>0\), the dilations define an action \(A:\Rp\acts \lieg\) by Lie group automorphisms. Denote by \(\alpha\colon \Rp\acts G\) the corresponding action by Lie group automorphisms.

\begin{definition}
	A \emph{graded Lie group} is a simply connected Lie group \(G\) such that its Lie algebra~\(\lieg\) admits a finite decomposition\[\lieg = \bigoplus_{j=1}^N {\lieg_j},\] with \([X,Y]\in \lieg_{j+k}\) for all \(X\in\lieg_j\) and \(Y\in\lieg_k\), where \(\lieg_j=\{0\}\) for \(j>N\). 
\end{definition}
	For a graded Lie group \(G\) with grading as above, \(A_\lambda(X)=\lambda^jX\) for \(X\in\lieg_j\) and \(\lambda>0\) defines a family of dilations. So \(G\) becomes a homogeneous Lie group. 
	However, homogeneous Lie groups are slightly more general. If all weights of a homogeneous Lie group are rational numbers, it is a (scaled) graded nilpotent Lie group (see \cite{fischer2016quantization}*{3.1.9}).	
	Also note that there are nilpotent Lie groups that do not admit a family of dilations as above (see \cite{dyer1970}).

\begin{example}
	A famous example of a homogeneous Lie group is the \emph{Heisenberg group}. Its Lie algebra \(\lieg\) is generated by \(\{X,Y,Z\}\) and \([X,Y]=Z\), \([X,Z]=0\) and \([Y,Z]=0\). 
	Hence, \(\lieg_1=\lspan\{X,Y\}\), \(\lieg_2=\lspan\{Z\}\) and \(\lieg_j=0\) for \(j> 2\) defines a grading on \(\lieg\). 
\end{example}

\begin{example}A Lie algebra \(\lieg\) may be equipped with different dilations. Choose a basis \(\{X_1,\ldots,X_n\}\)  for the Lie algebra of the Abelian group \(G=\R^n\). Then for all \((q_1,\ldots,q_n)\in\R^n_{>0}\) there is a dilation defined by \(DX_j=q_jX_j\). The standard dilation action on \(\R^n\) is given by scalar multiplication, that is, \(q_j=1\) for all \(j=1,\ldots,n\).
\end{example}

\subsection{Analysis on homogeneous Lie groups}
 \begin{definition}
	The \emph{homogeneous dimension} of a homogeneous Lie group \(G\) with weights \(1=q_1\leq q_2 \leq \ldots\leq q_n\) is defined as \(Q= q_1+q_2+\ldots+q_n\).
	A function \(f\) on \(G\!\setminus\!\{0\}\) is called \emph{\(\mu\)\nb-homogeneous} for \(\mu\in\C\) if \(f(\alpha_\lambda(x))=\lambda^\mu f(x)\) for all \(\lambda>0\) and \(x\neq 0\). 
\end{definition}

\begin{lemma}\label{haar} 
Let \(G\) be a homogeneous Lie group of homogeneous dimension \(Q\).
	The pullback of the Lebesgue measure under the exponential map defines a Haar measure on \(G\). The group \(G\) is unimodular and the Haar measure is \(Q\)\nb-homogeneous, that is,
	\[\int_G f(\alpha_\lambda(x))\diff x = \lambda^{-Q} \int_G f(x)\diff x\]
	for each \(\lambda>0\) and \(f\in L^1(G)\).
\end{lemma}

For simply connected nilpotent Lie groups it is true in general that the pullback of the Lebesgue measure defines a left and right Haar measure, see \cite{folland1982homogeneous}*{1.2}. The \(Q\)\nb-homogeneity follows from the behaviour of the Lebesgue measure under scaling.
\begin{definition}\label{def:multiindex}
	For a multi-index \(\alpha\in\N^n_0\) its \emph{homogeneous degree} is defined as \([\alpha]\defeq \alpha_1q_1+\ldots+\alpha_nq_n\). A function \(P\) on~\(G\) is called \emph{polynomial} if \(P\circ \exp\) is polynomial.
\end{definition}
\begin{example}
	The polynomials \(x^\alpha\) for \(\alpha\in\N^n_0\) are \([\alpha]\)\nb-homogeneous functions on~\(G\). 
\end{example}
The group law of a homogeneous Lie group is of a triangular form. Using the Baker--Campbell--Hausdorff formula and the homogeneity of the coordinate functions, the following is proved in \cite{folland1982homogeneous}*{p.~23}.

\begin{proposition}\label{triangularprop}
	For a homogeneous Lie group \(G\) with weights \(q_1\leq\ldots\leq q_n\) and a basis of eigenvectors \(X_1,\ldots,X_n\in\lieg\) there are constants \(c_{j,\alpha,\beta}\) for \(j=1,\ldots,n\) such that for all \(x,y\in G\) with respect to this basis
	\begin{equation}\label{triangular}
	(x\cdot y)_j = x_j+y_j+\sum_{\substack{\alpha,\beta\in\N^n_0\setminus\{0\}\\ [\alpha]+[\beta]=q_j}}{c_{j,\alpha,\beta}x^\alpha y^\beta}.
	\end{equation}
\end{proposition}

\noindent The basis of eigenvalues fixed above induces left- and right-invariant differential operators \(X_1,\ldots,X_n\) and \(Y_1,\ldots,Y_n\) on \(G\) by setting for \(f\in\Cont^1(G)\)
\begin{align} \label{def:leftinvariant_diff}
(X_jf)(x)&=\frac{\diff}{\diff t}f(x\cdot\exp(tX_j))\big|_{t=0},\\
(Y_jf)(x)&=\frac{\diff}{\diff t}f(\exp(tX_j)\cdot x)\big|_{t=0}.
\end{align}
Define for a multi-index \(\alpha\in \N^n_0\) the left-invariant differential operator \(X^\alpha=X_1^{\alpha_1}X_2^{\alpha_2}\cdots X_n^{\alpha_n}\). 
The triangular group law allows to express these in terms of the partial differential operators as follows. 
\begin{proposition}[\cite{fischer2016quantization}*{3.1.28}]\label{res:vectorfields}
Let \(G\) be a homogeneous Lie group with weights \(q_1\leq\ldots\leq q_n\). For \(j=1,\ldots,n\) and \(k>j\) there are \((q_k-q_j)\)\nb-homogeneous polynomials \(P_{jk}\) and \(Q_{jk}\) such that the vector fields \(X_j\) and \(Y_j\) defined above can be written as
\begin{align*}
X_j & = \frac{\partial}{\partial x_j}+\sum_{q_k>q_j}P_{jk}\frac{\partial}{\partial x_k} = \frac{\partial}{\partial x_j}+\sum_{q_k>q_j}\frac{\partial}{\partial x_k}P_{jk},\\Y_j & = \frac{\partial}{\partial x_j}+\sum_{q_k>q_j}Q_{jk}\frac{\partial}{\partial x_k} = \frac{\partial}{\partial x_j}+\sum_{q_k>q_j}\frac{\partial}{\partial x_k}Q_{jk}.
\end{align*} 
\end{proposition}
The polynomials \(P_{jk}\) and \(Q_{jk}\) only depend on \(x_1,\ldots,x_{k-1}\) because otherwise
they would be homogeneous of a higher order than \(q_k-q_j\). Hence, they commute
with the partial derivatives \(\frac{\partial}{\partial x_k}\).

Because the Euclidean norm does not behave well with respect to the dilations, homogeneous quasi-norms are used instead.

\begin{definition}[\cite{fischer2016quantization}*{3.1.33}]
A \emph{homogeneous quasi-norm} on a homogeneous Lie group \(G\) is a continuous function \(\norm{\,\cdot\,}\colon G\to [0,\infty)\) that satisfies \(\norm{x}=0\) if and only if \(x=0\), \(\norm{x^{-1}}=\norm{x}\) and \(\norm{\alpha_\lambda(x)}=\lambda\norm{x}\) for all \(x\in G\) and \(\lambda\in\Rp\).
\end{definition}

In the following, we fix a homogeneous quasi-norm on \(G\). For instance, let
\begin{equation*}
\norm{x}\defeq\sum_{j=1}^n {\abs{x_j}^{1/q_j}} \qquad \text{for }x\in G.
\end{equation*} 
In fact, by \cite{fischer2016quantization}*{3.1.35} all homogeneous quasi-norms on a given homogeneous Lie group are equivalent. There is an analogue of the triangle inequality and its consequences for a homogeneous quasi-norm. 
\begin{lemma}[\cite{folland1982homogeneous}*{1.8, 1.10}]\label{triangle}
	Let \(G\) be a homogeneous Lie group. There is a constant \(C\geq 1\) such that for all \(x,y\in G\)
	\begin{enumerate}[(a)]
		\item \(\norm{xy}\leq C(\norm{x}+\norm{y})\),
		\item \((1+\norm{x})^s(1+\norm{y})^{-s}\leq C^s(1+\norm{xy^{-1}})^s\) for all \(s>0\).
	\end{enumerate}
\end{lemma}
Identifying \(G\) with \(\R^n\) one can consider the Schwartz space \(\Schwartz(G)\). 
The following family of seminorms will be useful later on. 

\begin{definition}\label{seminorm}
For the fixed homogeneous quasi-norm \(\norm{\,\cdot\,}\) define for \(N\in\N_0\) 
\begin{equation*} 
\norm{f}_{N}\defeq \sup_{\abs{I}\leq N\text{, }x\in G}{(1+\norm{x})^{(N+1)(Q+1)}\abs{(\partial^If)(x)}} \quad \text{for }f\in\Cont^\infty(G).
\end{equation*}
\end{definition}
Polynomials in \(\norm{x}_E\) for the Euclidean norm can be estimated by polynomials in~\(\norm{x}\) for a homogeneous quasi-norm and vice versa. 
Thus, a sequence \((f_n)\) converges to \(f\) in \(\Schwartz(G)\) if and only if \(\norm{f-f_n}_{N}\to 0\) for all \(N\in\N_0\). 

The following integrability criterion for functions on a homogeneous Lie group is used later on. 

\begin{lemma}[\cite{folland1982homogeneous}*{1.17}]\label{res:convint}
Let \(\alpha\in\R\) and let \(f\) be a measurable function on a homogeneous Lie group \(G\) of homogeneous dimension \(Q\). Suppose \(\abs{f(x)}=O(\abs{x}^{\alpha-Q})\). If \(\alpha>0\) then \(f\) is integrable near \(0\). If \(\alpha<0\), then \(f\) is integrable near~\(\infty\). 
\end{lemma}

\subsection{Representation theory of homogeneous Lie groups}

Now, we recall some facts about the representation theory of nilpotent Lie groups \(G\) and their group \(\Cst\)-algebras. For homogeneous Lie groups the dilations induce actions on the respective spaces of representations.

The continuous compactly supported functions \(\Cont_c(G)\) become a *-algebra when equipped with the convolution and involution defined by
\begin{align*}
f^*(x)=\conj{f(x^{-1})}\quad \text{and}\quad
(f*g)(x)=\int_G f(y)g(y^{-1}x)\diff y\quad\text{for }x\in G.
\end{align*}
Denote by \(\widehat{G}\) the set of equivalence classes of irreducible, unitary representations \(\pi\colon G\to \mathcal{U}(\Hils_\pi)\). For such a representation \(\pi\) and \(f\in\Cont_c(G)\) define the operator
\begin{align*}
\widehat{\pi}(f)=\int_G f(x)\pi(x)\diff x\in\Bound(\Hils_\pi)\qquad \text{for }f\in \Cont_c(G).
\end{align*}
This defines a *-representation \(\widehat{\pi}\colon\Cont_c(G)\to\Bound(\Hils_\pi)\). The \emph{full group \(\Cst\)-algebra}~\(\Cst(G)\) is defined as the closure of \(\Cont_c(G)\) with respect to \(\norm{f}=\sup_{\pi\in\widehat{G}}\norm{\widehat{\pi}(f)}\). By \cite{dixmiernilpotent} homogeneous Lie groups are liminal so that all representations \(\widehat{\pi}\) map onto the compact operators \(\Comp(\Hils_\pi)\). 

The homogeneous structure allows to define an \(\Rp\)\nb-action on \(\widehat{G}\). For an irreducible, unitary representation \(\pi\) set  \((\lambda\cdot\pi)(x)=\pi(\alpha_{\lambda}(x))\) for \(\lambda>0\) and \(x\in G\). It is easy to see that \(\lambda\cdot \pi\) is again an irreducible, unitary representation and that the action is well-defined on the equivalence classes. 

Furthermore, define an action on \(\Cst(G)\) by \(\sigma_\lambda(f)(x)=\lambda^Qf(\alpha_\lambda(x))\) for \(f\in\Cont_c(G)\). Using \cref{haar} one checks that each \(\sigma_\lambda\) is a \(^*\)-homomorphism and an isometry with respect to the \(\Cst\)\nb-norm. This action induces in turn an action on the representations of \(\Cst(G)\) by \((\lambda\cdot \rho)(f)=\rho(\sigma_{\lambda}(f))\) for a \(^*\)-representation \(\rho\colon\Cst(G)\to\Bound(\Hils_\pi)\). It is well-defined on the equivalence classes of irreducible representations in \(\widehat{\Cst(G)}\). 

\begin{proposition}
	Let \(G\) be a homogeneous Lie group. The map \(\widehat{G}\to\widehat{\Cst(G)}\) induced by \(\pi\mapsto\widehat{\pi}\) 
	is an \(\Rp\)\nb-equivariant homeomorphism.  
\end{proposition}

\begin{proof}
	It is well-known that the map above is a homeomorphism for each locally compact group \(G\). The equivariance under the \(\Rp\)\nb-action follows from the \(Q\)\nb-homogeneity of the Haar measure as
	\[ (\lambda\cdot\widehat{\pi})(f) = \int_G (\sigma_{\lambda}f)(x)\pi(x) \diff x= \int_G f(x)\pi(\lambda\cdot x)\diff x = \widehat{\lambda\cdot\pi}(f)\]
	for \(\lambda>0\) and \(f\in \Cont_c(G)\).
\end{proof}

Kirillov's orbit method \cite{kirillov1962} allows to describe \(\widehat{G} \) as the orbit space of the coadjoint action of \(G\) on \(\lieg^*\), the dual of its lie algebra \(\lieg\).
Recall that the \emph{adjoint representation} \(\Ad\colon G\to \Aut(\lie{g})\) is given by \(\Ad(x)=d(\alpha_x)_0\colon T_0G\to T_0G\), where \(\alpha_x(y)=xyx^{-1}\) is given by conjugation. 
The \emph{coadjoint action} is defined by
\[ \langle \coAd(x) l, X\rangle \defeq \langle l, \Ad(x^{-1})X\rangle
\quad \text{for }l\in\lie{g}^*\text{, }x\in G\text{ and }X\in\lie{g}. \] The corresponding infinitesimal representation \(\coad\) of \(\lie{g}\) on \(\lie{g}^*\) is given by \(\langle\coad(X)l, Y\rangle =\langle l, [Y,X]\rangle\) for \(l\in\lie{g}^*\) and \(X,Y\in\lie{g}\). 

The orbit space \(G\backslash\lie{g}^*\) is equipped with the quotient topology. 	
For each \(l\in\lie{g}^*\), one can construct a unitary representation \(\pi_l\) of \(G\) as described in the following. Define a skew-symmetric bilinear form \(b_l\colon \lie{g}\times \lie{g} \to \R\) by 
\[b_l(X,Y)=\langle l,[X,Y]\rangle \quad\text{for }X,Y\in\lie{g}.\]
Denote by \(\lie{g}_l\) its radical. A subspace \(V\subset\lie{g}\) is isotropic with respect to \(b_l\) if \(b_l(X,Y)=0\) for all \(X,Y\in V\). A maximal isotropic subspace has codimension~\(\tfrac12\dim(\lie{g}/\lie{g}_l)\). 
A \emph{polarizing subalgebra} for \(l\) is a subalgebra \(\lie{h_l}\subset \lie{g}\) that is an isotropic subspace of codimension \(\tfrac12\dim(\lie{g}/\lie{g}_l)\). Such a polarizing subalgebra always exists (see \cite{corwin1990representations}*{1.3.3, 1.3.5}). 

 The formula \(\chi_l(\exp X)=e^{i\langle l,X\rangle}\) for \(X\in\lieh_l\) defines a one-dimensional representation of \(H_l=\exp(\lieh_l)\). It is multiplicative because if \(\exp X\cdot \exp Y = \exp Z\) for \(X,Y\in\lieh_l\), then \(Z\) is given by the Baker-Campbell-Hausdorff formula as
\[ Z = X+ Y+ \frac12 [X,Y]+\frac1{12}[X,[X,Y]]+\cdots,\] 
so that all higher terms lie in \([\lieh_l,\lieh_l]\subset\ker l\).
Denote by \(\pi_l=\Ind_{H_l}^G \chi_l\) the induced representation of \(\chi_l\) to \(G\). 

Let \(\Rp\) act on \(\lieg^*\) by \(\langle \lambda\cdot l, X\rangle = \langle l, A_\lambda(X)\rangle\) for \(\lambda>0\), \(l\in \lieg ^*\) and \(X\in \lieg\). This action descends to the orbit space of the co-adjoint action as \(A_\lambda\circ \Ad(x)=\Ad(\alpha_\lambda(x))\circ A_\lambda\) for \(\lambda>0\) and \(x\in G\). 

\begin{lemma}[\cite{corwin1990representations}*{2.1.3}]\label{ind}
	Let \(H\) be a subgroup of a locally compact group \(G\) and let \(\alpha\) be an automorphism of \(G\) and \(\pi\) a unitary representation of \(H\). Then \(\alpha^{-1}(H)\) is also a subgroup and
	\[ \Ind_{\alpha^{-1}(H)}^G\left(\pi\circ\alpha\right)\simeq \left(\Ind_H^G\pi\right)\circ \alpha. \]
\end{lemma}

\begin{lemma}
	Kirillov's map \(G\backslash\lie{g}^* \to \widehat{G}\) induced by \(l\mapsto \pi_l\) is an \(\Rp\)\nb-equivariant homeomorphism. 
\end{lemma}

\begin{proof}
	Kirillov proved in \cite{kirillov1962} that the map is a well-defined, so in particular, the equivalence class of \(\pi_l\) does not depend on the choice of the polarizing subalgebra~\(\lieh_l\). Two representations \(\pi_{l_1}\) and \(\pi_{l_2}\) are equivalent if and only if \(l_1\) and \(l_2\) lie in the same co-adjoint orbit. Moreover, he proved that the map is continuous and onto. The continuity of the inverse map is due to \cite{brown1973}. To see that the map is equivariant, note that \(\chi_{\lambda.l}=\chi_{l}\circ\alpha_{\lambda}\) and that \(\alpha_{\lambda^{-1}}(H_l)\) is a polarizing algebra for \(\lambda\cdot l\). Hence,  \cref{ind} yields that \( \pi_{\lambda\cdot l}\simeq \lambda\cdot\pi_l\).
\end{proof}
All \(l\in \lieg\) that vanish on \([\lieg,\lieg]\) induce one-dimensional representations~\(\pi_l\). In particular, \(l = 0\) induces the trivial representation on \(\C\). If the polarizing algebra is not all of \(\lieg\), the corresponding Hilbert space is infinite-dimensional.

\subsection{Stratification}
The goal of this section is to use Kirillov theory and the coarse stratification by Pukanszky \cite{pukanszky} to find a sequence of increasing, open, \(\Rp\)\nb-invariant subsets 
\begin{align}\label{sets}
\emptyset = V_0 \subset V_1 \subset V_2 \subset \ldots \subset V_m = \widehat{G}\!\setminus\!\{\pi_{\triv}\}\end{align}
such that all \(\Lambda_i:=V_{i}\setminus V_{i-1}\) are Hausdorff and the \(\Rp\)\nb-action on each of these subsets is free and proper. This sequence will play an essential role in \cref{fullness}.
Note that the following construction to find such a sequence of open subsets 1works for all simply connected nilpotent Lie groups. However, a dilation action is only defined for homogeneous Lie groups.

We start by describing \emph{Puk\'{a}nszky's stratification} of \(\lie{g}^*\). Recall that we fixed a basis \(\{X_1,\ldots,X_n\}\) of \(\lie{g}\) such that \(D_\lambda(X_j)=\lambda^{q_j}X_j\) for all \(1\leq j\leq n\). By the triangular form of the group law \eqref{triangular} all 
\[\lie{k}_i = \lspan\{X_{i+1},\ldots,X_n\}\qquad \text{for }i=0,\ldots,n\] form an ideal in \(\lie{g} \). 
In particular, \(\{X_1,\ldots,X_n\}\) is a strong Malcev basis of \(G\) as defined in \cite{corwin1990representations}*{1.1.13}, which is also called a Jordan--H{\"o}lder basis in \cite{pukanszky}. Note that they require \(\lspan\{X_1,\ldots,X_i\}\) to be ideals. We stick to the reversed ordering of the basis as this is standard for homogeneous Lie groups.

Let \(\{X_1^*,\ldots,X_n^*\}\) denote the corresponding dual basis of \(\lie{g}^*\) and let \(\lie{k}_i^\perp = \lspan\{X_{1}^*,\ldots,X_i^*\}\) for \(i=0,\ldots,n\).
An element \(l\in\lie{g}^*\) is contained in \(\lie{k}_i^\perp\) if and only if \(\langle l, \lie{k}_{i}\rangle = 0\). As the \(\lie{k}_i\) are ideals and are, therefore, invariant under the adjoint action,	the \(\lie{k}_i^\perp\) are invariant under the coadjoint action.
Hence \(G\) acts on each~\(\lie{g}^*/\lie{k}_i^\perp\). 

Write \(p _i\colon\lie{g}^*\to \lie{g}^*/\lie{k}_i^\perp\) for the projections. By \cite{corwin1990representations}*{3.1.4} the orbits \(G\cdot p_i(l)\) of \(p_i(l)\) under the coadjoint action are closed. Hence, they define submanifolds of \(\lie{g}^*/\lie{k}_i^\perp\). Following \cite{corwin1990representations},  make the following definition. 

\begin{definition}
	For \(l\in\lie{g
	}^*\) let \(d(l)= (d_0(l),d_1(l),\ldots, d_{n-1}(l))\) denote the sequence of orbit dimensions \(d_i(l) = \dim (G\cdot p_i(l))\).
\end{definition}
The entries of \(d(l)\) are decreasing. The corresponding stabilizer subgroups \(G_{p_i(l)}\) form an increasing sequence
\[G_{p_0(l)}\subseteq G_{p_1(l)}\subseteq \ldots\subseteq G_{p_{n-1}(l)}.\]
The same is true for their Lie algebras \(\lie{g}_{p_i(l)}\). By \cite{corwin1990representations}*{3.1.1} they are given by
\begin{align*}\lie{g}_{p_i(l)}&=\{ X\in\lie{g}\mid \coad(X) l\in \lie{k}_i^\perp\}\\
&= \{X\in\lie{g}\mid \langle l, [X,X_k]\rangle = 0 \text{ for }k=i+1,\ldots, n \}.
\end{align*}
\begin{example}\label{heisenberg}
	The computation in \cite{corwin1990representations}*{3.1.11} of the coadjoint action on the \(3\)-dimensional Heisenberg group \(H_1\) yields
	\[\coAd(x,y,z)(\alpha X^*+\beta Y^*+\gamma Z^*)=(\alpha+y\gamma)X^*+(\beta-x\gamma)Y^*+\gamma Z^*\]
	for \((x,y,z)\in H_1\) and \(\alpha,\beta,\gamma\in\R\). This shows for \(X_1=X\), \(X_2=Y\) and \(X_3=Z\) that
	\begin{align*}
	d(\alpha X^*+\beta Y^*+\gamma Z^*)&=(2,1,0) \quad\text{if }\gamma\neq 0\text{,}\\
	d(\alpha X^*+\beta Y^*)&=(0,0,0)\text{.}
	\end{align*}
\end{example}	

With the help of the next lemma an argument by Puk\'{a}nszky \cite{pukanszky1971unitary}*{p.~70} shows that the definition of \(d(l)\) as above coincides with the one given, for example, in \cite{beltictua2016fourier}, by jump indices. 
\begin{lemma}\label{skewbilinear}
	Let \(b\colon V\times V\to \R\) be a skew-symmetric bilinear form, \(V^\perp\) its radical and \(W\subset V\) a subspace. Then
	\[ \dim(W)+\dim(W^\perp) = \dim(V)+\dim(W\cap V^\perp).
	\]
\end{lemma}
\begin{lemma}\label{jump}
	The dimensions in \(d(l)\) decrease by steps of zero or one. There is a jump, that is, \(d_{i-1}(l) = d_{i}(l)+1\) if and only if
	\[X_i \notin \lie{g
	}_l+ \Span\{X_{i+1},\ldots,X_n\}.\]
\end{lemma}

\begin{proof} The orthogonal complement of \(\lie{g
	}_l+ \Span\{X_{i+1},\ldots,X_n\}\) with respect to the bilinear form \(b_l\) is \(\lie{g}_{p_i(l)}\). Hence, by \cref{skewbilinear} there is a change of dimension if and only if the dimension of the orthogonal complement decreases. This is the case if and only if \(X_i \notin \lie{g}_l+ \Span\{X_{i+1},\ldots,X_n\}\).
\end{proof}

Let \(D\) denote the finite set of all dimension sequences that occur for \(G\). Assemble all \(l\in\lie{g}^*\!\setminus\!\{0\}\) with the same sequence to
\[ \Omega_d = \{ l\in \lie{g}^*\!\setminus\!\{0\}\mid d(l) = d\}\] for \(d\in D\). The sets \(\Omega_d\) are \(G\)\nb-invariant  because the projections \(p_i\) are equivariant. As \(\lie{g}_l = \lie{g}_{\lambda\cdot l}\) for all \(\lambda\in\Rp\), \cref{jump} implies that they are also invariant under the dilation action. For \(d=(d_1,\ldots,d_n)\in D\) set \(d_{n+1}=0\) and define
\begin{align*}
S(d)&= \{i\in\{1,\ldots,n\}\mid d_{i} = d_{i+1}+ 1\},\\
T(d)&= \{i\in\{1,\ldots,n\}\mid d_{i}=d_{i+1}\},\\
\lie{g}^*_{S(d)} &= \Span\{X_i^*\mid i\in S(d)\},\\
\lie{g}^*_{T(d)} &= \Span\{X_i^*\mid i\in T(d)\}.
\end{align*}
The following theorem is due to Puk\'{a}nszky \cite{pukanszky1971unitary} and is also proved in \cite{corwin1990representations}.
\begin{theorem}[\cite{corwin1990representations}*{3.1.14}]\label{thm:ordering}
	There is an ordering \(d_1\geq d_2\geq \ldots\geq d_m\) of \(D\) such that
	all \(W_{i} = \bigcup_{d\geq d_i} \Omega_{d}\) for \(i=1,\ldots,m\) are \(G\)- and \(\Rp\)\nb-invariant and open. Each \(G\)-orbit in \(\Omega_d\) meets \(\lie{g
	}^*_{T(d)}\) in exactly one point. 
\end{theorem}
This allows to find a sequence as in \eqref{sets} using Kirillov's map by setting \(V_i=G\backslash W_i\) for \(i=1,\ldots,m\). It remains to check that the \(V_i\setminus V_{i-1}=G\backslash \Omega_{d_i}\) are Hausdorff and that the corresponding \(\Rp\)-action is free and proper.  
\begin{proposition}\label{res:freeproper}
	For \(d\in D\) let \(\Lambda_d \defeq \Omega_d\cap \lie{g
	}^*_{T(d)}\). The map \(\Lambda_d\to G\backslash \Omega_d\) induced by the inclusion is an \(\Rp\)\nb-equivariant homeomorphism. The corresponding \(\Rp\)\nb-action on \(\Lambda_d\) is free and proper.
\end{proposition}
\begin{proof} In \cite{corwin1990representations}*{3.1.14} it is proved that there is a birational, nonsingular map \(\psi_d\colon \Lambda_d\times\lie{g
	}^*_{S(d)}\to\Omega_d\). Furthermore, \(\pi_1\circ\psi_d^{-1}\) is \(G\)\nb-invariant, where \(\pi_1\) denotes the projection to \(\Lambda_d\). Hence, it induces a continuous map \(G\backslash\Omega_d\to \Lambda_d\). It is inverse to the map induced by the inclusion. Thus, the two spaces are homeomorphic.	
	As \(\Omega_d\) and \(\lie{g
	}^*_{T(d)}\) are invariant under the dilation action, so is \(\Lambda_d\). Therefore, the inclusion is equivariant. 
	Since \(0\in \lie{g
	}^*\) is not contained in any \(\Omega_d\), the \(\Lambda_d\) are subsets of some \(\R^l\!\setminus\!\{0\}\) equipped with the Euclidean subspace topology. Hence they are Hausdorff and the \(\Rp\)\nb-action, which is given for \(\lambda>0\) by multiplying the coordinate entries by different powers of \(\lambda\), is free and proper.
\end{proof}

\begin{example}
	From the computations for the Heisenberg group in \cref{heisenberg} we get as in \cite{corwin1990representations}*{3.1.15}, up to the reversed order,
	\begin{align*}
	\Omega_{(2,1,0)}&=\{\alpha X^*+\beta Y^* + \gamma Z^*\mid \alpha,\beta\in\R \text{ and }\gamma\neq 0\},\\
	\Omega_{(0,0,0)}&=\{\alpha X^*+\beta Y^*\mid (\alpha,\beta)\neq(0,0)\},\\
	T(2,1,0)&=\{3\},\\
	T(0,0,0)&=\{1,2,3\},\\
	\Lambda_{(2,1,0)}&=\{\gamma Z^*\mid \gamma\neq 0\},\\
	\Lambda_{(0,0,0)}&=\{\alpha X^*+\beta Y^*\mid (a,b)\neq (0,0)\}.
	\end{align*}
	Therefore, the desired sequence is \(\emptyset\subset G\backslash\Omega_{(2,1,0)}\subset\widehat{G}\setminus \{\pi_{\triv}\}\).
	The dilation action on \(\Lambda_{(2,1,0)}\cong\R\setminus\{0\}\) is multiplication with \(\lambda^{-2}\) for \(\lambda>0\), whereas on \(\Lambda_{(0,0,0)}\cong\R^2\setminus\{0\}\) it is scalar multiplication with \(\lambda^{-1}\).
\end{example}
\subsection{Plancherel theory}\label{sec:plancherel}
For a locally compact group \(G\), the operator-valued Fourier transform
\(f\mapsto\widehat{f}\) is defined for \(f\in L^1(G)\). It is given by
\begin{align*}
\widehat{f}(\pi)=\int_G f(x)\pi(x)\diff x\in\Bound(\Hils_\pi)
\end{align*}  
for an irreducible, unitary representation \(\pi\colon G\to \mathcal{U}(\Hils_\pi)\). We recall some results from Plancherel theory for locally compact, separable groups \(G\) of type~I  (see \cite{dixmier}*{18.8}, \cite{corwin1990representations}*{4.3} and \cite{fischer2016quantization}*{Appendix~B}). 

In this case, the topological space \(\widehat{G}\) can be equipped with a certain Borel measure \(\mu\), which is called the \emph{Plancherel measure}. For a simply connected nilpotent Lie group \(G\), it is supported within the orbits with maximal dimension sequence, these are the orbits in \(\Omega_{d_1}\subset \lie{g}^*\) with notation as in \cref{thm:ordering} (see \cite{corwin1990representations}*{p.~154}).

By \cite{dixmier}*{8.6.1} there is a subspace \(\Gamma\subset \prod_{\pi\in \widehat{G}}\Hils_\pi\) that turns \(((\Hils_\pi)_{\pi\in\widehat{G}},\Gamma)\) into a \emph{\(\mu\)-measurable field of Hilbert spaces over \(\widehat{G}\)} as defined in \cite{fischer2016quantization}*{Def.~B.1.3}. The elements of \(\Gamma\) are called the \emph{measurable sections}. 

For a Hilbert space \(\Hils\) denote by \(\mathrm{HS}(\Hils)\) the Hilbert space of Hilbert--Schmidt operators on \(\Hils\). Identifying \(\Hils\otimes\Hils^*\) with \(\mathrm{HS}(\Hils)\), one obtains for a simply connected nilpotent Lie group \(G\) the \(\mu\)-measurable field \(((\mathrm{HS}(\Hils_\pi))_{\pi\in\widehat{G}},\Gamma\otimes\Gamma^*)\) over \(\widehat{G}\) (see \cite{fischer2016quantization}*{B.1.3}). 

Define \(L^2(\widehat{G},\mathrm{HS}(\Hils_\pi))\) to be the direct integral of the Hilbert spaces \(\mathrm{HS}(\Hils_\pi)\). It consists of sections \(x\in\Gamma\otimes\Gamma^*\) such that \((\pi\mapsto \norm{x(\pi)})\in L^2(\widehat{G},\mu)\). It is a Hilbert space with respect to	\(\langle x,y\rangle\defeq \int_{\widehat{G}}{\langle x(\pi),y(\pi)\rangle_{\mathrm{HS}}\diff\mu(\pi)}\). 

The Plancherel Theorem \cite{fischer2016quantization}*{B.2.32} states that the Fourier transform defined above yields an isometric isomorphism \[\widehat{\,}\,\colon L^2(G)\to L^2(\widehat{G},\mathrm{HS}(\Hils_\pi)).\] 

The left group von Neumann algebra \(\VN_L(G)\) of \(G\) consists of bounded, left-invariant operators on \(L^2(G)\) (see \cite{fischer2016quantization}*{B.2.13}). The Plancherel Theorem yields that the Fourier transform extends to a \(^*\)-isomorphism
\begin{align}\label{eq:Fourier-VN-Linfty}
\widehat{\,}\,\colon \VN_L(G)\to  L^\infty(\widehat{G},\Bound(\Hils_\pi)).
\end{align}
Here, \(L^\infty(\widehat{G}, \Bound(\Hils_\pi))\) consists of \(a=(a(\pi))_{\pi\in\widehat{G}}\) with \(a(\pi)\in\Bound(\Hils_\pi)\) such that 
\begin{enumerate}
	\item \((a(\pi)x(\pi))_\pi\in \Gamma\) for all \((x(\pi))_\pi\in\Gamma\),
	\item \((\pi\mapsto\norm{a(\pi)})\in L^\infty(\widehat{G},\mu)\). 
\end{enumerate}
It is a von Neumann algebra with the pointwise operations and the norm given by
\[\norm{a}= \sup_{\pi\in \widehat{G}}{\norm{a(\pi)}_{\Bound(\Hils_\pi)}}.\]
The inverse Fourier transform maps a field \(a\in L^\infty(\widehat{G},\Bound(\Hils_\pi))\) to the operator \(T_a\in \VN_L(G)\) determined by
\[\widehat{T_a\psi}(\pi)=a(\pi)\widehat{\psi}(\pi) \qquad \text{for }\psi\in L^2(G)\text{, }\pi\in\widehat{G}.\] 
This operator-valued Fourier transform is essential for the definition of the symbolic pseudo-differential calculus of Fischer--Ruzhansky described in \cref{sec:nilpotent_calc}. 

\section{The tangent groupoid and its \(\Cst\)-algebra}\label{sec:tangent}
In this section, the tangent groupoid of a homogeneous Lie group \(G\) is defined as the transformation groupoid of an action of \(G\). 
We explain how this groupoid can be understood as a variant of Connes' tangent groupoid \cite{connes}. The homogeneous structure is taken into account by replacing addition of tangent vectors by multiplication in the group. Furthermore, the groupoid \(\Cst\)\nb-algebra of the tangent groupoid can be described as a continuous field of \(\Cst\)\nb-algebras. 
\subsection{The tangent groupoid}
\begin{definition}
	For a homogeneous Lie group \(G\) the \emph{tangent groupoid} is the continuous action groupoid
	\[\grpd= (G\times[0,\infty))\rtimes G\]
	of the action \((G\times[0,\infty))\curvearrowleft G\) given by \((x,t)\cdot v=(x\alpha_t(v),t)\). Here, \(\alpha_t\) for \(t>0\) are the dilations on \(G\) and \(\alpha_0(v)=\lim_{t\to 0}\alpha_t(v)=0\) for all \(v\in G\).
	
	The \emph{unit map} \(u\colon \grpd^0\defeq G\times[0,\infty)\to \grpd\), the \emph{range} and \emph{source map} \(r,s\colon \grpd \to  \grpd^0\), the \emph{inverse} and the \emph{multiplication} are given by
	\begin{align*}
	u(x,t)=(x,t,0), \qquad r(x,t,v)=(x,t), \qquad s(x,t,v)=(x\alpha_t(v),t),\\
	(x,t,v)^{-1}=(x\alpha_t(v),t,v^{-1}) ,\qquad (x,t,v)\cdot(x\alpha_t(v),t,w)=(x,t,vw),
	\end{align*}
	for \(x,v,w\in G\) and \(t\in [0,\infty)\).
	The \emph{range fibre} of \(\grpd\) over \((x,t)\in \grpd^0\) is 
	\begin{align*}
	\grpd^{(x,t)}&=r^{-1}(x,t)= \left\{(x,t,v)\in\grpd\mid v\in G \right\}.
	\end{align*}
\end{definition}
Let \(\theta\colon \grpd\to[0,\infty)\) denote the projection to the second coordinate. Recall that the pair groupoid of \(G\) is the groupoid \(G\times G\) with unit space \(G\) with \(r(x,y)=x\), \(s(x,y)=y\), \((x,y)^{-1}=(y,x)\) and \((x,y)\cdot(y,z)=(x,z)\) for \(x,y,z\in G\).
\begin{lemma}\label{continuousfield}
	Let \(G\) be a homogeneous Lie group. Then \((\grpd,[0,\infty),\theta)\) defines a continuous field of locally compact, amenable groupoids. The subgroupoids \(\theta^{-1}\{t\}\) for \(t>0\) are isomorphic to the pair groupoid of \(G\). The subgroupoid 
	\(TG\defeq \theta^{-1}\{0\}\) is the trivial field of groups over \(G\) with fibre~\(G\). 
\end{lemma}

\begin{proof}
	It is easy to check that \((\grpd,[0,\infty),\theta)\) defines a continuous field of locally compact groupoids in the sense of \cite{landsman} or \cite{beckus}. All subgroupoids \(\theta^{-1}\{t\}\) for \(t\geq 0\) are transformation groupoids of actions of \(G\) on itself. The group \(G\) is amenable as a nilpotent group, for that reason all \(\theta^{-1}\{t\}\) are amenable. 
	
	For all \(t>0\) the map \(\varphi_t(x,v)=(x,x\alpha_t(v))\) defines a groupoid isomorphism 
	\begin{equation}\label{pairiso}
	\varphi_t\colon \theta^{-1}\{t\} \to G\times G.
	\end{equation}  
	Its inverse is given by \((x,y)\mapsto (x,\alpha_{t^{-1}}(x^{-1}y))\). 
	
	The subgroupoid \(TG=\theta^{-1}\{0\}\) corresponds to a (noncommutative) version of the tangent bundle. For \((x,v)\in TG\) we interpret \(x\) as the base point and \(v\in G\cong \lieg\cong T_xG\) as a tangent vector at \(x\).  The groupoid multiplication is defined if and only if two vectors lie in the same fibre and is, in this case, defined by the group multiplication. Let \(q\colon TG\to G\) denote the projection to the base point, then \((TG,G,q)\) defines itself a continuous field of locally compact groupoids. It is the trivial field over \(G\) with fibre \(G\). Again, all fibres \(q^{-1}\{x\}\cong G\) for \(x\in G\) are amenable. 
\end{proof}

\begin{remark}If \(G\) is a graded nilpotent Lie group, the action \(G\times[0,\infty)\curvearrowleft G\) is smooth. Consequently, the tangent groupoid \(\grpd\) is a Lie groupoid. 
	A graded nilpotent Lie group is a special case of a filtered manifold as considered in \cite{erp2017tangent}. Therefore, one can define the tangent groupoid
	\[\T G = (TG\times \{0\}\cup (G\times G)\times (0,\infty)\rightrightarrows G\times [0,\infty))\]
	as a smooth field of groupoids over \([0,\infty)\) as in \cites{choi2019tangent,erp2017tangent,haj-higson}. Using the isomorphisms \(\varphi_t\colon \theta^{-1}\{t\}\to G\times G\) from \eqref{pairiso} and the definition of the smooth structure of \(\T G\), one obtains that~\(\grpd\) and \(\T G\) are isomorphic as Lie groupoids. 
\end{remark}
\subsection{The groupoid \(\Cst\)-algebra}
Now, we recall how the \emph{groupoid \(\Cst\)\nb-algebra} of the tangent groupoid \(\grpd\) is constructed. As the tangent groupoid of \(G\) is an action groupoid of an amenable group, \(\Cst(\grpd)\) is isomorphic to the reduced crossed product \(\Cred(G,\Cont_0(G\times[0,\infty))\) as remarked in \cite{renaultgroupoid}. Here, the left \(G\) action on \(\Cont_0(G\times[0,\infty))\) is given by
\[(v\cdot\psi)(x,t)=\psi((x,t)\cdot v)=\psi(x\alpha_t(v),t)\]
for \(\psi\in\Cont_0(G\times[0,\infty))\), \(v,x\in G\) and \(t\geq 0\).
For \(f,g\in\Cont_c(\grpd)\), viewed as elements of \(\Cont_c(G,\Cont_0(G\times[0,\infty))\) the involution and convolution defined in \eqref{conv} and \eqref{inv} are 
\begin{align*}
f^*(x,t,v)&=\conj{f(x\alpha_t(v),t,v^{-1})},\\
(f*g)(x,t,v)&=\int_G f(x,t,w)g(x\alpha_t(w),t,w^{-1}v)\diff w
\end{align*}
for \((x,t,v)\in \grpd\). The \(\norm{\,\cdot\,}_I\)\nb-norm is given by \(\norm{f}_I=\max\{\norm{f}_{I,1},\norm{f}_{I,2}\}\), where
\begin{align*}
\norm{f}_{I,1}&= \sup_{(x,t)}{\int_G \abs{f(x,t,v)}\diff v}
\end{align*}
and \(\norm{f}_{I,2}=\norm{f^*}_{I,1}\).
The groupoid \(\Cst\)\nb-algebra \(\Cst(\grpd)\) is the closure of \(\Cont_c(\grpd)\) under the representation \(\rho\) as in \eqref{crossedprod}. In particular, the \(\Cst\)-norm of \(f\in\Cont_c(\grpd)\) is bounded by \(\norm{f}_I\).

 \begin{lemma}The continuous field of groupoids \((\grpd,[0,\infty),\theta)\) gives rise to a continuous field of \(\Cst\)\nb-algebras \(\Cst(\grpd)\) over \([0,\infty)\) with fibres isomorphic to \(\Comp(L^2G)\) for \(t>0\) and \(\Cst(TG)\) at \(t=0\). 
\end{lemma}
\begin{proof}
	As all groupoids \(\theta^{-1}\{t\}\) are amenable, \(\Cst(\grpd)\) defines a continuous field of \(\Cst\)\nb-algebras with fibres \(\Cst(\theta^{-1}\{t\})\) by \cite{landsman}*{5.6}. 	
	 The groupoid \(\theta^{-1}\{t\}\) for \(t>0\) is isomorphic to the pair groupoid \(G\times G\) by \cref{continuousfield}. The Haar measure on \(\theta^{-1}\{t\}\) is taken under \(\varphi_t\) to the left Haar measure \(\{\mu_t^x\}_{x\in G}\) on \(G\times G\) with
	\[\int K(\gamma)\diff\mu_t^x(\gamma)=t^{-Q}\int_GK(x,y)\diff y \qquad\text{for }K\in\Cont_c(G\times G).\]
	There is a well-known isomorphism \(\beta_t\colon\Cst(G\times G,\mu_t)\to \Comp(L^2G)\) with
	\[(\beta_t(K)\psi)(x)=t^{-Q}\int_G K(x,y)\psi(y)\diff y\]
	for \(K\in\Cont_c(G\times G)\) and \(\psi\in \Cont_c(G)\).	
	For \(t>0\) compose \(\beta_t\) and the homomorphism induced by \(\varphi_t^{-1}\) to \(p_t\colon \Cst(\grpd)\to \Comp(L^2G)\) given by
	\begin{align}\label{eq:p_t}
	\left(p_t(f)\psi\right)(x)= t^{-Q}\int_G f\left(x,t,\alpha_{t^{-1}}(x^{-1}y)\right)\psi(y)\diff y
	\end{align} 
	for \(f\in\Cont_c(\grpd)\), \(\psi\in \Cont_c(G)\) and \(x\in G\). It restricts to an isomorphism between \(\Cst(\theta^{-1}\{t\})\) and \(\Comp(L^2G)\). 
\end{proof}
\begin{lemma}\label{res:kernelev}
	The subset \(G\times(0,\infty)\subset \grpd^0\) is open and invariant. Let \(\grpd^{G\times(0,\infty)}=r^{-1}(G\times(0,\infty))\). There is an isomorphism \(p\colon\Cst(\grpd^{G\times(0,\infty)})\to\Cont_0(\Rp,\Comp(L^2G))\) given by \(p(f)(t)=p_t(f)\) for \(f\in\Cont_c(\grpd^{G\times(0,\infty)})\).
\end{lemma}
\begin{proof}
	The subgroupoid \(\grpd^{G\times(0,\infty)}\) is isomorphic to the trivial field of groupoids over \(\Rp\) with the pair groupoid \(G\times G\) as fibre via \((x,t,v)\mapsto(t,\varphi_t(x,v))\). Composing the induced isomorphism of the corresponding groupoid \(\Cst\)-algebras with the respective \(\beta_t\) for \(t>0\) gives the claim. 
\end{proof}
Denote by \(p_0\colon\Cst(\grpd)\to\Cst(TG)\) the \(^*\)\nb-homomorphism induced by  restricting to \(t=0\). There is a short exact sequence by \cites{hilsumskandalis}
\begin{equation}\label{tangentgroupoidext}
\begin{tikzcd}
\Cont_0(\Rp,\Comp(L^2G)) \arrow[r,hook] & \Cst(\grpd)\arrow[r,twoheadrightarrow,"p_0"] & \Cst(TG).
\end{tikzcd}
\end{equation}
If \(G=\R^n\), the \(\Cst\)\nb-algebra on the right is isomorphic to \(\Cont_0(T^*\R^n)\) via Fourier transform. In general, \(\Cst(TG)\) can be noncommutative. It is the trivial field of \(\Cst\)-algebras over \(G\) with fibres isomorphic to the group \(\Cst\)\nb-algebra \(\Cst(G)\). 
\section{Pseudo-differential extension using generalized fixed point algebras}\label{sec:pdo_using_fix}
In this section we use the dilations on \(G\) to define a certain \(\Rp\)-action on \(\Cst(\grpd)\). We show that the generalized fixed point algebra construction can be applied when the action is restricted to an ideal \(J\idealin\Cst(\grpd)\). In particular, we prove the existence of a continuously square-integrable subset in~\(J\). Moreover, we obtain the pseudo-differential extension. 

\subsection{The generalized fixed point algebra of the zoom action}
In the Euclidean case, the principal symbol of a pseudo-differential operator of order zero can be understood as a generalized fixed point of the proper action of \(\Rp\) on \(T^*\R^n\!\setminus\! (\{0\}\times \R^n)\) given by scaling in the fibres, that is, \(\lambda\cdot (x,\xi)=(x,\lambda^{-1}\xi)\) for~\(x\in\R^n\), \(\xi\in T_x^*\R^n\) and \(\lambda>0\). 
Under Fourier transform this action corresponds to \((\sigma_\lambda f)(x,X)= \lambda^nf(x,\lambda X)\) for \(x\in\R^n\), \(X\in T_x\R^n\), \(f\in\Cont_c(T\R^n)\) and \(\lambda>0\). Using the dilations we can define an analogous action on \(\Cst(TG)\) for a homogeneous Lie group \(G\) and extend it to \(\Cst(\grpd)\).

\begin{lemma}
For a homogeneous Lie group \(G\) of homogeneous dimension \(Q\) the maps \(\sigma_\lambda\colon\Cont_c(\grpd)\to\Cont_c(\grpd)\) defined by \((\sigma_\lambda f)(x,t,v)=\lambda^Qf(x,\tfrac{t}{\lambda},\alpha_\lambda(v))\) for \(\lambda>0\) and \(f\in\Cont_c(\grpd)\) extend to a strongly continuous \(\Rp\)\nb-action on \(\Cst(\grpd)\).
\end{lemma}

\begin{proof}
It is easy to check that \(\sigma_\lambda\) are linear maps satisfying \(\sigma_\lambda(f*g)=(\sigma_\lambda f)*(\sigma_\lambda g)\) and \(\sigma_\lambda(f^*)=(\sigma_\lambda f)^*\) for all \(f,g\in\Cont_c(\grpd)\) and \(\lambda>0\). Moreover, \(\sigma_1=\Id\) and \(\sigma_{\lambda\mu}= \sigma_{\lambda}\circ\sigma_{\mu}\) hold for all \(\lambda,\mu>0\). Each \(\sigma_\lambda\) is an isometry with respect to the \(I\)\nb-norm and, therefore, extends to an automorphism of \(\Cst(\grpd)\).
\end{proof}	

\begin{remark}\label{actiononideal}
Let \(\tau\colon \Rp\curvearrowright \Cont_0(\Rp)\) be given by \(\tau_\lambda(f)(t)=f(\lambda^{-1}t)\) for \(\lambda,t>0\) and \(f\in\Cont_0(\Rp)\).
The identity \(p_t\,\circ\,\sigma_\lambda=p_{t\lambda^{-1}}\) for all \(t,\lambda>0\) shows that the restriction of \(\sigma\) to the invariant ideal \(\Cst(\grpd^{G\times(0,\infty)})\) is mapped to the action \(\tau\otimes 1\colon H\curvearrowright \Cont_0(\Rp)\otimes \Comp(L^2G)\) under the isomorphism \(p\) from \cref{res:kernelev}. 
In particular, \(\sigma\) corresponds to the zoom action of \(\Rp\) on \(\T G\) defined in \cite{erp2015groupoid}*{Def.~17} or \cite{haj-higson}*{10.6}.
\end{remark}

As described above, in the Euclidean case the scaling action on \(T^*\R^n\) is restricted to \(T^*\R^n\!\setminus\! (\{0\}\times \R^n)\). This is necessary as the zero section is fixed by the scaling action, hence, the action on \(T^*\R^n\) is not proper. For an arbitrary homogeneous Lie group, we must also choose an ideal inside the \(\Cst\)-algebra of the tangent groupoid in order to obtain a continuously square-integrable \(\Rp\)-\(\Cst\)-algebra. For \(f\in\Schwartz(\R^n)\) the property \(f(0)=0\) is equivalent to \(\int \widehat{f}(x)\diff x=0\), where \(\widehat{f}\) is the Fourier transform of \(f\). Moreover, \(\Cont_0(\R^n\!\setminus\!\{0\})\) corresponds under Fourier transform to \(\ker(\widehat{\pi}_{\triv})\idealin\Cst(\R^n)\) where \(\pi_{\triv}\) is the trivial representation of \(\R^n\).

For a homogeneous Lie group \(G\) let \(q_x\colon \Cst(TG)\to\Cst(G)\) for \(x\in G\) be the \(^*\)-homomorphism induced by restricting \(f\in\Cont_c(TG)\) to the fibre \(T_xG\).

\begin{definition}
	Let \(G\) be a homogeneous Lie group and \(\pi_{\triv}\) its trivial representation. Let \(J\) be the closed ideal in
	\(\Cst(\grpd)\) defined by
	\begin{align*}
	J=\bigcap_{x\in G} {\ker{(\widehat{\pi}_{\triv}\circ q_x\circ p_0)}}.
	\end{align*}
\end{definition}
The ideal \(J\) is invariant under the action \(\sigma\colon\Rp\acts\Cst(\grpd)\) above.
Now, we define a linear subspace \(\Rel_\grpd\subset J\) for the generalized fixed point algebra construction.

\begin{definition}
	Let \(\Rel_\grpd\) consist of \(f\in \Cont^\infty(\grpd)\) satisfying the following conditions:
	\begin{enumerate}[(a)]
		\item \label{compactness} \(r(\supp f)\subset \grpd^0\) is compact,
		\item \label{schwartz}\((x,t)\mapsto f|_{\grpd^{(x,t)}}\) and \((x,t)\mapsto (\partial_tf)|_{\grpd^{(x,t)}}\) are continuous maps \(\grpd^0\to\Schwartz(G)\),
		\item \label{vanishing2}\(\int_G f(x,0,v)\diff v = 0\) for all \(x\in G\).
	\end{enumerate}
Using the seminorms from \cref{seminorm}, set for \(N\in\N_0\)
\begin{equation*}
\norm{f}_{(N)}\defeq\sup_{(x,t)\in\grpd^0}{\norm{f|_{\grpd^{(x,t)}}}_{N}} =\sup_{(x,t,v)\in\grpd, \abs{I}\leq N}{(1+\norm{v})^{(N+1)(Q+1)}\abs{\partial_v^If(x,t,v)}}.
\end{equation*}
\end{definition}

For \(f\in\Rel_\grpd\) conditions \ref{compactness} and \ref{schwartz} ensure that \(\norm{f}_{(N)}<\infty\) and \(\norm{\partial_tf}_{(N)}<\infty\) for all \(N\in\N\). Note that \(\norm{f^*}_{(0)}=\norm{f}_{(0)}\) holds. Hence, by \cref{res:convint} \(\norm{f}_{I}\leq D \norm{f}_{(0)}\) holds for a constant \(D>0\), so that the elements of \(\Rel_\grpd\) lie indeed in the groupoid \(\Cst\)-algebra \(\Cst(\grpd)\). Condition \ref{vanishing2} forces them to be in the ideal \(J\). 
The goal of this section is to show that the generalized fixed point construction can be applied to the \(\Rp\)\nb-action \(\sigma\) on \((J,\cl{\Rel_\grpd})\). 

\begin{lemma}\label{res:meanvalue}
			 For every~\(k\in\N\) there is a group constant \(C_k>0\) and such that for all \(g\in\Cont^\infty(\grpd)\) with \(\norm{g}_{(k)}<\infty\) and all \(x,v,w\in G\) and~\(t\geq 0\)
		\[\abs{g\left((x,t,v)(x\alpha_t(v),t,w)\right)-g(x,t,v)}\leq C_k\norm{g}_{(k)}\frac{(1+\norm{w})^{(k+2)(Q+1)}}{(1+\norm{v})^{k(Q+1)}}\sum_{j=1}^n\norm{w}^{q_j}.\]
\end{lemma}
\begin{proof}
	It suffices to show the claim for real-valued  \(g\). Define a function \(H\colon G\times[0,\infty)\times G \times G\times[0,1]\to \R\) by
	\[H(x,t,v,w,h)=g((x,t,v)\cdot(x\alpha_t(v),t,\alpha_h(w)).\]
	Hence, we obtain
	\begin{align*}
	g((x,t,v)(x\alpha_t(v),t,w))-g(x,t,v)=\int_0^1 \partial_hH(x,t,v,w,s)\diff s.
	\end{align*}
	To estimate \(\abs{\partial_hH}\) note that \(H=g\circ (\Id_{G\times[0,\infty)}\times m)\circ(\Id_{G\times[0,\infty)\times G}\times\alpha)\), where \(m\colon G\times G\to G\) denotes the group multiplication and \(\alpha(w,h)=\alpha_h(w)\) for \(w\in G\) and \(h\in[0,1]\). One calculates
	\[\partial_hH(x,t,v,w,s)=\sum_{i,j=1}^n (\partial_{v_i}g)(x,t,v\alpha_s(w))\cdot \partial_{n+j}m_{i}(v,\alpha_s(w))\cdot \partial_h\alpha_j(w,s).\]
	By the polynomial form of the group law there is a group constant \(\tilde{D}>0\) such that
	\[ \abs{\partial_{n+j}m_{i}(v,\alpha_s(w))}\leq \tilde{D}(1+\norm{v})^{Q}(1+\norm{\alpha_s(w)})^{Q}\leq \tilde{D}(1+\norm{v})^{Q+1}(1+\norm{w})^{Q+1}\]
	for all \(v,w\in G\) and \(s\in[0,1]\). Using \cref{triangle} we estimate
	\begin{align*}
	\abs{(\partial_{v_i}g)(x,t,v\alpha_s(w))}&\leq \norm{g}_{(k)} (1+\norm{v\alpha_s(w)})^{-(k+1)(Q+1)}
	\\ &\leq \norm{g}_{(k)}C^{(k+1)(Q+1)} \frac{(1+\norm{\alpha_s(w)})^{(k+1)(Q+1)}}{(1+\norm{v})^{(k+1)(Q+1)}}\\
	&\leq \norm{g}_{(k)}C^{(k+1)(Q+1)} \frac{(1+\norm{w})^{(k+1)(Q+1)}}{(1+\norm{v})^{(k+1)(Q+1)}}.\end{align*}
	As \(\alpha_j(w,h)=h^{q_j}w_j\), it follows that \(\abs{\partial_h\alpha_j(w,s)}\leq Q \abs{w_j}\leq Q \norm{w}^{q_j}\). Together, these estimates imply the claim. 
\end{proof}
\begin{lemma}\label{res:mainestimate}
	Consider the action \(\sigma\colon \Rp\acts J\). For \(f\in\Rel_\grpd^*\) the operator \(\BRA{f}\), defined as in \eqref{BRAKET}, satisfies 
	\(\BRA{f}g\in L^1(\Rp,J)\) for all \(g\in\Rel_\grpd^*\). 
\end{lemma}
\begin{proof}
	Let \(f,g\in\Rel_\grpd\). We show that \(\BRA{f^*}g^*\in L^1(\Rp,J)\). Because  \(\sigma_\lambda\) is an isometry with respect to the \(I\)\nb-norm for \(\lambda>0\), the following equality holds
	\begin{equation}\label{symmetry}
	\norm{\sigma_{\lambda^{-1}}(f)*g^*}_I = \norm{f*\sigma_\lambda(g^*)}_I=\norm{\sigma_\lambda(g)*f^*}_I.
	\end{equation}
	Hence, it suffices to prove for all \(f,g\in\Rel_\grpd\) that
	\begin{equation}\label{toshow} 
	\int_1^\infty \norm{\sigma_\lambda(f)*g^*}_I\tfrac{\diff \lambda}{\lambda}<\infty.\end{equation}
	So let \(\lambda\geq 1\) in the following. Using the homogeneity of the Haar measure, we compute
	\begin{align*}
	\norm{\sigma_\lambda(f)*g^*}_{I,1}&= \sup_{(x,t)}{\int_G \abs{(\sigma_\lambda(f)*g^*)(x,t,v)}\diff v} \\
	&= \sup_{(x,t)}{ \int_G\Bigl|\int_G f\left(x,\tfrac{t}{\lambda},w\right)\conj{g\left((x,t,v)^{-1}(x,t,\alpha_{\lambda^{-1}}(w))\right)}\diff w\Bigr|\diff v}.
	\end{align*}
	To estimate this integral, let
	\begin{align*}R_1(x,t,v)&\defeq f(x,t,v)-f(x,0,v)&&\text{for }x,v\in G\text{, }t\geq0,\\
	R_2(\gamma,\eta)& \defeq g(\gamma\cdot\eta)-g(\gamma) &&\text{for }(\gamma,\eta)\in\grpd^{(2)}.
	\end{align*}
	As \(f\) satisfies condition \ref{vanishing2}, we get
	\begin{equation*}
	\begin{split}
	\norm{\sigma_\lambda(f)*g^*}_{I,1} \leq  \sup_{(x,t)} {\Bigl(\int_G\int_G \bigl|R_1\left(x,\tfrac{t}{\lambda},w\right)g\left(x\alpha_{t}(v),t,v^{-1}\alpha_{\lambda^{-1}}(w)\right)\bigr|\diff w\diff v}\\
	 + \int_G\int_G \bigl| f\left(x,0,w\right)R_2\left((x,t,v)^{-1},(x,t,\alpha_{\lambda^{-1}}(w))\right)\bigr|\diff w\diff v\Bigr).
	\end{split}
	\end{equation*}
	We start by estimating the first summand. As \(\partial_tf\) is rapidly decaying, one has \(\norm{R_1(x,t,v)}\leq t\norm{\partial_tf}_{(1)}(1+\norm{v})^{-2(Q+1)}\) for all \((x,t,v)\in\grpd\). Let \(T>0\) be such that \(f(x,t,v)=0\) holds whenever~\(t>T\). Using \cref{triangle} and \(\lambda\geq 1\), we find
	\begin{align*}
	& \int_G\int_G \abs{R_1\left(x,\tfrac{t}{\lambda},w\right)g\left(x\alpha_{t}(v),t,v^{-1}\alpha_{\lambda^{-1}}(w)\right)}\diff w\diff v \\ \leq&\,  \lambda^{-1}T\norm{\partial_tf}_{(1)}\norm{g}_{(0)} \int_G\int_G \frac{1}{(1+\norm{w})^{2(Q+1)}}\cdot\frac{1}{(1+\norm{v^{-1}\alpha_{\lambda^{-1}}(w)})^{Q+1}}\diff w\diff v\\
	\leq & \, \lambda^{-1}T\norm{\partial_tf}_{(1)}\norm{g}_{(0)}C^{Q+1}\int_G\int_G \frac{1}{(1+\norm{w})^{2(Q+1)}}\cdot\frac{(1+\lambda^{-1}\norm{w})^{Q+1}}{(1+\norm{v})^{Q+1}}\diff w\diff v\\
	\leq & \,  \lambda^{-1}T\norm{\partial_tf}_{(1)}\norm{g}_{(0)}C^{Q+1}\left(\int_G\frac{1}{(1+\norm{v})^{Q+1}}\diff v\right)^2.
	\end{align*}
	This integral converges by \cref{res:convint}. The estimate is independent of \((x,t)\in\grpd^0\). 	
	For the second summand, apply \cref{res:meanvalue} to \(g\) with \(k=1\). We obtain
	\begin{align*}
	\abs{R_2((x,t,v)^{-1},(x,t,\alpha_{\lambda^{-1}}(w)))}&\leq  C_{1}\norm{g}_{(1)}\,\frac{(1+\norm{\alpha_{\lambda^{-1}}(w)})^{3(Q+1)}}{(1+\norm{v^{-1}})^{Q+1}}\sum_{j=1}^{n}\norm{\alpha_{\lambda^{-1}}(w)}^{q_i}\\
	&= C_{1}\norm{g}_{(1)}\,\frac{(1+\lambda^{-1}\norm{ w})^{3(Q+1)}}{(1+\norm{v})^{Q+1}}\sum_{j=1}^{n}\lambda^{-q_i}\norm{w}^{q_i}\\
	&\leq nC_{1}\norm{g}_{(1)}\lambda^{-1}\, \frac{(1+\norm{ w})^{4(Q+1)}}{(1+\norm{v})^{Q+1}}
	\end{align*}
	Consequently, we obtain using the rapid decay of \(f\) that
	\begin{align*}&\int_G\int_G \bigl| f\left(x,0,w\right)R_2\left((x,t,v)^{-1},(x,t,\alpha_{\lambda^{-1}}(w))\right)\bigr|\diff w\diff v\\ 
	\leq & \lambda^{-1}C_{1}\norm{f}_{(5)}\norm{g}_{(1)}\int_G\int_G (1+\norm{w})^{-5(Q+1)}\frac{(1+\norm{ w})^{4(Q+1)}}{(1+\norm{v})^{Q+1}}\diff w\diff v
	\end{align*}
	Again, this converges by \cref{res:convint} and the estimate does not depend on \((x,t)\in\grpd^0\).  For \(\norm{\sigma_\lambda(f)*g^*}_{I,2}\), one can estimate analogously by replacing \((x,v,t)\) by \((x,v,t)^{-1}\) in the argument above. 
	Thus, the convergence of \(\int_1^\infty \lambda^{-2}\diff \lambda\) implies~\eqref{toshow}. 
	
	Moreover, together with the respective estimate for \(\lambda<1\) using \eqref{symmetry}, we obtain a constant \(\tilde{D}>0\) such that for all \(f,g\in\Rel_\grpd\) that vanish for \(t>T\)
	\begin{align}\label{normestimate}
	\norm{\BRAKET{f^*}{g^*}}\leq \norm{\KET{f^*}g^*}_1\leq \tilde{D}(\norm{f}_{(5)}+T\norm{\partial_tf}_{(1)})( \norm{g}_{(5)}+T\norm{\partial_t g}_{(1)}) 
	\end{align}
	holds.
\end{proof}
\begin{definition}\label{defideal}
	Let \(\Rel\) be the \(^*\)-subalgebra of \(J\) containing all \(f\in \Cont_c^\infty(\grpd)\) with
	\begin{align}\label{vanishing}
	\int_G f(x,0,v)\diff v = 0 \qquad\text{for all } x\in G.
	\end{align}
\end{definition}

A function \(f\in\Cont_c^\infty(\grpd)\) lies in \(J\) if and only if it satisfies the vanishing integral condition \eqref{vanishing}. Note that \(\Rel\) is contained in \(\Rel_\grpd^*\).
\begin{proposition}\label{res:Jisfixable}
	Let \(G\) be a homogeneous Lie group and \(J\idealin\Cst(\grpd)\) be defined as in \cref{defideal}. 
	Denote by \(\cl{\Rel}\) the completion of \(\Rel\)  with respect to the \(\norm{\,\cdot\,}_\si\)\nb-norm. Then \((J,\cl{\Rel})\) is a continuously square-integrable \(\Rp\)\nb-\(\Cst\)\nb-algebra.
\end{proposition}

\begin{proof}
	First, we show that \(\Rel\) is dense in \(J\). Let \(f\in J\) and \(\varepsilon>0\). Then \(f\) can be approximated by \(g\in\Cont_c^\infty(\grpd)\) such that \(\norm{f-g}<\varepsilon/2\). To adjust \(g\) to lie in \(\Rel\) define a function \(h\) by
	\[ h(x)=\int_G g(x,0,v)\diff v  \qquad \text{for }x\in G.\] 
	As \(f\) lies in \(J\), it satisfies \(\abs{h(x)}=\abs{\widehat{\pi}_{\triv}(p_0(q_x(g)))-\widehat{\pi}_{\triv}(p_0(q_x(f))) }\leq \norm{f-g}<\varepsilon/2\) for all \(x\in G\). Choose a non-negative \(k\in\Cont_c^\infty(G)\) with \(\int_G k(v)\diff v=1\) and \(\omega\in \Cont_c^\infty([0,\infty))\) with \(\omega(0)=1\) and \(\norm{\omega}_\infty \leq 1\). The function \(\tilde{g}\in\Cont_c^\infty(\grpd)\) defined by \(\tilde{g}(x,t,v)=h(x)k(v)\omega(t)\) satisfies \(\norm{\tilde{g}}_I\leq \varepsilon/2\). Then \(g-\tilde{g}\in\Rel\) and \(\norm{f-(g-\tilde{g})}<\varepsilon\) holds. 
	
	As the Laurent symbol of \(\BRAKET{f}{g}\) is given by \(\BRA{f}g\), \cref{res:mainestimate} implies \(\BRAKET{f}{g}\in\Cred(\Rp,J)\). Now, by \cite{meyer2001}*{6.8} the set \(\Rel\) is square-integrable and relatively continuous. It is also \(\Rp\)\nb-invariant, and \(f*g\in\Rel\) holds for \(f,g\in\Rel\). Hence, \cref{res:completion} gives the claim. 
\end{proof}

\begin{remark}
	The action \(\sigma\colon\Rp\acts (J,\Rel)\) even satisfies Rieffel's original definition in \cite{rieffel1988}, where he requires \(\Rel\) to be a dense invariant \(^*\)\nb-subalgebra of \(J\) such that \(\lambda\mapsto f*\sigma_\lambda( g^*)\) is in \(L^1(\Rp,J)\) for all \(f,g\in\Rel\). 
\end{remark}

In the following, it will be useful to know that the \(\norm{\,\cdot\,}_{\si}\)\nb-closure of \(\Rel\) contains the space \(\Rel_\grpd^*\). In particular, this implies \(\cl{\Rel^*_\grpd}=\cl{\Rel}\).
\begin{lemma}
	The linear space \(\Rel_\grpd^*\) is contained in the completion of \(\Rel\) with respect to the \(\norm{\,\cdot\,}_{\si}\)\nb-norm.
\end{lemma}

\begin{proof}
	\cref{res:mainestimate} shows that \(\BRAKET{f}{g}\in\Cred(\Rp,J)\) for all \(f,g\in\Rel_\grpd^*\).
	Hence, by~\cite{meyer2001}*{6.8} all elements of the dense subspace \(\Rel_\grpd^*\) are square-integrable. 
	
	For \(f\in\Rel_\grpd^*\)  we construct a sequence \(g_n\in\Rel\) with \(\norm{f-g_n}_{\si}\to 0\). Let \(T>0\) be such that \(f\) vanishes for \(t>T\). Choose a sequence \((f_n)\) of \(f_n\in\Cont^\infty_c(\grpd)\) that vanish for \(t>T \), and satisfy \(\norm{f^*-f_n^*}_{(5)}\to 0\) and \(\norm{\partial_t(f^*)-\partial_t(f_n^*)}_{(1)}\to 0\). Let \(k\in\Cont^\infty_c(G)\) be non-negative with \(\int_Gk(v)\diff v=1\) and let \(\omega\in\Cont^\infty([0,T])\) satisfy \(\omega(0)=1\), \(\norm{\omega}_\infty\leq 1\) and \(\norm{\partial_t\omega}_\infty\leq 1\). Define functions \(g_n\in\Rel\) with 
	\[g_n(x,t,v)=f_n(x,t,v)-k(v)\omega(t)\int_G f_n(x,0,w)\diff w.\] 
	It follows that \(\norm{f^*-g_n^*}_{(5)}\to 0\) and \(\norm{\partial_t(f^*-g_n^*)}_{(1)}\to 0\). We estimate using \eqref{normestimate}
	\begin{align*}
	\norm{f-g_n}_{\si}& =\norm{f-g_n}+\norm{\BRAKET{f-g_n}{f-g_n}}^{1/2}\\
	& \leq D \norm{f-g_n}_{(0)}+\tilde{D}^{1/2}(\norm{f^*-g_n^*}_{(5)}+T\norm{\partial_t(f^*-g_n^*)}_{(1)}).
	\end{align*}
	Consequently, \(f\) lies in the closure of \(\Rel\) with respect to the \(\norm{\,\cdot\,}_{\si}\)\nb-norm.
\end{proof} 
Therefore, the generalized fixed point algebra \(\Fix^\Rp(J,\cl{\Rel})\) of the \(\Rp\)\nb-action on \(J\) is defined as in \cref{deffix}. 
The elements \(\KET{f}\BRA{g}\) for \(f,g\in\Rel^*_\grpd\) can be characterized more explicitly. We fix for the rest of this article a monotone increasing net   \((\chi_i)_{i\in I}\) of continuous compactly supported functions \(\chi_i\colon\Rp\to[0,1]\) with \(\chi_i\to 1\) uniformly on compact subsets to cut off at zero and infinity. We may assume that \(\chi_i(\lambda^{-1})=\chi_i(\lambda)\) for all \(\lambda >0\). As described in \cref{res:strictlimit}
\begin{align*}
\int_0^\infty \chi_i(\lambda)\sigma_\lambda(f^**g)\,\tfrac{\diff \lambda}{\lambda}
\end{align*}
converges to \(\KET{f}\BRA{g}\) with respect to the strict topology as multipliers of \(J\).

\subsection{The pseudo-differential extension}
For a homogeneous Lie group \(G\), let \(J_G\defeq\ker(\widehat{\pi_{\triv}})\idealin\Cst(G)\) and \(\Rel_G=\{f\in\Cont^\infty_c(G)\mid\int_G f(v)\diff v=0\}\). 
In the following, we use generalized fixed point algebras to derive a pseudo-differential extension
\begin{equation*}
\begin{tikzcd}
\Comp(L^2 G)\arrow[hook,r] & \Fix^\Rp(J,\cl{\Rel})\arrow[r,twoheadrightarrow,] &\Cont_0(G,\Fix^\Rp(J_{G},\cl{\Rel_G})).
\end{tikzcd}
\end{equation*}
We will justify the name ``pseudo-differential'' extension in \cref{sec:symbols} by comparing it to the calculus defined by Fischer--Ruzhansky--Fermanian-Kammerer in \cites{fischer2016quantization, fischer2017defect} for graded Lie groups.  

The homomorphism \(p_0\colon\Cst(\grpd)\to\Cst(TG)\) induced by restriction to \(t=0\) maps~\(J\) onto the \(\Rp\)-invariant ideal \(J_0\subset \Cst(TG)\) with
\[J_0=\bigcap_{x\in G} {\ker{(\widehat{\pi}_{\triv}\circ q_x)}}.\]
The short exact sequence from \eqref{tangentgroupoidext} restricts to
\begin{equation}
\begin{tikzcd}\label{restrictedses}
\Cont_0(\Rp)\otimes \Comp(L^2G)\arrow[r,hook] & J\arrow[r,twoheadrightarrow,"p_0"] & J_0.\end{tikzcd}
\end{equation}
\begin{proposition}
	Let \(G\) be a homogeneous Lie group and let \(\Rel_0=p_0(\Rel)\).  The zoom action \(\sigma\) on \eqref{restrictedses} induces an extension
	\begin{equation}
	\begin{tikzcd}\label{ses:fixedalg} 
	\Comp(L^2 G)\arrow[hook,r] & \Fix^\Rp(J,\cl{\Rel})\arrow[r,twoheadrightarrow,"\widetilde{p}_0"] &\Fix^\Rp(J_0,\cl{\Rel_0}).
	\end{tikzcd}
	\end{equation}
\end{proposition}
\begin{proof}\cref{res:sesfixed} gives an extension of generalized fixed point algebras
	\begin{equation*}
	\begin{tikzcd}
	\Fix^\Rp(\ker(p_0),\cl{\Rel}\cap\ker(p_0)) \arrow[r,hook] & \Fix^\Rp(J,\cl{\Rel})\arrow[r,twoheadrightarrow,"\widetilde{p}_0"] & \Fix^\Rp(J_0,\cl{p_0(\cl{\Rel})}).
	\end{tikzcd}
	\end{equation*} 		
	By \cref{rem:contsi} the completion of \(p_0(\cl{\Rel})\) with respect to the \(\norm{\,\cdot\,}_\si\)-norm  on the right hand side is~\(\cl{\Rel_0}\). The isomorphism \(p\) from \cref{res:kernelev} induces an isomorphism
	\begin{align*}
	p_*\colon \Fix^\Rp(\ker(p_0),\cl{\Rel}\cap\ker(p_0)) \to \Fix^\Rp(\Cont_0(\Rp,\Comp(L^2G)),\tilde{\Rel}).
	\end{align*}
	where \(\tilde{\Rel}\defeq p(\cl{\Rel}\cap \ker(p_0))\). By the description of the \(\Rp\)-action on \(\Cont_0(\Rp,\Comp(L^2G))\) in \cref{res:kernelev}, it follows that \(\tilde{\Rel}\) is the unique relatively continuous, complete and dense subset by \cref{res:spectrallyproper}. \cref{res:trivialfix} gives an isomorphism  
	\begin{align*}
	\Psi\colon\Fix^\Rp(\Cont_0(\Rp,\Comp(L^2G)),\tilde{\Rel})\to\Comp(L^2G).
	\end{align*}
	Therefore, we obtain \eqref{ses:fixedalg}. The inclusion of \(\Comp(L^2G)\) into \(\Fix^\Rp(J,\cl{\Rel})\) is given by \((\Psi\circ p_*)^{-1}\).
\end{proof}	

By \cref{res:fixcontfield} the symbol algebra \(\Fix^\Rp(J_0,\cl{\Rel_0})\) is a continuous field of \(\Cst\)-algebras over \(G\) with fibres \(\Fix^\Rp(J_G,\cl{\Rel_G})\). We show it is in fact a trivial continuous field.
\begin{lemma}\label{res:symbolfix_field}
	The map \(\Theta\colon \Fix^\Rp(J_0,\cl{\Rel_0})\to \Cont_0(G,\Fix^\Rp(J_G,\cl{\Rel_G}))\) given by
	\[\Theta(\KET{f}\BRA{g})(x)=\widetilde{q}_x(\KET{f}\BRA{g})=\KET{q_x(f)}\BRA{q_x(g)} \qquad \text{for }f,g\in\Rel_{0}\text{, }x\in G\]
	is an isomorphism. 
\end{lemma}

\begin{proof}
	By \cref{res:sesfixed} and \cref{rem:contsi} each \(\widetilde{q}_x\) maps \(\Fix^\Rp(J_{0},\cl{\Rel_{0}})\) onto \(\Fix^\Rp(J_G,\cl{\Rel_G})\) for \(x\in G\). Let \(f,g\in\Rel_{0}\), we show that \(\Theta(\KET{f}\BRA{g})\) is continuous. For~\(\varepsilon>0\) and \(x\in G\) there is an open neighbourhood \(U\) of \(x\) such that \(\norm{q_x(f)-q_y(f)}_{5}<\varepsilon\) and \(\norm{q_x(g)-q_y(g)}_{5}<\varepsilon\) for all \(y\in U\). The estimate of the norm in~\eqref{normestimate} implies that there is a constant \(C>0\) such that \(\norm{\KET{h}}\leq C\norm{h}_{5}\) for all \(h\in\Rel_G\).  Hence for \(y\in U\) we obtain
	\begin{align*}
	& \norm{\Theta(\KET{f}\BRA{g})(x)-\Theta(\KET{f}\BRA{g})(y)}\\
	\leq \,& \norm{\KET{q_x(f)}}\cdot\norm{\KET{q_x(g)-q_y(g)}}+\norm{\KET{q_y(g)}}\cdot\norm{\KET{q_x(f)-q_y(f)}} \\
	\leq \,& \varepsilon\,\left(\sup_{x\in G}{\norm{q_x(f)}_{5}}+\sup_{x\in G}{\norm{q_x(g)}_{5}}\right).
	\end{align*}
	As \(f\) and \(g\) are compactly supported in the \(x\)-direction it follows that \(\Theta(\KET{f}\BRA{g})\) is again compactly supported. Extend \(\Theta\) linearly to the span of \(\KET{f}\BRA{g}\) for \(f,g\in\Rel_{0}\) and let \(T\) be inside the linear span. As \(\norm{\widetilde{q}_x(T)}\leq\norm{T}\) for all \(x\in G\) it follows that \(\norm{\Theta(T)}\leq\norm{T}\). Let \(\psi\in J_{0}\) satisfy \(\norm{\psi}=1\). As \(\Cst(TG)\) is a continuous field of \(\Cst\)-algebras over \(G\) with fibres \(\Cst(G)\) it follows that
	\[\norm{T\psi}=\sup_{x\in G}\norm{q_x(T\psi)}=\sup_{x\in G}\norm{\widetilde{q}_x(T)q_x(\psi)}\leq \sup_{x\in G}\norm{\widetilde{q}_x(T)}=\norm{\Theta(T)}.\]
	Hence, \(\Theta\) is an isometry and extends by continuity to \(\Fix^\Rp(J_{0},\cl{\Rel_{0}})\). As \(\widetilde{q}_x\) is a homomorphism for each \(x\in G\), \(\Theta\) is a homomorphism. 
	
	Denote by \(W\subset\Fix^\Rp(J_G,\cl{\Rel_G})\) the linear span of \(\KET{f}\BRA{g}\) with \(f,g\in\Rel_G\), which is dense in \(\Fix^\Rp(J_G,\cl{\Rel_G})\). Then \(\Cont_c(G)\otimes^{\text{alg}}W\) is dense in \(\Cont_0(G,\Fix^\Rp(J_G,\cl{\Rel_G}))\). The space \(\Cont_c(G)\otimes^{\text{alg}}W\) is contained in the image of \(\Theta\) as for \(a\in\Cont_c(G)\) and \(f,g\in\Rel_G\) we can pick a function \(b\in\Cont_c(G)\) with \(b|_{\supp a}\equiv 1\) so that \(\Theta(\KET{a\otimes f}\BRA{b\otimes g})=a\otimes\KET{f}\BRA{g}\). As the image of \(\Theta\) is closed, the claim follows.
\end{proof}

\subsection{Representation as bounded operators on \(L^2(G)\)}
We show that the \(\Cst\)\nb-algebra of ``order zero pseudo-differential operators'' \(\Fix^\Rp(J,\cl{\Rel})\) admits a faithful representation as bounded operators on \(L^2(G)\).

The \(^*\)\nb-homomorphisms \(p_t\colon \grpd\to\Comp(L^2G)\) defined in \eqref{eq:p_t} can be restricted to the ideal \(J\). The restrictions are still surjective. Hence, they yield strictly continuous representations \[\widetilde{p}_t\colon \Fix^\Rp(J,\cl{\Rel})\to \Mult(\Comp(L^2G))=\Bound(L^2G)\qquad \text{for all }t>0.\]

\begin{lemma}
	The representation \(\widetilde{p}_1\colon \Fix^\Rp(J,\cl{\Rel})\to \Bound(L^2G)\) is faithful.
\end{lemma}

\begin{proof}
	As seen in \cref{actiononideal} the representations \(p_t\) of \(J\) are related by
	\begin{align}\label{eq:reps_>0}
	p_t\circ \sigma_\lambda = p_{t\lambda^{-1}} \qquad\text{for }t,\lambda>0.
	\end{align}
	This equality still holds for the corresponding extensions to the multiplier algebras. As each \(T\in\Fix^\Rp(J,\cl{\Rel})\) is invariant under \(\sigma\), it follows that \(\widetilde{p}_t(T)=\widetilde{p}_1(T)\) for all~\(t>0\). Therefore, \(T\in\ker(\widetilde{p}_1)\) implies  \(\widetilde{p}_t(T)=0\) for all \(t>0\). Then it follows for all \(f\in J\) that \(p_t(Tf)=\tilde{p}_t(T)p_t(f)=0\) for all \(t>0\). As \(\Cst(\grpd)\) is a continuous field of \(\Cst\)-algebras this implies by continuity that \(Tf=0\). Since this holds for all \(f\in J\), it follows \(T=0\). 
\end{proof}
\begin{lemma}\label{res:fix_on_l2}
	 Let \(f,g\in\Rel_\grpd^*\) and \(h=f^**g\). Then the operators \(T_i(h)\) given by
	\begin{align}\label{def:T_i}
	T_i(h)\psi(x)=\int_0^\infty \chi_i(\lambda)\lambda^{-Q}\int h(x,\lambda,\alpha_{\lambda^{-1}}(x^{-1}y))\psi(y)\diff y
	 \tfrac{\diff\lambda}{\lambda}\end{align} for \(\psi\in L^2(M)\), \(x\in M\),
	converge strictly to \(\widetilde{p}_1(\KET{f}\BRA{g})\) as multipliers of \(\Comp(L^2M)\). 
\end{lemma}
\begin{proof} 		
	Strict continuity of \(\widetilde{p}_1\) together with \eqref{eq:reps_>0} imply that
	\begin{align*}
	\widetilde{p}_1\left(\KET{f}\BRA{g}\right)=\lim_{s} \int_0^\infty \chi_i(\lambda )p_1(\sigma_\lambda(f^**g))\tfrac{\diff\lambda}{\lambda}=\lim_{s} \int_0^\infty \chi_i(\lambda)p_\lambda(h)\tfrac{\diff\lambda}{\lambda}.
	\end{align*}
	The operators \(T_i(h)\) above are obtained by inserting the definition of \(p_\lambda\). 
\end{proof}
\begin{lemma}\label{res:compboundcommute}
	The following diagram commutes, where the horizontal maps are the inclusions:
	\begin{equation*}
	\begin{tikzcd}
	\ker(\widetilde{p}_0) \arrow[r,hook] \arrow[d, "(p)_*","\cong"']& \Fix^\Rp(J,\cl{\Rel})\arrow[dd, "\widetilde{p}_1"]	\\
	\Fix^\Rp(\Cont_0(\Rp,\Comp(L^2G)),\tilde{\Rel})\arrow[d,"\Psi", "\cong"'] \\
	\Comp(L^2G)\arrow[r,hook] & \Bound(L^2G).
	\end{tikzcd}
	\end{equation*}	
\end{lemma}
\begin{proof}
	Let \(\psi_1,\psi_2\in\Cont_c(\Rp,\Comp(L^2G))\). Then by strict continuity of \(\widetilde{p}_1\) and \eqref{eq:reps_>0}
	\begin{align*}
	\widetilde{p}_1({(p_*)^{-1}({\KET{\psi_1}\BRA{\psi_2}})})&= \widetilde{p}_1(\KET{p^{-1}(\psi_1)}\BRA{p^{-1}(\psi_2)})\\
	& = \widetilde{p}_1\Big({\lim_s\int_0^\infty\chi_i(\lambda)\sigma_\lambda(p^{-1}(\psi_1^*\psi_2))\tfrac{\diff \lambda}{\lambda}}\Big)\\
	& = \lim_s \int_0^\infty \chi_i(\lambda)p_{\lambda^{-1}}(p^{-1}(\psi_1^*\psi_2))\tfrac{\diff \lambda}{\lambda}\\
	& = \int_0^\infty \psi_1(\lambda^{-1})^*\psi_2(\lambda^{-1})\tfrac{\diff \lambda}{\lambda} = \Psi(\KET{\psi_1}\BRA{\psi_2})
	\end{align*}
	holds. The claim follows as the linear span of \(\KET{\psi_1}\BRA{\psi_2}\) with \(\psi_1,\psi_2\in\Cont_c(\Rp,\Comp(L^2G))\) is dense in  \(\Fix^\Rp(\Cont_0(\Rp,\Comp(L^2G)),\tilde{\Rel})\).
\end{proof}
\begin{lemma}\label{res:compact_fact}
	Let \(h\in \Rel^*_\grpd\cap\ker(p_0)\) and let \(T_i(h)\) be defined as in \eqref{def:T_i}. Then \((T_i(h))\) converges in norm in \(\Comp(L^2G)\). In particular, its strict limit as multipliers of \(\Comp(L^2G)\) exists and is contained in \(\widetilde{p}_1(\Fix^\Rp(J,\cl{\Rel}))\). 
\end{lemma}
\begin{proof}
	As \(h\in\Rel^*_\grpd\) vanishes at \(t=0\), it can be written as \(h=tf\) with \(f\in\Cont^\infty(\grpd)\cap\Cst(\grpd)\). By definition of the representation \(p_\lambda\) in \eqref{eq:p_t}, it follows that \(p_\lambda(h)=\lambda p_\lambda(f)\) for all \(\lambda>0\). Hence, for all \(\lambda>0\) \begin{align*}
	\norm{p_\lambda(h)}\leq \lambda\norm{p_\lambda(f)}\leq \lambda\norm{f}.
	\end{align*}
	We show that \((T_i(h))\) is Cauchy. Let \(T>0\) be such that \(h\) vanishes for \(t\geq T\). For \(j\geq i\), we estimate
	\begin{align*}
	\norm{T_j(h)-T_i(h)}\leq \int_0^\infty (\chi_j(\lambda)-\chi_i(\lambda))\norm{p_\lambda(h)}\tfrac{\diff\lambda}{\lambda}
	\leq \norm{f} \int_0^T (1-\chi_i(\lambda))\diff\lambda.
	\end{align*}
	As \(\chi_i\to 1\) on compact subsets, the first claim follows. The second claim holds as convergence in norm implies strict convergence and \(\Comp(L^2G)\) is contained in \(\widetilde{p}_1(\Fix^\Rp(J,\cl{\Rel}))\) by \cref{res:compboundcommute}.
\end{proof}
\section{Operators of type \(0\) as generalized fixed points}\label{sec:fix_type_0}
Let \(G\) be a homogeneous Lie group. In this section we show that \(\Fix^\Rp(J_G,\cl{\Rel_G})\) is the \(\Cst\)-closure of the operators of type \(0\) on~\(G\).
\subsection{Homogeneous distributions}
Recall that the dilations yield an \(\Rp\)-action on \(\Schwartz(G)\subset\Cst(G)\) given by \(\sigma_\lambda f(x)=\lambda^{Q}f(\alpha_\lambda(x))\) for \(\lambda>0, f\in\Schwartz(G)\) and \(x\in G\).
This action can be extended to \(\Schwartz'(G)\) by
\[\langle \sigma_{\lambda*} u, f\rangle \defeq \lambda^{Q}\langle u, \sigma_{\lambda^{-1}} f\rangle \quad \text{for }u\in\Schwartz'(G),f\in\Schwartz(G) \text{ and }\lambda>0.\]
\begin{definition}[\cite{fischer2016quantization}*{3.2.9}]
	Let \(G\) be a homogeneous Lie group and let \(\nu\in\R\). A distribution~\(u\in\Schwartz'(G)\) is called a \emph{kernel of type \(\nu\)} if it is smooth away from zero and \(\sigma_{\lambda*}(u)=\lambda^\nu u\) for all~\(\lambda>0\). Denote by \(\mathcal{K}^\nu(G)\) the space of kernels of type \(\nu\). 
	
	For \(u\in\mathcal{K}^\nu(G)\) the corresponding continuous operator \(T_u\colon\Schwartz(G)\to\Schwartz'(G)\) given by \(T_u(f)=u* f\) is called an \emph{operator of type \(\nu\)}.	
\end{definition}
\begin{remark}
	Let \(Q\) be the homogeneous dimension of the group \(G\). The kernels of type \(\nu\) are also called \((\nu-Q)\)-homogeneous distributions in the literature. This is because they coincide with \((\nu-Q)\)-homogeneous functions outside zero. 
\end{remark}
As in \cite{fischer2016quantization}*{3.2.7} one calculates that an operator \(T\) of type \(\nu\) satisfies \(T(\sigma_{\lambda^{-1}}f)=\lambda^\nu\sigma_{\lambda^{-1}}(Tf)\) for all \(\lambda>0\) and \(f\in\Schwartz(G)\).
\begin{example}\label{ex:deriv_askernel}Let \(\delta\in\Schwartz'(G)\) denote the delta distribution.
	For \(\alpha\in\N^n_0\) the distribution~\(X^\alpha\delta\in\Schwartz'(G)\) defines a kernel of type~\(-[\alpha]\). The corresponding operator \(T_{X ^\alpha\delta}\) is the right-invariant differential operator \(Y^\alpha\) on \(G\) defined as in \eqref{def:leftinvariant_diff}. 
\end{example}
\subsection{Factorization}
Define the following subspace of \(\Schwartz(G)\) which is the Euclidean Fourier transform of the space of Schwartz functions that vanish with all derivatives at \(0\).
\begin{definition}[\cite{christ1992pseudo}]
	For a homogeneous Lie group \(G\), let \(\Schwartz_0(G)\) consist of all functions \(f\in\Schwartz(G)\) with \(\int_G v^\alpha f(v)\diff v=0\) for all \(\alpha\in\N^n_0\).
\end{definition}

\begin{lemma}\label{res:schwartz0}
	The space \(\Schwartz_0(G)\) has the following properties:
	\begin{enumerate}
		\item \label{item:s0_subalgebra} \(\Schwartz_0(G)\) is a \(^*\)-subalgebra of \(\Schwartz(G)\),
		\item \label{item:closed} \(\Schwartz_0(G)\) is closed in \(\Schwartz(G)\) with respect to the Schwartz semi-norms,
		\item\label{item:invar_hom_distr} \(u * f\in\Schwartz_0(G)\) for all \(u\in\mathcal{K}^\nu(G)\), \(\nu\in\R\) and \(f\in\Schwartz_0(G)\),
		\item \label{item:deriv} \(Xf\in\Schwartz_0(G)\) for all \(X\in\lieg\) and \(f\in\Schwartz_0(G)\).			
	\end{enumerate}
\end{lemma}
\begin{proof}
	It is clear that \(\Schwartz_0(G)\) is a linear subspace of \(\Schwartz(G)\). It is left to show for \ref{item:s0_subalgebra} that \(\Schwartz_0(G)\) is closed under involution and convolution. Note that for \(v\in G\)
	\begin{align*}
	(v^{-1})^\alpha &= (-v)^\alpha=(-1)^{\abs{\alpha}}v^\alpha.			
	\end{align*}
	Moreover, by the polynomial group law in \cref{triangularprop} one can write \(v\cdot w\) for \(v,w\in G\) as a linear combination of \(v^\beta w^\gamma\).
	
	To show \ref{item:closed}, let \(f\in\Schwartz(G)\) and \(\alpha\in\N^n_0\). We estimate
	\[\left|\int v^\alpha f(v)\diff v \right|\leq \int{(1+\norm{v})^{[\alpha]}\abs{f(v)}}\diff v \leq \norm{f}_{[\alpha]}\int (1+\norm{v})^{-Q-1} \diff v .\] 
	The integral converges by \cref{res:convint}. 
	It follows that for a sequence \((f_k)_{k\in\N}\) in~\(\Schwartz_0(G)\) which converges in \(\Schwartz(G)\), the limit lies in \(\Schwartz_0(G)\) as well. 
	
	Property \ref{item:invar_hom_distr} is proved in \cite{geller_liouville}*{Lemma 4}.  For \(\ref{item:deriv}\) write \(X\in\lieg\) as \(X=\sum_{j=1}^n a_jY_j\) with~\(a_j\in\R\). Then \(Xf=X\delta * f= \sum_{j=1}^n a_j X_j\delta*f\)
	As the \(X_j\delta\) define kernels of type \(-q_j\) by \cref{ex:deriv_askernel}, the claim follows from  \ref{item:invar_hom_distr}.		
\end{proof}
\subsection{Operators of type \(0\)}We compare now the operators of type \(0\) to the generalized fixed points of the \(\Rp\)-actions on \(J_G\).
\begin{proposition}[\cite{folland1982homogeneous}*{1.65}]\label{res:kerneltype_average}
	Let \(G\) be a homogeneous Lie group and \(f\in\Schwartz(G)\) with \(\int_G f(v)\diff v=0\). Then 
	\[\int_0^\infty \chi_i(\lambda)\sigma_\lambda(f)\tfrac{\diff\lambda}{\lambda}\]
	converges in \(\Schwartz'(G)\) to a kernel of type \(0\). 
\end{proposition}

We will show that, conversely, every \(u\in\kernel^0(G)\) can be written as an average over the dilation action as above. First, this is proved for \(u=\delta\), which is a kernel of type~\(0\).
\begin{lemma}\label{res:delta}
	There are \(f_j\in\Rel_G\) and \(g_j\in\Schwartz_0(G)\), \(j=1,\ldots,n\), such that
	\[ \sum_{j=1}^n \int_0^\infty \chi_i(\lambda)\sigma_\lambda(f_j^**g_j)\tfrac{\diff\lambda}{\lambda}\to \delta \quad \text{in }\Schwartz'(G).\]
\end{lemma}
\begin{proof}	 	
	As noted in \cite{christ1992pseudo}, one can find a \(\phi\in\Schwartz_0(G)\) with 
	\[\delta = \lim_{i}\int_0^\infty\chi_i(\lambda)\sigma_\lambda(\phi)(x)\tfrac{\diff \lambda}{\lambda}.\]
	For example, take a function \(h\in\Cont_c^\infty(\Rp)\) with \(\int_{0}^{\infty}h(\lambda^{-1})\tfrac{\diff \lambda}{\lambda}=1\). Define \(\psi\in\Schwartz(\lieg^*)\) with  \(\psi(\xi)=h(\norm{\xi})\), where \(\norm{\,\cdot\,}\) is a homogeneous quasi-norm on \(\lie{g}^*\) which is smooth outside zero. The invariance of the Haar measure on \(\Rp\) implies that \(\int_{0}^{\infty}\psi(\alpha_{\lambda^{-1}}(\xi))\tfrac{\diff \lambda}{\lambda}=1\) for all \(\xi\neq 0\).  Therefore, \(\phi\in\Schwartz_0(G)\) can be taken as the inverse Euclidean Fourier transform of \(\psi\).
	
	Now, \(\phi\) needs to be factorized appropriately. Dixmier and Malliavin proved in \cite{dixmier1978factorization}*{7.2} that one can find \(k,l\in\Schwartz(G)\) such that \(\phi=k*l\). In the first step of the proof they show \(\phi = \theta*\mu\), where \(\mu\) is a measure and \(\theta\) is the limit of a sequence of polynomials in \(X^\alpha \phi\) in \(\Schwartz(G)\). It follows \(\theta\in\Schwartz_0(G)\) by \cref{res:schwartz0}. Repeating this procedure with \(\theta\), the factorization \(\phi =k*l\) is achieved after finitely many steps with \(k\in\Schwartz_0(G)\) and \(l\in\Schwartz(G)\). 
	By \cite{folland1982homogeneous}*{1.60}, \(k\) can be written as  \(k =\sum_{j=1}^{n}X_jk_j\) with \(k_1,\ldots,k_n\in\Schwartz_0(G)\). Therefore, we obtain
	\[\phi = \sum_{j=1}^n{(X_jk_j)*l}=\sum_{j=1}^{n}{k_j*(Y_jl)}.\]
	Using \cref{res:vectorfields} one obtains that \(\int (Y_jl)(x)\diff x = 0\) holds. Consequently, the claim holds with \(f_j\defeq k_j^*\in\Schwartz_0(G)\) and \(g_j\defeq Y_jl\in\Rel_G\) for \(j=1,\ldots,n\). 
\end{proof}
\begin{corollary}\label{res:everykernelaverage}
	Every \(u\in\kernel^0(G)\) can be written as a finite sum of 
	\[\lim_i\int_0^\infty \chi_i(\lambda)\sigma_\lambda(f^**g)\tfrac{\diff\lambda}{\lambda}\]
	with \(f\in\Schwartz_0(G)\) and \(g\in\Rel_G\). 
\end{corollary}
\begin{proof}
	Let \(f_j\in\Schwartz_0(G)\) and \(g_j\in\Rel_G\) for \(j=1,\ldots,n\) be functions as in \cref{res:delta}. Then use the homogeneity of \(u\) to compute 
	\begin{align*} u&= u*\delta= \lim \left(\sum_{j=1}^n  \int_0^\infty \chi_i(\lambda)u*(\sigma_\lambda(f_j*g_j))\tfrac{\diff \lambda}{\lambda}\right)
	\\&=\lim\left(\sum_{j=1}^l{ {\int_0^\infty \chi_i(\lambda)\sigma_\lambda((u*f_j)*g_j)\tfrac{\diff \lambda}{\lambda}}}\right)	\\
	& = \sum_{j=1}^l\left(\lim \int_0^\infty \chi_i(\lambda) \sigma_\lambda((u*f_j)*g_j)\tfrac{\diff\lambda}{\lambda}\right).
	\end{align*}
	Here we used that \(u*(\phi*\psi)=(u*\phi)*\psi\) for all \(\phi,\psi\in\Schwartz(G)\). The last equality holds as \(u*f_j\) lies in \(\Schwartz_0(G)\) for all \(j=1,\ldots,n\) by \cref{res:schwartz0}, so that the expression in the bottom line converges in \(\Schwartz'(G)\) by \cref{res:kerneltype_average}.
\end{proof}
The description of the kernels of type \(0\) above is already very close to the generalized fixed point algebra construction. We compare now the corresponding convolution operators to elements of \(\Fix^\Rp(J_G,\cl{\Rel_G})\subset \Mult^\Rp(J_G)\). To understand this fixed point algebra better, we show that the restrictions of the left regular representation \(\lambda_G\colon \Cst(G)\to\Bound(L^2G)\)  to \(J_G\) is still non-degenerate.
\begin{lemma}\label{res:nondeg}
	Let \(G\) be a homogeneous Lie group \(G\). The restriction of the left regular representation \(\lambda_G\colon \Cst(G)\to\Bound(L^2G)\) to \(J_G=\ker(\widehat{\pi}_{\triv})\) is a non-degenerate representation. 
\end{lemma}
\begin{proof}
	Suppose \(\psi\in L^2(G)\) is such that \(f*\psi=0\) holds for all \(f\in J_G\).  As \(\Cst(G)\) acts by right-invariant operators on \(L^2(G)\), this is equivalent to \(\widehat{f}(\pi)\widehat{\psi}(\pi)=0\) for all \(f\in J_G\) and for almost all \(\pi\in\widehat{G}\) by the Plancherel Theorem, see \eqref{eq:Fourier-VN-Linfty}. 
	The ideal \(J_G\idealin\Cst(G)\) is liminal. Hence for \(\pi\in\widehat{J_G}=\widehat{G}\!\setminus\!\{\pi_{\triv}\}\) we have that \(\widehat{f}(\pi)\widehat{\psi}(\pi)=0\) for all \(f\in J_G\) is equivalent to \(\Comp(\Hils_\pi)\widehat{\psi}(\pi)=0\). But as \(\widehat{\psi}(\pi)\) is Hilbert--Schmidt, this means that \(\widehat{\psi}(\pi)=0\) for \(\pi\neq\pi_{\triv}\). The Plancherel measure is supported within the representations corresponding to orbits of maximal dimension sequence. In particular, \(\{\pi_{\triv}\}\) has measure zero and, therefore, \(\psi=0\) must hold.  
\end{proof}
Consequently, as the restriction is also faithful, the multiplier algebra \(\Mult(J_G)\) can be identified with the idealizer of \(\lambda_G(J_G)\) in \(\Bound(L^2G)\). Therefore, elements of the generalized fixed point algebra \(\Fix^\Rp(J_G,\cl{\Rel_G})\) can be viewed as bounded, right-invariant operators on \(L^2(G)\). This allows us to reprove the following theorem of \cite{knappstein}.

\begin{theorem}\label{res:type0bounded}
	Let \(G\) be a homogeneous Lie group. Every operator \(T\) of type \(0\) extends uniquely to a bounded operator \(L^2(G)\to L^2(G)\). 
\end{theorem}

\begin{proof}Note that if such an extension exists, it is unique as \(\Schwartz(G)\subset L^2(G)\) is dense.
	Let \(f,g\in\Rel_G\) and consider the kernel of type \(0\) given by
	\[u = \lim_i \int_0^\infty \chi_i(\lambda)\sigma_\lambda(f^**g)\tfrac{\diff\lambda}{\lambda}.\]
	If we can show that \(T_{u}\colon\psi\mapsto u*\psi\) defined on \(\Schwartz(G)\) extends to an operator in \(\Bound(L^2G)\), the claim follows as every kernel of type \(0\) is a finite sum of kernels of this form by \cref{res:everykernelaverage}.  Let \(h\in\Rel_G\) and \(\psi\in \Schwartz(G)\). The description of \(\KET{f}\BRA{g}\) as a strict limit as in \cref{res:strictlimit} shows that
	\begin{align*}
	\widetilde{\lambda}_G\left(\KET{f}\BRA{g}\right)(\lambda_G(h)\psi)=T_{u}(\lambda_G(h)\psi).
	\end{align*}
	As the restricted left regular representation is non-degenerate by \cref{res:nondeg}, \(\{\lambda_G(h)\psi\mid h\in\Rel_G\text{ and }\psi\in\Schwartz(G)\}\) is dense in \(L^2(G)\). Hence, the operator~\(\widetilde{\lambda}_G\left(\KET{f}\BRA{g}\right)\) is the unique continuous extension of \(T_{u}\).
\end{proof}

\begin{proposition}\label{res:kernelzerograded}
	Let \(G\) be a homogeneous Lie group \(G\).  Then \(\Fix^\Rp(J_G,\cl{\Rel_G})\) is the \(\Cst\)-closure of the operators of type \(0\)  in \(\Bound(L^2G)\). 
\end{proposition}
\begin{proof}
	This follows from the one-to-one correspondence between the operators of type~\(0\) and the linear span of \(\KET{f}\BRA{g}\) with \(f,g\in\Rel_G\) obtained from \cref{res:kerneltype_average}, \cref{res:everykernelaverage} and \cref{res:type0bounded}. The latter is dense in the \(\Cst\)-algebra \(\Fix^\Rp(J_G,\cl{\Rel_G})\).
\end{proof}

\section{Comparison to the calculus for graded nilpotent Lie groups}
\label{sec:symbols}
For this section, let \(G\) be a graded nilpotent Lie group.
We compare the sequence
\begin{equation*}
\begin{tikzcd}
\Comp(L^2 G)\arrow[hook,r] & \Fix^\Rp(J,\cl{\Rel})\arrow[r,twoheadrightarrow,"\Theta\circ\widetilde{p}_0"] &\Cont_0(G,\Fix^\Rp(J_{G},\cl{\Rel_G})).
\end{tikzcd}
\end{equation*}  from \cref{sec:pdo_using_fix} to the pseudo-differential extension of order zero in the calculus of Fischer, Ruzhansky and Fermanian-Kammerer.
\subsection{The calculus of Fischer--Ruzhansky--Fermanian-Kammerer}\label{sec:nilpotent_calc}
Recall the operator-valued Fourier transform for a nilpotent Lie group \(G\) described in \cref{sec:plancherel}. It maps a left-invariant operator on \(L^2(G)\) to a field of operators \(\{a(\pi)\in\Bound(\Hils_\pi)\mid \pi\in\widehat{G}\}\).
It is used in \cite{fischer2016quantization} to define a symbolic pseudo-differential calculus. In~\cite{fischer2017defect}, homogeneous expansions for the symbols are defined. We give a short introduction to their calculus. The symbols in their calculus are fields of operators \[\{a(x,\pi)\colon\Hils^\infty_\pi\to\Hils_\pi\mid x\in G\text{, }\pi\in\widehat{G}\}.\] Here, \(\Hils^\infty_\pi\) are the smooth vectors in \(\Hils_\pi\).
\begin{remark}
	Note that \cite{fischer2016quantization} uses a different convention for the Fourier transform 
	\[\Four(f)(\pi)=\int_G f(x)\pi(x)^*\diff x \qquad\text{for }f\in L^1(G).\]
	This leads to \(\Four(f*g)(\pi)=\Four(g)(\pi)\Four(f)(\pi)\). 
\end{remark}
The pseudo-differential calculus is defined in \cite{fischer2016quantization} using a positive Rockland operator. Each \(\pi\in\widehat{G}\) yields an \emph{infinitesimal representation} \(\diff\pi\) of \(\Uni(\lie{g})\) on \(\Hils^\infty_\pi\) (see \cite{fischer2016quantization}*{1.7.3}). As in \cite{fischer2016quantization} we will write \(\pi(P)\defeq \diff\pi(P)\) for left-invariant differential operators \(P\) on \(G\).
\begin{definition}[\cite{fischer2016quantization}*{4.1.1,~4.1.2}]
	Let \(G\) be a homogeneous Lie group. A left-invariant differential operator \(P\) on \(G\) satisfies the \emph{Rockland condition} if \(\pi(P)\) is injective on \(\Hils^\infty_\pi\) for all \(\pi\in\widehat{G}\setminus\{\pi_{\triv}\}\). 	
	A left-invariant differential operator \(P\) which is homogeneous of positive degree and satisfies the Rockland condition is called a \emph{Rockland operator}.
\end{definition}
\begin{example}
	For \(G=\R^n\) the Laplace operator \(\Delta_n=\sum_{j=1}^n\partial_j^2\) is a Rockland operator. There is an isomorphism \(\R^n\to\widehat{\R^n}\), \(\xi\mapsto \pi_\xi\), given by \(\pi_\xi(x)=e^{-i \langle\xi,x\rangle}\). One computes that \(\pi_\xi(\partial_j)=- i \xi_j\). Hence, \(\pi_\xi(\Delta_n)=-\norm{\xi}^2\neq 0\) for \(\xi\neq 0\). 
\end{example}
\begin{example}[\cite{fischer2016quantization}*{4.1.8}]
	Let \(G\) be a graded Lie group with weights \(q_1\leq q_2\leq\ldots\leq q_n\) and corresponding basis \(X_1,\ldots,X_n\) of \(\lie{g}\). Let \(q\) be a common multiple of the weights. Then the following operator is a Rockland operator 
	\[\sum_{j=1}^n (-1)^{q/q_j}X_j^{2q/q_j}.\]
\end{example}

\begin{remark}
	The existence of a positive Rockland operator on a homogeneous Lie group is equivalent to the group being (up to rescaling) graded (see \cite{fischer2016quantization}*{4.1.3, 4.1.8}). 
\end{remark}

\begin{remark}The Helffer--Nourrigat Theorem \cite{helffernourrigat} states that a left-invariant homogeneous differential operator on a graded Lie group satisfies the Rockland condition if and only if it is hypoelliptic.\end{remark}

From now on, let \(\Rock\) be a fixed positive Rockland operator of homogeneous degree~\(q\) on \(G\). It will plays the role of the Laplace operator  on~\(\R^n\) in the Euclidean calculus. 

Using the positive Rockland operator and its functional calculus, the Sobolev spaces \(L^2_s(G)\) for \(s\in\R\) are defined in \cite{fischer2016quantization}*{4.4.2}. Moreover, the operator-valued Fourier transform extends to a Fourier transform between left-invariant operators in \(\Bound(L^2_a(G),L^2_b(G))\) for \(a,b\in\R\) to a space of fields denoted by \(L^\infty_{a,b}(\widehat{G})\), see \cite{fischer2016quantization}*{5.1.21,~5.1.24}. The spaces of corresponding convolution kernels in \(\Schwartz'(G)\) are denoted by \(\kernel_{a,b}(G)\).

The derivatives in the cotangent direction in the Euclidean calculus are replaced with \emph{difference operators} \(\Delta^\alpha\) for \(\alpha\in\N^n_0\), which are defined in the following. This is based on the observation that the Euclidean Fourier transform intertwines \(\partial^\alpha\) and with multiplication by \(x^\alpha\). For \(u\in\Schwartz'(G)\) denote by \(x^\alpha u\) for \(\alpha\in\N^n_0\) the tempered distributions defined by
\[\langle x^\alpha u, f\rangle = \langle u, x^\alpha f\rangle \quad\text{for }f\in\Schwartz(G).\]
Let \(u\in\kernel_{a,b}(G)\) be a kernel such that \(x^\alpha u\in\kernel_{a',b'}(G)\) for some \(a',b'\in\R\). Then the difference operator \(\Delta^\alpha\) is defined as in \cite{fischer2016quantization}*{5.2.1}	by  \[\Delta^\alpha\widehat{u}(\pi)\defeq \widehat{x^\alpha u}(\pi)\quad\text{for }\pi\in\widehat{G}.\]
The following symbol classes are adapted to the notion of order induced by the dilations. Hence, the homogeneous degree \([\alpha]\) for \(\alpha\in\N^n_0\) as in \cref{def:multiindex} is used.

\begin{definition}[\cite{fischer2016quantization}*{5.2.11}]
	A field \(\{a(x,\pi)\colon\Hils^\infty_\pi\to\Hils_\pi\mid x\in G\text{, }\pi\in\widehat{G}\}\) is a \emph{symbol of order \(m\in\R\)} if for all \(\alpha,\beta\in\N^n_0\), the field of operators \(X^\beta_x\Delta^\alpha a(x,\pi)\) is defined on the smooth vectors and satisfies
	\[\sup_{(x,\pi)\in G\times\widehat{G}}\norm{X^\beta_x\Delta^\alpha a(x,\pi)\pi(I+\Rock)^\frac{[\alpha]-m}{q}}_{\Bound(\Hils_\pi)}<\infty.\]
	Denote the class of symbols of order \(m\) by \(S^m\).
	For \(a\in S^m\) and \(\alpha,\beta\in\N^n_0\) set 
	\[ \norm{a}_{S^m,\alpha,\beta}=\sup_{(x,\pi)\in G\times\widehat{G}}\norm{X^\beta_x\Delta^\alpha a(x,\pi)\pi(I+\Rock)^\frac{[\alpha]-m}{q}}_{\Bound(\Hils_\pi)}.\]
	The \emph{smoothing symbols} are \(S^{-\infty}=\bigcap_{m\in\R}S^m\).
\end{definition}
One can form asymptotic expansions of these symbols in the following sense.
\begin{proposition}[\cite{fischer2016quantization}*{5.5.1}]
	Let \(\{a_j\}_{j\in\N_0}\) be a sequence of symbols \(a_j\in S^{m_j}\) with \(m_j\) strictly decreasing to \(-\infty\) as \(j\to\infty\). Then there is a symbol \(a\in S^{m_0}\), unique modulo \(S^{-\infty}\), such that
	\[ a-\sum_{j=0}^M{a_j}\in S^{m_{M+1}} \quad\text{for all }M\in\N.\]
\end{proposition}
In this case, one writes \( a \sim \sum_{j=0}^\infty a_j\).
\begin{proposition}[\cite{fischer2017defect}*{5.2.12, 5.2.17}]
	The symbol classes have the following properties:
	\begin{enumerate}
		\item \(S^{m_1}\subset S^{m_2}\) for \(m_1<m_2\).
		\item Each differential operator \(\sum c_\alpha(x)X^\alpha\) with coefficients \(c_\alpha\in\Cont^\infty(G)\) is contained in \(S^m\), where \(m=\max\{[\alpha]\mid c_\alpha\neq 0\}\).
		\item For \(a\in S^m\) and \(\alpha,\beta\in\N_0^n\) the symbol \(X^\beta\Delta^\alpha a\) is contained in \(S^{m-[\alpha]}\).  		
	\end{enumerate}
\end{proposition}
For \(a\in S^m\) the following quantization formula is well-defined and yields a continuous operator \(\Op(a)\colon\Schwartz(G)\to\Schwartz(G)\) by \cite{fischer2016quantization}*{5.2.15}
\begin{align}\label{eq:quantization}
\Op(a)\varphi(x)=\int_{\widehat{G}}\tr\left({\pi(x)a(x,\pi)\widehat{\varphi}(\pi)}\right)\diff\mu(\pi) \qquad\text{for }\varphi\in\Schwartz(G)\text{, }x\in G.
\end{align}
Let \(\kappa_x\in\Schwartz'(G)\) be the convolution kernel of the left-invariant operator whose Fourier transform is~\(\{a(x,\pi)\mid\pi\in\widehat{G}\}\). Then \(\Op(a)\varphi(x)=\varphi*\kappa_x\) holds. Denote by \(K_a\in\Schwartz'(G\times G)\) the integral kernel of \(\Op(a)\). It is formally given by \(K_a(x,y)=\kappa_x(y^{-1}x)\).

In the following, we will consider operators with compactly supported integral kernels. Let \(S^m_{\cp}\) consist of all symbols \(a\in S^m\) such that \(\Op(a)\) has a compactly supported integral kernel. Set \(S^{-\infty}_{\cp}=\bigcap_{m\in\R}S^m_{\cp}\).

The following properties for the symbols with compactly supported integral kernels follow from the respective properties for the symbol classes \(S^m\) shown in \cite{fischer2016quantization}*{5.5.8,~5.5.12,~5.7.2,~5.4.9,~5.2.21}.
\begin{proposition}\label{res:calc}
	The pseudo-differential calculus has the following properties:
	\begin{enumerate}
		\item Let \(m_1,m_2\in\R\). For \(A\in\Op(S^{m_1}_{\cp}), B\in\Op(S^{m_2}_{\cp})\) the composition \(AB\) lies in \(\Op(S^{m_1+m_2}_{\cp})\). 
		\item Let \(m\in\R\). For \(A\in\Op(S^m_{\cp})\) the formal adjoint \(A^*\) lies in \(\Op(S^m_{\cp})\).  
		\item \(A\in \Op(S^m_{\cp})\) extends to a bounded operator \(L^2_s(G)\to L^2_{s-m}(G)\) for \(s\in\R\).
		\item \(A\in\Op(S^{-\infty}_{\cp})\) if and only if its integral kernel lies in \(\Cont^\infty_c(G\times G)\). 
	\end{enumerate}
\end{proposition}
For the following lemma, see also \cite{fischer2017defect}*{4.24}.
\begin{lemma}\label{compactop}
	Let \(A\in\Op(S^m_{\cp})\) for \(m<0\). Then \(A\) extends to a compact operator on \(L^2(G)\).
\end{lemma}
\begin{proof}
	By \cref{res:calc}, \(A\) extends to a bounded operator \(A\colon L^2(G)\to L^2_{-m}(G)\). Let \(\chi\in\Cont_c^\infty(G)\) be constant \(1\) on the support of \(A\) in the \(x\)-direction and be supported in a compact subset \(K\subset G\). The map \(f\mapsto \chi\cdot f\) extends to a bounded operator \(L^2_{-m}(G)\to H^{-m/q_n}(K)\) by \cite{fischer2017defect}*{2.17}, where \(H^{-m/q_n}(K)\) denotes the Euclidean Sobolev space. By Rellich's Theorem the embedding \[H^{\frac{-m}{q_n}}(K)\injto L^2(\R^n)=L^2(G)\] is compact as \(-m/q_n>0\).  Therefore, the composition \(A\colon L^2(G)\to L^2(G)\) is compact. 
\end{proof}

Moreover, in \cite{fischer2017defect} classical pseudo-differential operators, which admit a homogeneous expansion of their symbol, are defined.  
\begin{definition}[\cite{fischer2017defect}*{4.1,~4.20}]
	Let \(m\in\R\). A field \(\{a(x,\pi)\colon\Hils^\infty_\pi\to\Hils_\pi\mid x\in G\text{, }\pi\in\widehat{G}\}\) is a \emph{regular \(m\)-homogeneous symbol} if
	\begin{enumerate}
		\item \(a(x,\lambda\cdot\pi)=\lambda^ma(x,\pi)\) for all \(x\in G\) and almost all \(\pi\) and \(\lambda>0\),
		\item for all \(\alpha,\beta\in\N^n_0\), the field of operators \(X^\beta_x\Delta^\alpha a(x,\pi)\) is defined on smooth vectors and satisfies
		\[\sup_{(x,\pi)\in G\times\widehat{G}}\norm{X^\beta_x\Delta^\alpha a(x,\pi)\pi(\Rock)^\frac{[\alpha]-m}{q}}_{\Bound(\Hils_\pi)}<\infty,\]			
	\end{enumerate}
	Denote by \(\dot{S}^m\) the class of all regular \(m\)-homogeneous symbols and by \(\dot{S}^m_c\) the ones with compact support in the \(x\)-direction.
\end{definition}
\begin{example}[\cite{fischer2017defect}*{4.3, 4.4}]
	For each \(c\in\Cont_c^\infty(G)\) and multi-index \(\alpha\in\N^n_0\), the symbol \(c(x)\pi(X)^\alpha\) belongs to \(\dot{S}^{[\alpha]}_c\).
\end{example}
In the Euclidean case, homogeneous symbols are cut off in a neighbourhood of the zero section to obtain actual elements of the symbol classes. This corresponds to the following procedure for graded Lie groups.
\begin{proposition}[\cite{fischer2017defect}*{4.6}]\label{cutoff}
	Let \(\psi\in\Cont^\infty([0,\infty))\) be a cutoff function with \(0\leq\psi\leq 1\) and \(\psi|_{[0,1]}\equiv 0\) and \(\psi|_{[2,\infty)}\equiv 1\). For all \(m\in\R\) there is a linear map \(c_m\colon\dot{S}^m\to S^m\) given by \(a(x,\pi) \mapsto a(x,\pi)\psi(\pi(\Rock))\).
\end{proposition}
This allows to define a homogeneous expansion of symbols.
\begin{proposition}[\cite{fischer2017defect}*{4.14}]
	Let \(\{a_j\}_{j\in\N_0}\) be a sequence of homogeneous symbols \(a_j\in \dot{S}^{m_j}\) with \(m_j\) strictly decreasing to \(-\infty\) as \(j\to\infty\). Then there is a symbol \(a\in S^{m_0}\), unique modulo \(S^{-\infty}\), such that
	\[a(x,\pi)-\sum_{j=0}^M{a_j(x,\pi)\psi(\pi(\Rock))}\in S^{m_{M+1}} \quad\text{for all }M\in\N.\]
	Moreover, if \(a\in S^m\) for \(m<m_0\), it follows that \(a_0=0\).
\end{proposition}
In this case, we also write \(a \sim \sum a_j\). There is a well-defined principal symbol for these operators.	
\begin{definition}[\cite{fischer2017defect}*{4.17}]
	Let \(a\in S^m\) be a symbol that admits a homogeneous expansion \(a\sim\sum_{j=0}^\infty a_j\) with \(a_j\in \dot{S}^{m-j}\) as above. The \emph{principal part} of \(\Op(a)\) is defined as \(\princg_m(\Op(a))\defeq a_0\).  
\end{definition}
\begin{definition}[\cite{fischer2017defect}*{4.20}]
	A symbol \(a\in S^m_{\cp}\) is a \emph{classical pseudo\-differential symbol of order \(m\)} if it admits a homogeneous expansion \(a\sim\sum_{j=0}^\infty a_j\) with \(a_j\in \dot{S}^{m-j}\). 					
	Denote by \(S^m_{\class}\) the classical pseudo-differential symbols of order \(m\) and by \(\Psi^m_{\class}=\Op(S^m_{\class})\) the corresponding operators.		
\end{definition}
\begin{proposition}[\cite{fischer2017defect}*{4.19}]
	Let \(A\in\Op(S^m)\) and \(B\in\Op(S^l)\) be operators whose symbols have a homogeneous expansion as above. Then the symbols of \(AB\) and \(A^*\) admit homogeneous expansions as well. Moreover, the following holds
	\begin{align*}\princg_{m+l}(AB)&=\princg_{m}(A)\cdot \princg_{l}(B).\\
	\princg_{m}(A^*)&=\princg_{m}(A)^*.
	\end{align*}
	On the right hand side, the pointwise operations are used. 
\end{proposition}
In particular, \(\princg_0\) is a \(^*\)-homomorphism.
\begin{lemma}\label{psiext}
	For \(m\in\Z\), there are short exact sequences
	\begin{equation}\label{psdoext}
	\begin{tikzcd}
	\Psi^{m-1}_{\class}\arrow[hook,r] & \Psi^m_{\class}\arrow[r,twoheadrightarrow,"\princg_m"] &\dot{S}_c^m.
	\end{tikzcd}
	\end{equation}
	For \(m=0\), it is a short exact sequence of \(^*\)-algebras. 
\end{lemma}
\begin{proof}Except for surjectivity of the principal symbol map, exactness is clear. Let \(a_0\in\dot{S}^m_c\) and \(a(x,\pi)\defeq a_0(x,\pi)\psi(\pi(\Rock))\in S^m\) as in \cref{cutoff}. Then \(\princg_m(\Op(a))=a_0\) holds. We need to adjust \(a\) in a way such that its integral kernel \(K_a\) is compactly supported.  Let \(r>0\) be such that the support of \(a\) in the \(x\)-direction in contained in \(B(0,r)\) with respect to the homogeneous quasi-norm. Choose \(c\in\Cont^\infty_c(G)\) which is constant \(1\) on \(B(0,2C r)\). Here, \(C\) is the constant from the homogeneous triangle inequality in \cref{triangle}. Denote by \(M_c\colon L^2(G)\to L^2(G)\) the multiplication operator \(\varphi\mapsto c\cdot\varphi\). It belongs to \(\Psi^0_{\class}\) with symbol given by \(c(x)\Id_{\Hils_\pi}\) for \(x\in G\), \(\pi\in\widehat{G}\). Let \(Q\defeq\Op(a)M_c\). Its integral kernel \(K_a(x,y)c(y)\) is compactly supported. Moreover, \(Q\) belongs to \(\Psi^m_{\class}\) and 
	\[\princg_m(Q)=\princg_m(\Op(a))-\princg_m(\Op(a)-Q).\]
	We show that last term is zero. Let \(\kappa\) denote the convolution kernel of \(\Op(a)\). The convolution kernel of \(\Op(a)-Q\) is \(\kappa_x(y)(1-c(xy^{-1}))\). This is only non-zero if \(\norm{xy^{-1}}\geq 2C r\) and \(\norm{x}<r\). The homogeneous triangle inequality implies that \(\norm{y}\geq r\). But the \(\kappa_x\) are smooth outside zero by~\cite{fischer2016quantization}*{5.4.1}. Therefore, \(\Op(a)-Q\) is a smoothing operator and its principal symbol vanishes.
\end{proof}
\subsection{Comparison of the symbol algebras}
In this section, we identify \(\Fix^\Rp(J_G,\cl{\Rel_G})\) with the \(\Cst\)-algebra of invariant \(0\)-homogeneous symbols defined in \cite{fischer2017defect}*{5.1, 5.5}.

\begin{definition}
	The \(^*\)-algebra of \emph{invariant \(0\)\nb-homogeneous symbols} \(\tilde{S}^0\) consists of all \(a\in L^\infty(\widehat{G}, \Bound(\Hils_\pi))\) with \begin{enumerate}
		\item \(a(x,\lambda\cdot\pi)=a(x,\pi)\) for almost all \(\pi\) and \(\lambda>0\),
		\item for all \(\alpha\in\N^n_0\), the field of operators \(\Delta^\alpha a(\pi)\) is defined on smooth vectors and satisfies
		\[\sup_{\pi\in\widehat{G}}\norm{\Delta^\alpha a(\pi)\pi(\Rock)^\frac{[\alpha]}{q}}_{\Bound(\Hils_\pi)}<\infty,\]			
	\end{enumerate}	
	The \emph{\(\Cst\)-algebra of invariant \(0\)-homogeneous symbols} \(\Cst(\tilde{S}^0)\) is the closure of \(\tilde{S}^{0}\) with respect to \(\norm{a}=\sup_{\pi\in\widehat{G}/\Rp}{\norm{a(\pi)}}\). 
\end{definition}
\begin{lemma}
	The \(^*\)-algebra \(\kernel^0(G)\) of kernels of type \(0\) is isomorphic to \(\tilde{S}^0\) under Fourier transform. Moreover, \(\Fix^\Rp(J_G,\cl{\Rel_G})\) is isomorphic to \(\Cst(\tilde{S}^0)\).
\end{lemma}
\begin{proof}By \cite{fischer2017defect}*{5.3} the kernels of type \(0\) correspond exactly to the invariant, \(0\)-homogeneous symbols under Fourier transform. By uniqueness of the \(\Cst\)-completion, \(\Cst(\tilde{S}^0)\) is isomorphic to the \(\Cst\)-algebra generated by kernels of type \(0\). This is \(\Fix^\Rp(J_G,\cl{\Rel_G})\) by \cref{res:kernelzerograded}.
\end{proof}
By \cref{res:symbolfix_field}, \(\Fix^\Rp(J_0,\cl{\Rel_0})\) is isomorphic \(\Cont_0(G,\Fix^\Rp(J_G,\cl{\Rel_G}))\).
Identifying the space of~\(0\)-homogeneous symbols with compact support in the \(x\)-direction \(\dot{S}^0_c\) with \(\Cont^\infty_c(G,\tilde{S}^0)\) as in the proof of \cite{fischer2017defect}*{5.9}, yields the following result.
\begin{corollary}\label{res:symbols=fix0}
	The \(^*\)-algebra of \(0\)-homogeneous symbols \(\dot{S}^0_c\) is isomorphic under inverse Fourier transform to a dense \(^*\)-subalgebra of \(\Fix^\Rp(J_0,\cl{\Rel_0})\).
\end{corollary}	
 
\subsection{Comparison of the operators}
To compare the sequence of generalized fixed point algebras to the order zero pseudo-differential extension from \eqref{psdoext}, we show that operators in \(\widetilde{p}_1(\Fix^\Rp(J,\cl{\Rel}))\) can be written as in \eqref{eq:quantization} in terms of a symbol.

\begin{lemma}\label{res:T_i(h)_special}
	Let \(f\in\Rel_\grpd^*\) with \(f_0\in\Schwartz_0(TG)\). Define operators \(T_i(f)\) by
	\[T_i(f)\varphi(x)=\int_0^\infty \chi_i(\lambda)\lambda^{-Q}\int f(x,\lambda,\alpha_{\lambda^{-1}}(x^{-1}y))\varphi(y)\diff y \tfrac{\diff\lambda}{\lambda}\]for \(\varphi\in L^2(G)\) and \(x\in G\). Then \((T_i(f))\) converges strictly to an element \(T(f)\in\widetilde{p}_1(\Fix^\Rp(J,\cl{\Rel})) \). 
\end{lemma}
\begin{proof}Using the Dixmier--Malliavin Theorem as in the proof of \cref{res:delta} one can factorize \(f_0=\sum_{i=1}^n g_i^**h_i\) with \(g_i,h_i\in\Rel_0\). Choose \(G_i,H_i\in\Rel^*_\grpd\) that restrict to \(g_i,h_i\) at \(t=0\). Then we can write
	\[T_i(f)=T_i\left(f-\sum_{i=1}^n{G_i^**H_i}\right)-\sum_{i=1}^n{T_i(G_i^**H_i)}.\]
	The first part converges strictly in \(\widetilde{p}_1(\Fix^\Rp(J,\cl{\Rel}))\) by \cref{res:compact_fact}, whereas the second converges by \cref{res:fix_on_l2}.
\end{proof}

We consider now a slightly different generalized fixed point algebra. The reason is that one has to understand the convolution kernels \(\kappa_x\) as a family of left-invariant operators in order to take their Fourier transform. Let \(B\defeq\Cst(TG)\otimes\Cont_0([0,\infty))\) 
with \(\Rp\)-action given by \[\beta_\lambda(h)(x,v,t)=\lambda^Qh(x,\lambda \cdot v,\lambda^{-1}t) \quad\text{for } \lambda>0 \text{ and }h\in\Cont_c(TG\times[0,\infty)).\]For \(t\geq 0\) let \(\ev_t\colon B\to \Cst(TG)\) be the homomorphism induced by the restriction.
Define the following~\(\Rp\)-invariant ideal \(J_B\) with dense subset \(\Rel_B\subset J_B\): \begin{align*}J_B&=\bigcap_{x\in G}\ker(\widehat{\pi}_{\triv}\circ q_x\circ\ev_0), \\
\Rel_B&=\left\{h\in\Cont^\infty_c([0,\infty)\times G,\Schwartz(G)) \,\middle|\,\int_G h(x,v,0)\diff v=0 \text{ for all }x\in G\right\}.\end{align*}
Similar arguments as in \cref{res:mainestimate} and \cref{res:Jisfixable} show:
\begin{lemma}
	The \(^*\)-subalgebra \(\Rel_B\subset J_B\) is square-integrable for the action \(\beta\) of \(\Rp\). Furthermore,
	\((J_B,\cl{\Rel_B})\) is a continuously square-integrable \(\Rp\)-\(\Cst\)-algebra. 
\end{lemma}
Hence, \(\Fix^\Rp(J_B,\cl{\Rel_B})\) is defined. The evaluations at \(t=1\) and \(x\in G\) composed with the right regular representation \(\Cst(G)\to\Bound(L^2G)\), yield strictly continuous representations
\[\widetilde{\rho}_x\colon \Fix^\Rp(J_B,\cl{\Rel_B})\to \VN_L(G).\]
\begin{lemma}\label{res:J_B_symbols}
	For \(h\in\Rel_B\) with \(h_0\in\Schwartz_0(TG)\) \[\int_0^\infty\chi_i(\lambda)\beta_\lambda(h)\tfrac{\diff\lambda}{\lambda}\] converges strictly to an element of \(\Fix^\Rp(J_B,\cl{\Rel_B})\). Its image under \(\widetilde{\rho}_x\) is given by
	\[\widetilde{\rho}_x(h)\varphi=\lim_{i}{ \int_0^\infty\chi_i(\lambda)\lambda^{-Q}\rho_x\left(h(x,\alpha_{\lambda^{-1}}(\,\cdot\,),\lambda)\right)\varphi\,\tfrac{\diff\lambda}{\lambda}}\quad\text{for }\varphi\in L^2(G).\]
\end{lemma}
\begin{proof} The first claim is proved as in \cref{res:T_i(h)_special}. This uses that there is an isomorphism
	\[\Psi\colon\ker(\ev_0)\to\Cont_0(\Rp)\otimes\Cst(TG)\]
	which is induced by \(\Psi(h)(t)=t^{-Q}\ev_t(h)\). The action \(\beta\) corresponds to \(\tau\otimes 1\) under the isomorphism. Here, \(\tau\) is induced by the action of \(\Rp\) on itself by scaling.	
	The second claim follows from strict continuity of \(\widetilde{\rho}_x\) and the computation \(\ev_t\circ\beta_\lambda(h)=\lambda^Q\ev_{\lambda^{-1}t}(h)\) for \(\lambda,t>0\).
\end{proof}

\begin{lemma}\label{kernelssymbols}
	For \(f\in\Rel^*_\grpd\) with \(f_0\in\Schwartz_0(TG)\), there is a symbol \(\{a_f(x,\pi)\in\Bound(\Hils_\pi)\mid x\in G, \pi\in\widehat{G}\}\) such that \(T(f)=\Op(a_f)\).
\end{lemma}
\begin{proof}
	We show that one can write \(T(f)\varphi(x)=\varphi*\kappa_x\) for a smooth family of kernels \(\kappa_x\in\Schwartz'(G)\) such that \(\varphi\mapsto\varphi*\kappa_x\) extends to a bounded operator on \(L^2(G)\) for all \(x\in G\). This implies that one can apply Fourier transform to \(\kappa_x\) and one obtains a symbol as above with \(a_f(x,\pi)=\widehat{\kappa_x}(\pi)\).
	
	For each \(i\in I\), one can write \(T_i(f)\varphi(x)=\varphi*\kappa_{i,x}\) with
	\[\kappa_{i,x}(v)=\int_0^\infty \chi_i(\lambda)\lambda^{-Q} f(x,\alpha_{\lambda^{-1}}(v^{-1}),\lambda) \tfrac{\diff\lambda}{\lambda}. \] We claim that all  \((\kappa_{i,x})\) for \(x\in G\) converge to distributions \(\kappa_{x}\in\Schwartz'(G)\) whose convolution operators are bounded. 
	Note that \(f\) can be understood as an element of \(\Rel_B\) with \(f_0\in\Schwartz_0(TG)\). It follows from \cref{res:J_B_symbols}
	that \(\kappa_{i,x}\) converges in \(\Schwartz'(G)\) to the convolution kernel of \(\widetilde{\rho}_x(f)\). In particular, \(a_f(x,\pi)\) is the strict limit of
	\[\int_0^\infty\chi_i(\lambda)\widehat{f}(x,\lambda\cdot \pi,\lambda)\tfrac{\diff\lambda}{\lambda}\]
	as multipliers of \(\Comp(\Hils_\pi)\).
\end{proof}
Let \(a_0\in\dot{S}^0_c\). As discussed in \cref{cutoff} \(a_0\psi(\pi(\Rock))\) is in \(S^0\). Using \cref{res:J_B_symbols} we obtain a different way to attach a symbol to \(a_0\). By \cref{res:symbols=fix0}, there is a \(h_0\in\Schwartz_0(TG)\) such that \(a_0\) is the Fourier transform of
\[\int_0^\infty \sigma_\lambda(h_0)\tfrac{\diff\lambda}{\lambda}\in\kernel^0(TG).\]
Let \(\omega\in\Cont_c^\infty([0,\infty))\) be a function with \(\omega|_{[0,1]}\equiv 1\) and \(\omega|_{[2,\infty)}\equiv 0\). Define \(h\in\Rel_B\)  by \(h(x,v,t)\defeq \omega(t)h_0(x,v)\). By \cref{kernelssymbols} this yields a symbol \(a_h\in\Cont_0(G,L^\infty(\widehat{G}))\). We compare now the symbols \(a_0\psi(\pi(\Rock))\) and \(a_h\). As a preparation, the following lemma is proved.
\begin{lemma}\label{res:prelemma}
	Let \(h_0\in\Schwartz_0(TG)\) and let \(\omega\in\Cont_c^\infty([0,\infty))\) be a function with \(\omega|_{[0,1]}\equiv 1\) and \(\omega|_{[2,\infty)}\equiv 0\). Define \(h\in\Rel_B\)  by \(h(x,v,t)\defeq \omega(t)h_0(x,v)\).
	Let \(a_h(x,\pi)\) be the Fourier transform of \(\widetilde{\rho}_x(h)\) and \(a_0(x,\pi)\) the Fourier transform of \[\int_0^\infty\sigma_{\lambda}(h_0)\tfrac{\diff\lambda}{\lambda}.\] Then for all \(m>0\), there exists a constant \(C_m>0\) with
	\[\norm{(a_0(x,\pi)-a_h(x,\pi))\psi(\pi(\Rock))(1+\pi(\Rock))^\frac{m}{q}}\leq C_m\norm{\widehat{h_0}}_{S^{-m},0,0}\]
	for all \(x\in G\) and almost all \(\pi\in\widehat{G}\).  
\end{lemma}
\begin{proof}The symbols can be written as strict limits
	\begin{align*}
	a_0(x,\pi)&= \lim_s\int_0^\infty \chi_i(\lambda)\widehat{h_0}(x,\lambda\cdot\pi)\tfrac{\diff \lambda}{\lambda},\\
	a_h(x,\pi)&=\lim_s \int_0^\infty \chi_i(\lambda)\omega(\lambda)\widehat{h_0}(x,\lambda\cdot\pi)\tfrac{\diff \lambda}{\lambda},
	\end{align*}
	as multipliers of \(\Comp(\Hils_\pi)\) for almost all \(\pi\in\widehat{G}\). This implies that
	\[b_i(x,\pi)\defeq \int_0^\infty \chi_i(\lambda^{-1})(1-\omega(\lambda))\widehat{h_0}(x,\lambda\cdot\pi)\tfrac{\diff \lambda}{\lambda}\,\psi(\pi(\Rock))(1+\pi(\Rock))^\frac{m}{q}\] converges strongly to  \(b(x,\pi)\defeq(a_0(x,\pi)-a_h(x,\pi))\psi(\pi(\Rock))(1+\pi(\Rock))^\frac{m}{q}\) on \(\Hils^\infty_\pi\). We show that \(b_i(x,\pi)\) is a Cauchy sequence. As \(\Hils^\infty_\pi\) is dense, this will imply that \(b_i(x,\pi)\) converges to \(b(x,\pi)\) in norm. For \(j>i\) we estimate
	\begin{align*}
	&\norm{b_j(x,\pi)-b_i(x,\pi)}\\
	=\,& \Bigl\lVert{ \int_0^\infty (\chi_j(\lambda)-\chi_i(\lambda))(1-\omega(\lambda))\widehat{h_0}(x,\lambda.\pi)\psi(\pi(\Rock))(1+\pi(\Rock))^\frac{m}{q}\tfrac{\diff \lambda}{\lambda}}\Bigr\rVert\\
	\leq\,  & \sup_{t\geq 1}{\left(\frac{1+t}{t}\right)^{\frac{m}{q}}}\int_0^\infty (\chi_j(\lambda)-\chi_i(\lambda))(1-\omega(\lambda)) \sup_{(x,\pi)}{\Bigl\lVert \widehat{h_0}(x,\lambda\cdot \pi)\pi(\Rock)^{\frac{m}{q}}\Bigr\rVert}\tfrac{\diff \lambda}{\lambda}\\
	\lesssim\, & \int_0^\infty (1-\chi_i(\lambda))\,\frac{1-\omega(\lambda)}{\lambda^m} \sup_{(x,\pi)}{\Bigl\lVert  \widehat{h_0}(x,\lambda\cdot\pi)(\lambda\cdot\pi)(\Rock)^{\frac{m}{q}}\Bigr\rVert}\tfrac{\diff \lambda}{\lambda}\\
	\lesssim \,& \sup_{t\geq 0}{\left(\frac{t}{1+t}\right)^{\frac{m}{q}}}\int_0^\infty (1-\chi_i(\lambda))\frac{1-\omega(\lambda)}{\lambda^m} \sup_{(x,\pi)}{\Bigl\lVert  \widehat{h_0}(x,\lambda\cdot\pi)(1+(\lambda\cdot \pi)(\Rock))^{\frac{m}{q}}\Bigr\rVert}\tfrac{\diff \lambda}{\lambda}\\
	\lesssim\, & \bigl\lVert{\widehat{h_0}}\bigr\rVert_{S^{-m},0,0}\int_0^\infty (1-\chi_i(\lambda)) \frac{1-\omega(\lambda)}{\lambda^{m+1}} \diff \lambda.
	\end{align*}
	The integral converges to \(0\) as the dominated convergence theorem can be applied due to the assumptions on \(\omega\). 	Note that \(\widehat{h_0}(x,\pi)\) is a smoothing symbol by \cite{fischer2016quantization}*{5.2.21}, so that 
	\(\norm{\widehat{h_0}}_{S^{-m},0,0}\) is finite for  all \(m>0\). Using the same estimates there is a constant \(C_m>0\) such that \(\norm{b_i(x,\pi)}\leq C_m\norm{\widehat{h_0}}_{S^{-m},0,0}\) for all \(i\in I\) and \((x,\pi)\in G\times\widehat{G}\). As \(b(x,\pi)\) is the norm limit of this net, the claim follows.
\end{proof}
\begin{remark}\label{rem:pi}
	The same result holds when \(\psi(\pi(\Rock))(1+\pi(\Rock))^\frac{m}{q}\) is replaced by~\(\pi(\Rock)^\frac{m}{q}\). 
\end{remark}
\begin{lemma}\label{res:diffsmoothing}
	Let \(h_0\in\Schwartz_0(TG)\), \(h\in\Rel_B\), \(a_0\) and \(a_h\) be as in \cref{res:prelemma}. Then
	\(a_0\psi(\pi(\Rock))-a_h\) is a smoothing symbol. 
\end{lemma}
\begin{proof}
	Write
	\( a_0\psi(\pi(\Rock))-a=(a_0-a)\psi(\pi(\Rock))-a(1-\psi)(\pi(\Rock))\).	
	We claim that both summands are smoothing symbols. Recall that a symbol \(b\) is smoothing if for all \(m>0\) and \(\alpha,\beta\in\N^n_0\) 
	\[ \sup_{(x,\pi)}{\Bigl\lVert X_x^\beta\Delta^\alpha\{b(x,\pi)\}(1+\pi(\Rock))^{\frac{[\alpha]+m}{q}}\Bigr\rVert}<\infty.\]
	For \((a_0-a)\psi(\pi(\Rock))\) consider first the case \(\alpha=0\). Then the result follows by applying \cref{res:prelemma} to \(X_x^\beta h_0\in\Schwartz_0(TG)\).
	For arbitrary \(\alpha\in\N^n_0\), the Leibniz rule for difference operators \cite{fischer2017defect}*{(3.1)} yields
	\[ \Delta^\alpha\{(a_0-a)(x,\pi)\psi(\pi(\Rock))\}=\sum_{[\alpha_1]+[\alpha_2]=[\alpha]}{\left[\Delta^{\alpha_1}(a_0-a)(x,\pi)\right]\left[\Delta^{\alpha_2}\psi(\pi(\Rock))\right]}.\]
	For \(\alpha_2\neq 0\), it is shown in \cite{fischer2017defect}*{4.8} that 
	\[\sup_{\pi}{\Bigl\lVert\pi(\Rock)^{\frac{-m-[\alpha_1]}{q}}\Delta^{\alpha_2}\psi(\pi(\Rock))(1+\pi(\Rock))^{\frac{m+[\alpha]}{q}}\Bigr\rVert}<\infty.\]
	Applying \cref{rem:pi} and \cref{res:prelemma} to \(X_x^\beta v^{\alpha_1}h_0\) yields
	\[\sup_{(x,\pi)}{\Bigl\lVert X^\beta\Delta^{\alpha_1}(a_0-a)(x,\pi)\pi(\Rock)^{\frac{m+[\alpha_1]}{q}}\Bigr\rVert}<\infty.\]
	For \(\alpha_2=0\), \cref{res:prelemma} is applied to \(X^\beta_x v^\alpha h_0\in\Schwartz_0(TG)\).
	
	Consider now the symbol \(a(1-\psi)(\pi(\Rock))\). As \((1-\psi)\) is supported in \([0,2]\) and \((1+t)^{\frac{[\alpha]+m}{q}}\) is bounded on this subset, it suffices to show for all \(\alpha,\beta\in\N^n_{0}\) that 
	\[\sup_{(x,\pi)}{\Bigl\lVert X^\beta_x\Delta^\alpha a(x,\pi)\Bigr\rVert}<\infty.\]
	For \(\alpha=0\), this follows from \cref{kernelssymbols} applied to \(X^\beta_x h\). For \(\alpha\neq 0\), the net
	\[\int_0^\infty \chi_i(\lambda)\omega(\lambda )\lambda^{[\alpha]}\widehat{(X^\beta_xv^\alpha h_0})(x,\lambda\cdot\pi )\tfrac{\diff\lambda}{\lambda}\]
	is Cauchy in \(\Cont_0(G,L^\infty(\widehat{G}))\). This follows from \(\omega(\lambda)\lambda^{[\alpha]}\leq\omega(\lambda) 2^{[\alpha]}\) and \cref{kernelssymbols} applied to \(X^\beta_xv^\alpha h_0\). Then \(X^\beta_x\Delta^\alpha\) is the limit of this net as the respective convolution kernels converge in~\(\Schwartz'(G)\). 
\end{proof}
\begin{theorem}\label{pseudoext}
	Let \(G\) be a graded Lie group. 
	The order zero pseudo-differential extension from \cref{psiext} embeds into the generalized fixed point algebra extension for \(G\) such that the following diagram commutes
	\begin{equation}\label{ses:pseudo_in_fix}
	\begin{tikzcd}
	\Psi^{-1}_{\class} \arrow[d,hook]\arrow[r,hook]& \Psi^0_{\class}\arrow[r,twoheadrightarrow,"\princg_0"]\arrow[d,hook] & \dot{S}_c^0\arrow[d,hook]\\
	\Comp(L^2G) \arrow[r,hook]& \Fix^\Rp(J,\cl{\Rel}) \arrow[r,twoheadrightarrow,"\widetilde{p}_0"] & \Fix^\Rp(J_0,\cl{\Rel_0}).
	\end{tikzcd}
	\end{equation}		
\end{theorem}
\begin{proof}
	Every operator in \(\Psi^{-1}_{\class}\) extends to a compact operator on \(L^2(G)\) by \cref{compactop}.		
	Let \(A\) be a classical pseudo-differential operator of order zero with principal symbol  \(a_0\in\dot{S}_c^0\). Let \(Q\in\Psi^0_{\class}\) be the element constructed in the proof of \cref{psiext} with \(\princg_0(Q)=a_0\). Recall that \(Q=\Op(a_0(x,\pi)\psi(\pi(\Rock)))M_c\) for a certain \(c\in\Cont^\infty_c(G)\).
	In the following we show that there is an element \(T\in\widetilde{p}_1(\Fix^\Rp(J,\cl{\Rel}))\) with \(\widetilde{p}_0(T)=a_0\) and \(Q-T\in\Comp(L^2G)\). Once is this established, writing
	\[A=A-Q+Q-T+T\] shows that \(A\) lies in \(\widetilde{p}_1(\Fix^\Rp(J,\cl{\Rel}))\) as \(A-Q\) has order \(-1\) since its principal symbol vanishes. In particular, \(A-Q\) is compact by \cref{compactop}. The above decomposition also shows that
	\[\tilde{p}_0(A)=\tilde{p}_0(T)=a_0=\princg_0(A)\]
	so that the diagram in \eqref{ses:pseudo_in_fix} commutes.		
	To construct \(T\), let \(h_0\in\Schwartz_0(TG)\) be such that \(a_0\) is the Fourier transform of 
	\[\int_0^\infty\sigma_\lambda(h_0)\tfrac{\diff\lambda}{\lambda}.\]  Let \(h\in\Rel_B\) and \(a_h\) be as in \cref{res:prelemma} and \(g(x,v,t)\defeq h(x,v,t)c(x\alpha_t(v))\). Then \(T(g)=T(f)M_c\) holds. \cref{kernelssymbols} implies that
	\[Q-T(f)=\Op(a_0\psi(\pi(\Rock))-a_h)M_c.\]
	This is a compact operator as its convolution kernel is smooth by \cref{res:diffsmoothing} and compactly supported. 
\end{proof}
Denote by \(\Cst(\Psi^0_{\class})\) the closure of the \(^*\)-algebra \(\Psi^0_{\class}\) in \(\Bound(L^2G)\).
\begin{corollary}
	The \(\Cst\)-algebra \(\Cst(\Psi^0_{\class})\) generated by classical order zero pseudo-differential operators on a graded Lie group \(G\) is isomorphic to \(\Fix^\Rp(J,\cl{\Rel})\). There is an extension of \(\Cst\)\nb-algebras 
	\begin{equation*}
	\begin{tikzcd}
	\Comp(L^2 G)\arrow[hook,r] & \Cst(\Psi^0_{\class})\arrow[r,twoheadrightarrow,"\widetilde{p}_0"] &\Cst(\dot{S}_c^0),
	\end{tikzcd}
	\end{equation*}
	such that \(\widetilde{p}_0\) extends the principal symbol map \(\princg_0\colon \Psi^0_{\class}\to\dot{S}^0_c\).
\end{corollary}
\begin{proof}
	We show that \(\Cst(\Psi^0_{\class})=\widetilde{p}_1(\Fix^\Rp(J,\cl{\Rel}))\). The \(\Cst\)-algebra of  pseudo-differential operators of order \(0\) is contained in \(\widetilde{p}_1(\Fix^\Rp(J_0,\cl{\Rel}))\) by \cref{pseudoext}. 
	
	For the converse, note first that \(\Comp(L^2G)\subset \Cst(\Psi^0_{\class})\). This holds as \(\Psi^0_{\class}\) contains the kernels in \(\Cont_c^\infty(G\times G)\) and these generate the compact operators on \(L^2(M)\). 		
	Now, let \(f,g\in\Rel\). Let \(a\in \dot{S}_c^0\) be the inverse of \(\KET{f_0}\BRA{g_0}\in \Cont_0(G,\kernel^0(G))\) under Fourier transform. Since the principal symbol map is surjective, there is a \(P\in\Psi^0_{\class}\) with \(\princg_0(P)=a\).  Then the operator \[\widetilde{p}_1(\KET{f}\BRA{g})= \widetilde{p}_1(\KET{f}\BRA{g})-\Op(P)+\Op(P)\] is contained in \(\Cst(\Psi^0_{\class})\). This is because \(\Op(P)\) is and \(\widetilde{p}_1(\KET{f}\BRA{g})-\Op(P)\in\Comp(L^2G)\) as the diagram in \eqref{ses:pseudo_in_fix} commutes. The \(\Cst\)-algebra \(\Fix^\Rp(J,\cl{\Rel})\) is generated by \(\KET{f}\BRA{g}\) with \(f,g\in\Rel\). Thus, the result follows. 
\end{proof}

\section{Morita equivalence and \(\K\)-theory}\label{fullness}
In this section, we will show that \((J,\cl{\Rel})\) and \((J_0,\cl{\Rel_0})\) are saturated for the zoom action of \(\Rp\). Therefore, for each homogeneous Lie group \(G\) the \(\Cst\)\nb-algebras of order zero pseudo-differential operators \(\Fix^\Rp(J,\cl{\Rel})\) and principal symbols \(\Fix(J_0,\cl{\Rel_0})\) are Morita--Rieffel equivalent to \(\Cred(\Rp,J)\) and \(\Cred(\Rp,J_0)\), respectively. For the Euclidean scalings on \(G=\R^n\) this is a result of \cite{debordskandalis2014}. 
\subsection{Stratification and saturatedness}
First, consider a homogeneous Lie group \(G\). Recall the sequence of open, dilation invariant subsets of \(\widehat{G}\!\setminus\!\{\pi_{\triv}\}\) found in \eqref{sets}:
\begin{align*}
\emptyset = V_0 \subset V_1 \subset V_2 \subset \ldots \subset V_m = \widehat{G}\!\setminus\!\{\pi_{\triv}\},\end{align*}
where \(\Lambda_i=V_i\setminus V_{i-1}\) are Hausdorff for all \(i=1,\ldots,m\). Moreover, the induced \(\Rp\)\nb-action on each \(\Lambda_i\) is free and proper by \cref{res:freeproper}.
There is a corresponding increasing sequence of closed, dilation invariant ideals in \(\Cst(G)\) \begin{align}\label{ideals}
0=J_0\idealin J_1\idealin J_2\idealin\ldots\idealin J_m=J_G
\end{align} which is given by
\[ J_i=\{f\in \Cst(G)\mid \widehat{\pi}(f)=0 \text{ for }\pi\in\widehat{G}\setminus V_i\}.\]
In this section, it will be shown that the subquotients \(J_i/J_{i-1}\) of the filtration in \eqref{ideals} define continuous fields of \(\Cst\)\nb-algebras over \(\Lambda_i\), respectively. This will allow us to prove, using \cref{res:ses_saturated}, that \(\Fix(J_G,\cl{\Rel_G})\) is Morita--Rieffel equivalent to the crossed product \(\Cred(\Rp,J_G)\).
\begin{remark}
	In \cite{beltictua2016fourier} Pedersen's fine stratification \cite{pedersen1989geometric} is used to obtain a similar sequence of increasing ideals, where the respective subquotients are even isomorphic to trivial fields \(\Cont_0(\tilde{\Lambda}_i,\Comp(\Hils_i))\) for some finite- or infinite-dimensional Hilbert spaces~\(\Hils_i\). For our purposes the coarse stratification suffices. 	
\end{remark}
\begin{proposition}\label{res:quotients_fields}
	Each subquotient \(J_i/J_{i-1}\) for \(i=1,\ldots,m\) is isomorphic to a continuous field of \(\Cst\)\nb-algebras over \(\Lambda_i\) with a unique dense, relatively continuous and complete subset \(\Rel_i\) for the induced \(\Rp\)\nb-action. Furthermore, \((J_i/J_{i-1},\Rel_i)\) is saturated.
\end{proposition}
\begin{proof}
	The subquotient \(J_i/J_{i-1}\) has Hausdorff spectrum as
	\[ \widehat{J_i/J_{i-1}} \cong \widehat{J_i}\setminus\widehat{J_{i-1}} \cong V_i\setminus V_{i-1} = \Lambda_i .\]
	Therefore,  \(J_i/J_{i-1}\) is isomorphic to a continuous field of \(\Cst\)-algebras over \(\Lambda_i\), see \cite{nilsen1996}*{3.3}. The isomorphism takes \([f]\in J_i/J_{i-1}\) to the section \(\widehat{f}\) defined by\
	\[\widehat{f}(\pi)=\widehat{\pi}(f)=\int_G f(x)\pi(x)\diff x\in \Bound(\Hils_\pi) \qquad \text{for }\pi\in\Lambda_i.\]
	The dilation action on \(J_i/J_{i-1}\) satisfies \(\widehat{\sigma_\lambda(f)}(\pi)=\widehat{f}(\lambda^{-1}\cdot\pi)\) for all \(\lambda>0\). Denote by \(\alpha_\lambda(\widehat{f})\) the section given by \(\alpha_\lambda(\widehat{f})(\pi)=\widehat{f}(\lambda^{-1}\cdot\pi)\). 
	Let \(\theta_i\colon\Cont_0(\Lambda_i)\hookrightarrow  Z \Mult(J_i/J_{i-1})\) denote the non-degenerate homomorphism which is given by pointwise multiplication when \(J_i/J_{i-1}\) is viewed as a continuous field. It satisfies the compatibility condition
	\begin{align*}
	\alpha_\lambda(\theta_i(\phi)\widehat{f})=\theta_i(\tau_\lambda\phi)\alpha_\lambda(\widehat{f}) \qquad\text{for }\phi\in\Cont_0(\Lambda_i)\text{ and }[f]\in J_i/J_{i-1},
	\end{align*}
	where \(\tau\) denotes the \(\Rp\)\nb-action on \(\Cont_0(\Lambda_i)\) given by \(\tau_\lambda(\phi)(\pi)=\phi(\lambda^{-1}\cdot\pi)\). Therefore, \(J_i/J_{i-1}\) is an \(\Rp\)-\(\Cont_0(\Lambda_i)\)-algebra. The dilation action on \(\Lambda_i\) is free and proper by \cref{res:freeproper}. By \cref{res:saturatedspectrum} \(J_i/J_{i-1}\) is saturated with respect to the subset
	\[ \Rel_i\defeq\overline{\theta_i(\Cont_c(\Lambda_i))(J_i/J_{i-1})}.\]
	It is the unique dense, complete, relatively continuous subset by \cref{res:spectrallyproper} as \(J_i/J_{i-1}\) is spectrally proper.
\end{proof}
Using \cref{res:ses_saturated} and an inductive argument for the sequence in~\eqref{ideals}, we obtain as a consequence:
\begin{corollary}\label{res:pointwisesaturated}
	For a homogeneous Lie group \(G\) the \(\Rp\)-\(\Cst\)-algebra  \((J_G,\cl{\Rel_G})\) is saturated for the dilation action of \(\Rp\). 		
	The generalized fixed point algebra \(\Fix^\Rp(J_G,\cl{\Rel_G})\) is Morita--Rieffel equivalent to \(\Cred(\Rp, J_G)\).
\end{corollary}
\subsection{Computation of the spectrum of the symbol algebra}
Recall that it was shown in \cref{res:kernelzerograded} that for a homogeneous Lie group \(G\), \(\Fix^\Rp(J_G,\cl{\Rel_G})\) is the \(\Cst\)\nb-algebra generated by kernels of type \(0\). 
As an application of the above results, we give a different proof of the description of its spectrum obtained in \cite{fischer2017defect}*{5.5}. 

\begin{proposition}
	Let \(G\) be a homogeneous Lie group. Then \(\Fix^\Rp(J_G,\cl{\Rel_G})\) is of type I. Furthermore, there is a homeomorphism 
	\[\left(\widehat{G}\!\setminus\!\{\pi_{\triv}\}\right)/\Rp\to \widehat{\Fix^\Rp(J_G,\cl{\Rel_G})}\] induced by
	\(\pi\mapsto \left(\widehat{\pi}\right)^{\sim}\) for \(\pi\in\widehat{G}\!\setminus\!\{\pi_{\triv}\}\). 	
\end{proposition}
\begin{proof}
	The ideals in~\eqref{ideals} yield short exact sequences of generalized fixed point algebras for \(i=1,\ldots,m\) by \cref{res:sesfixed}:
	\begin{equation*}
	\begin{tikzcd}
	\Fix^\Rp(J_{i-1},\cl{\Rel_G}\cap J_{i-1})\arrow[r,hook] & \Fix^\Rp(J_{i},\cl{\Rel_G}\cap J_{i})\arrow[r,twoheadrightarrow,"\widetilde{q}"] & \Fix^\Rp(J_i/J_{i-1},\Rel_i).
	\end{tikzcd}
	\end{equation*}
	Each quotient \(J_i/J_{i-1}\) is an \(\Rp\)\nb-\(\Cont_0(\Lambda_i)\)\nb-algebra with a free and proper \(\Rp\)-action on \(\Lambda_i\) by \cref{res:quotients_fields}. Therefore, the spectrum of \(\Fix^\Rp(J_i/J_{i-1},\Rel_i)\) is \(\Lambda_i/\Rp\) by \cref{res:saturatedspectrum}. In particular, the spectrum is Hausdorff and, thus, \(T_0\). As \(\Fix^\Rp(J_i/J_{i-1},\Rel_i)\) is separable, this implies that it is of type I by \cite{dixmier}*{3.1.6, 9.1}. If an ideal \(I\) of a \(\Cst\)-algebra \(A\) and the quotient \(A/I\) are of type I, it follows that \(A\) is of type I. Thus, one can use an inductive argument for the sequences above to show that \(\Fix^\Rp(J_G,\cl{\Rel_G})\) is of type~I.
	
	We proceed by showing the second claim. The spectrum of \(J_G\) is given by
	\[\widehat{J_G}=\widehat{G}\!\setminus\!\{\pi_{\triv}\}=\Lambda_1\cup\ldots\cup\Lambda_m.\]
	By \cref{res:freeproper}, \(\Rp\) acts freely on \(\Lambda_i\) for all \(i=1,\ldots,m\). Hence, the action of \(\Rp\) on the spectrum of \(J_G\) is free. As \(J_G\) and \(\Fix^\Rp(J_G,\cl{\Rel_G})\) are of type I and separable, their spectra can be identified with their primitive ideal spaces. We show that there is a homeomorphism 
	\[\psi\colon \Prim(J_G)/\Rp\to \Prim(\Fix^\Rp(J_G,\cl{\Rel_G}))\]
	with \(\psi([\ker(\pi)])=\ker(\widetilde{\pi})\) for \(\pi\in\widehat{J_G}\).
	By \cite{kwa-meyer-stone}*{6.3, 6.4} there is a continuous, open  and surjective quasi-orbit map
	\[\rho\colon \Prim(\Cred(\Rp,J_G))\to\Prim(J_G)/\Rp.\]
	As the action of the amenable group \(\Rp\) on \(\Prim(J_G)\) is free, the quasi-orbit map is also injective by \cite{gootman-rosenberg-EH}*{3.3}. Hence, it is a homeomorphism. In particular, \(J_G\) separates ideals in \(\Cred(\Rp,J_G)\). We describe now the inverse of \(\rho\). 
	First, there is a homeomorphism \(\phi\colon \Prim(J_G)/\Rp \to \Prime(\mathbb{I}^\Rp(J_G))\) by \cite{kwa-meyer-stone}*{6.3}, which is induced by mapping a primitive ideal \(P\) of \(J_G\) to the largest \(\Rp\)-invariant ideal contained in \(P\):
	\begin{align*}
	P&\mapsto \bigcap_{\lambda\in\Rp}\lambda\cdot P.
	\end{align*}
	It follows from \cite{kwa-meyer-stone}*{2.11, 6.1, 6.3} that \(\rho^{-1}=i\circ \phi\) with
	\begin{align*}i\colon \Prime(\mathbb{I}^\Rp(J_G))&\to \Prim(\Cred(\Rp,J_G)),\\
	Q&\mapsto \Cred(\Rp,Q).\end{align*}
	By \cref{res:pointwisesaturated}, \(\Cred(\Rp,J_G)\) is Morita--Rieffel equivalent to \(\Fix^\Rp(J_G,\cl{\Rel_G})\). Therefore, the Rieffel correspondence gives a homeomorphism
	\[r\colon \Prim(\Cred(\Rp,J_G))\to\Prim(\Fix^\Rp(J_G,\cl{\Rel_G})).\]
	Together, we obtain a homeomorphism \[\psi\defeq r\circ \rho^{-1}\colon \Prim(J_G)/\Rp\to\Prim(\Fix^\Rp(J_G,\cl{\Rel_G})).\]	
	It is left to show that \(\psi([\ker(\pi)])=\ker(\widetilde{\pi})\) for \(\pi\in\widehat{J_G}\). Let \(Q=\phi([\ker(\pi)])\). Using that the action on \(J_G\) is saturated, \cref{ses:reduced} implies
	\begin{align*} \rho^{-1}([\ker(\pi)])&=\Cred(\Rp,Q)=\Cred(\Rp,Q)\cap\Cred(\Rp,J_G)=J^\Rp(Q,\cl{\Rel_G}\cap Q).\end{align*}
	This ideal is mapped to \(\Fix^\Rp(Q,\cl{\Rel_G}\cap Q)\) under the Rieffel correspondence. 
	
	We show that \(\Fix^\Rp(Q,\cl{\Rel_G}\cap Q)=\ker(\widetilde{\pi})\). Let \(a,b\in\cl{\Rel_G}\cap Q\). Then \(\pi(a)=\pi(b)=0\) and, consequently, \(\widetilde{\pi}(\KET{a}\BRA{b})=\KET{\pi(a)}\BRA{\pi(b)}=0\). It follows that \(\Fix^\Rp(Q,\cl{\Rel_G}\cap Q)\subseteq\ker(\widetilde{\pi})\). 	
	
	Now let \(T\in\ker(\widetilde{\pi})\). As elements of the generalized fixed point algebra are invariant under the \(\Rp\)-action, \(\left(\lambda\cdot \pi\right)^{\sim}(T)=\widetilde{\pi}(T)=0\) holds for all \(\lambda>0\). We use a similar argument as in the proof of \cref{res:sesfixed} and show that \(T^*T\in\Fix^\Rp(Q,\cl{\Rel_G}\cap Q)\). For each \(a\in\cl{\Rel_G}\), we obtain 
	\[(\lambda\cdot \pi)(T^*a)=\left(\lambda\cdot \pi\right)^{\sim}(T^*)(\lambda\cdot \pi)(a)=0 \quad\text{for all }\lambda>0.\]
	It follows that \(T^*a\in \cl{\Rel_G}\cap Q\). Now the same argument as in \cref{res:sesfixed} shows that \(T\in\Fix^\Rp(Q,\cl{\Rel_G}\cap Q)\).
\end{proof}	
As \(\Fix^\Rp(J_0,\cl{\Rel_0})\) is the trivial continuous field over \(G\) with fibre \(\Fix^\Rp(J_G,\cl{\Rel_G})\) we obtain the following result (compare \cite{fischer2017defect}*{5.11}):
\begin{corollary}
	For each homogeneous Lie group \(G\) there is a homeomorphism
	\begin{align*}
	\widehat{\Fix^\Rp(J_0,\cl{\Rel_0})}&\cong G\times (\widehat{G}\!\setminus\!\{\pi_{\triv}\})/\Rp.
	\end{align*}
\end{corollary}
\subsection{Morita equivalence}
We deduce saturatedness for the respective ideals in the \(\Cst\)-algebras of \(TG\) and the tangent groupoid \(\grpd\). 
\begin{proposition}\label{res:morita}
	Let \(G\) be a homogeneous Lie group.
	The \(\Cst\)\nb-algebra of order~\(0\) principal symbols \(\Fix^\Rp(J_0,\cl{\Rel_0})\) is Morita--Rieffel equivalent to \(\Cred(\Rp,J_0)\). The \(\Cst\)\nb-algebra  of order \(0\) pseudo-differential operators \(\Fix^\Rp(J,\cl{\Rel})\) is Morita--Rieffel equivalent to \(\Cred(\Rp,J)\). 
\end{proposition}

\begin{proof}
	As \((J_G,\cl{\Rel_G})\) is saturated by \cref{res:pointwisesaturated},  \((J_0,\cl{\Rel_0})\) is saturated by \cref{res:field_saturated}. Therefore, the generalized fixed point algebra construction gives the Morita--Rieffel equivalence between \(\Fix^\Rp(J_0,\cl{\Rel_0})\) and \(\Cred(\Rp,J_0)\).
	
	The second claim follows from \cref{res:ses_saturated} applied to the sequence in \eqref{restrictedses} if saturatedness for the ideal \(\Cont_0(\Rp)\otimes\Comp(L^2G)\) is shown. By \cref{actiononideal} the \(\Rp\)-action is given by \(\tau\otimes 1\), where \(\tau\) is induced by the action of \(\Rp\) on itself by multiplication. Then \(\Rel\cap (\Cont_0(\Rp)\otimes\Comp(L^2G))\) is the unique dense, relatively continuous and complete subspace and \(\tau\) is free and proper. Therefore, the action is saturated by \cref{res:saturatedspectrum}. The Morita--Rieffel equivalence follows again from the generalized fixed point algebra construction. 
\end{proof}

\subsection{\(\K\)-theory of the \(\Cst\)-algebra of \(0\)-homogeneous symbols}
The Morita--Rieffel equivalence between the \(\Cst\)\nb-algebra of \(0\)-homogeneous symbols and the crossed product \(C_r^*(\R,J_0)\) allows us to compute its \(\K\)\nb-theory. We recover the same result as in the Euclidean setting. 
\begin{theorem}
	Let \(G\) be a graded nilpotent Lie group with \(n=\dim\lieg\). Then \(\Fix^\Rp(J_G,\cl{\Rel_G})\) is \(\KK\)-equivalent to \(\Cont(S^{n-1})\). The \(\Cst\)-algebra of principal symbols \(\Fix^\Rp(J_0,\cl{\Rel_0})\) is \(\KK\)-equivalent to \(\Cont_0(S^*\R^n)\).
\end{theorem}
\begin{proof}
The Morita--Rieffel equivalences between \(\Fix^\Rp(J_G,\cl{\Rel_G})\) and \(\Cred(\R,J_G)\) obtained in \cref{res:pointwisesaturated} implies that they are \(\KK\)\nb-equivalent. By the Connes--Thom isomorphism, \(\Cred(\R,J_G)\) is in turn \(\KK\)\nb-equivalent to \(\Cont_0(\R)\otimes J_G\). 

Let \(\lieg\) be the Lie algebra of \(G\) and for each \(t\in[0,1]\) define \([X,Y]_t\defeq t[X,Y]\) for \(X,Y\in\lieg\). Note that here the usual scalar multiplication by \(t\in[0,1]\) is used and not the dilation action. One checks that \([\,\cdot\,,\,\cdot\,]_t\) defines a Lie bracket for all \(t\in[0,1]\). Denote by \(\lieg_t\) the corresponding Lie algebra and by \(G_t\) its Lie group. All Lie algebras \(\lieg_t\) for \(t>0\) are isomorphic to \(\lieg\) via \(X\mapsto tX\). 

Consider the groupoid \(\mathcal{D}_G=\R^n\times [0,1]\rightrightarrows[0,1]\), where source and range are given by the projection to the last coordinate and the multiplication in \(s^{-1}(t)=r^{-1}(t)=\R^n\), identified with \(G_t\) under the exponential map, is given by group multiplication in \(G_t\). This is a continuous field of groups over \([0,1]\) that deforms the graded nilpotent Lie group \(G\) into the Abelian group \(\R^n\). Using Fourier transform at \(t=0\) one obtains the short exact sequence
\begin{equation*}
\begin{tikzcd}
 \Cont_0((0,1])\otimes \Cst(G)\arrow[r,hook] & \Cst(\mathcal{D}_G) \arrow[r,twoheadrightarrow,"\ev_0"]& \Cont_0(\R^n).\end{tikzcd}
\end{equation*}
Consider the associated \(\KK\)-element \([\ev_0]^{-1}\otimes[\ev_1]\in\KK(\Cont_0(\R^n),\Cst(G))\), as described in \cite{debordlescure}. First, we shall prove as in \cite{nistor} that it is a KK-equivalence for any simply connected, nilpotent Lie group \(G\) by induction on the dimension of \(G\). If \(G\) is one-dimensional, it must be Abelian, so that \(G_t\) is the constant field and \([\ev_1]^{-1}\otimes[\ev_0]\) is the inverse class. If \(G\) has dimension greater than one, it can be written as a semidirect product \(G=G'\rtimes \R\). Furthermore \(\mathcal{D}_G\cong\mathcal{D}_{G'}\rtimes\R\) and \(\Cont_0(\R^n)\cong \Cont_0(\R^{n-1})\rtimes\R\) such that the following diagram commutes
\begin{equation*}
\begin{tikzcd}
\Cont_0(\R^n)\arrow[d,"\cong"]& \Cst(\mathcal{D}_G) \arrow[l,swap,"\ev^G_0"]\arrow[r,"\ev^G_1"]\arrow[d,"\cong"] & \Cst(G)\arrow[d,"\cong"]\\
\Cont_0(\R^{n-1})\rtimes\R & \Cst(\mathcal{D}_{G'})\rtimes\R \arrow[l,swap,"(\ev^{G'}_0)_*"]\arrow[r,"(\ev^{G'}_1)_*"] & \Cst(G')\rtimes\R.
\end{tikzcd}
\end{equation*}
The naturality of the Connes--Thom isomorphism shows that the bottom row defines a \(\KK\)\nb-equivalence by induction hypothesis, which yields that \(\Cst(G)\) and \(\Cont_0(\R^n)\)
are \(\KK\)-equivalent. We show that it restricts to a \(\KK\)\nb-equivalence between \(J_G\) and \(\Cont_0(\R^n\!\setminus\!\{0\})\). Consider the ideal \(I_G\subset\Cst(\mathcal{D}_G)\) that consists of all sections \((a_t)\in\Cst(\mathcal{D}_G)\) such that all \(a_t\in\Cst(G_t)\) lie in the kernel of the trivial representation of \(G_t\). In the commuting diagram
\begin{equation*}
\begin{tikzcd}
J_G \arrow[r,hook] & \Cst(G) \arrow[r,twoheadrightarrow] & \C\\
I_G\arrow[r,hook]\arrow{u}[swap]{\ev_1}\arrow{d}{\ev_0} & \Cst(\mathcal{D}_G)\arrow[r,twoheadrightarrow]\arrow{u}[swap]{\ev_1}\arrow{d}{\ev_0} & \Cont([0,1]) \arrow{u}[swap]{\ev_1}\arrow{d}{\ev_0}\\
\Cont_0(\R^n\!\setminus\!\{0\})\arrow[r,hook] & \Cont_0(\R^n) \arrow[r,twoheadrightarrow]& \C
\end{tikzcd}
\end{equation*}
the associated \(\KK\)-classes in the middle and on the right are \(\KK\)-equivalences. The long exact sequences in \(\KK\)-theory show that the deformation element on the left is also a \(\KK\)\nb-equivalence. In conclusion, \(\Fix^\Rp(J_G,\cl{\Rel_G})\) is \(\KK\)-equivalent to \(\Cont_0(\R)\otimes\Cont_0(\R^n\!\setminus\!\{0\})\). In the Euclidean case, the generalized fixed point algebra \(\Cont(S^{n-1})\) is likewise \(\KK\)-equivalent to \(\Cont_0(\R)\otimes\Cont_0(\R^n\!\setminus\!\{0\})\).

By \cref{res:morita}, \(\Fix^\Rp(J_0,\cl{\Rel_0})\) is Morita-equivalent to \(\Cred(\R,J_0)\), which is again by the Connes--Thom isomorphism \(\KK\)-equivalent to \(\Cont_0(\R)\otimes J_0\). As \(J_0\cong \Cont_0(\R^n)\otimes J_G\), it follows that \(\Fix^\Rp(J_0,\cl{\Rel_0})\) is \(\KK\)-equivalent to \(\Cont_0(S^*\R^n)\).
\end{proof}

\end{document}